\documentclass[11pt]{article}
 
\usepackage[margin=0.5in]{geometry} 
\usepackage{amsmath,amsthm,amssymb}
\usepackage{mathrsfs}
\usepackage[T1]{fontenc}
\usepackage[english]{babel}
\usepackage{graphicx}
\usepackage{color} 

\newcommand{\ts}{\textsuperscript}

\usepackage{hyperref}
\usepackage{afterpage}

\usepackage{color}

\usepackage{atbegshi}
\AtBeginDocument{\AtBeginShipoutNext{\AtBeginShipoutDiscard}}
\usepackage{pdfpages} 
\usepackage{longtable} 
\usepackage{amsmath,amsthm,amsfonts,amssymb,relsize,bm} 
\usepackage{xparse,xcolor} 
\usepackage{graphicx,float} 
\usepackage{etoolbox}
\usepackage[shortlabels]{enumitem} 
\usepackage{tikz-cd} 
\usepackage{cite} 
\usepackage[normalem]{ulem}
\usepackage{comment}
\usepackage{tocloft}

\numberwithin{equation}{section}

\theoremstyle{plain}
\newtheorem{theorem}{Theorem}[section]
\newtheorem*{theorem*}{Theorem}
\newtheorem{lemma}{Lemma}[section]
\newtheorem{corollary}{Corollary}[section]

\theoremstyle{definition}

\newtheorem*{definition*}{Definition}

\newtheorem{remark}{Remark}[section]

\begin{document}
 
 \title{Generalizations (in the spirit of Koshliakov) of some formulas from Ramanujan's Lost Notebook}

\author{{Pedro Ribeiro,\quad    Semyon  Yakubovich} } 
\thanks{


Department of Mathematics, Faculty of Sciences of University of Porto, Rua do Campo Alegre,  687; 4169-007 Porto (Portugal). 

\,\,\,\,\,E-mail of Corresponding author: pedromanelribeiro1812@gmail.com}
\date{}
 
\maketitle

\begin{center}
\textit{Dedicated to the memory of our Friend and Professor Jos\'e Carlos Petronilho}
\end{center}

\abstract{In his lost notebook, Ramanujan recorded beautiful identities. These
include earlier versions of Koshliakov's formula for the divisor function
and the transformation formula for the logarithm of Dedekind's $\eta-$function.
In this paper we establish some generalizations of these formulas
of Ramanujan in a setting that only recently reemerged in the literature and which concerns a beautiful theory due to Koshliakov.}

\tableofcontents 

\section{Introduction} \label{sec:1}

Let $\zeta(s)$ denote the Riemann zeta function. On entry 8 of Chapter
15 of his second notebook \cite{berndt_notebooks_II, berndt_evans}, Ramanujan stated the beautiful formula
for $\zeta(1/2)$, 
\\

\textbf{Entry 8:} If $x>0$, then the following identity holds
\begin{equation}
\sum_{n=1}^{\infty}\frac{1}{e^{n^{2}x}-1}=\frac{1}{4}+\frac{\pi^{2}}{6x}+\frac{1}{2}\sqrt{\frac{\pi}{x}}\,\zeta\left(\frac{1}{2}\right)+\frac{1}{2}\,\sqrt{\frac{\pi}{x}}\,\sum_{n=1}^{\infty}\frac{1}{\sqrt{n}}\left(\frac{\cos(2\pi\sqrt{\pi n/x})-\sin(2\pi\sqrt{\pi n/x})-e^{-2\pi\sqrt{\pi n/x}}}{\cosh(2\pi\sqrt{\pi n/x})-\cos(2\pi\sqrt{\pi n/x})}\right),\,\,\,\,x>0.\label{Ramanujan formula def entry 8}
\end{equation}

On Entry 8.3.1., page 332, of his lost notebook, Ramanujan restates
(\ref{Ramanujan formula def entry 8}) in two different ways, which
might be an indication that this representation for $\zeta(1/2)$ was dear to him. 

The first proof of (\ref{Ramanujan formula def entry 8}) was given
by Berndt and Evans \cite{berndt_evans}, who employed the Poisson
summation formula and the transformation formula for Jacobi's $\theta-$function.
After this, several generalizations and analogues of (\ref{Ramanujan formula def entry 8})
have been obtained, including character analogues and evaluations
of the Riemann zeta function at certain rational points. The reader
may find in [\cite{Ramanujan_Notebook_Lost}, pp. 191-193] an exhaustive account
of the activity around (\ref{Ramanujan formula def entry 8}).

Of course, (\ref{Ramanujan formula def entry 8}) coexists alongside
other beautiful formulas in Ramanujan's lost notebook. For instance, on
page 253 \cite{Ramanujan_lost, Guinand_Ramanujan, Ramanujan_Notebook_Lost}, Ramanujan
states the following formula, quoted from [\cite{Ramanujan_Notebook_Lost},
p. 94], 
\\

\textbf{Entry 3.3.1. (p. 253)} Let $\sigma_{k}(n)=\sum_{d|n}d^{k}$, and let $\zeta(s)$ denote the
Riemann zeta function. If $\alpha$ and $\beta$ are positive numbers
such that $\alpha\beta=\pi^{2}$ and if $s$ is any complex number,
then
\begin{align}
\sqrt{\alpha}\sum_{n=1}^{\infty}\sigma_{-s}(n)\,n^{s/2}K_{\frac{s}{2}}(2n\alpha) & -\sqrt{\beta}\sum_{n=1}^{\infty}\sigma_{-s}(n)\,n^{s/2}K_{\frac{s}{2}}(2n\beta)=\nonumber \\
\frac{1}{4}\Gamma\left(-\frac{s}{2}\right)\zeta(-s)\left\{ \beta^{(1+s)/2}-\alpha^{(1+s)/2}\right\}  & +\frac{1}{4}\Gamma\left(\frac{s}{2}\right)\zeta(s)\left\{ \beta^{(1-s)/2}-\alpha^{(1-s)/2}\right\} .\label{Guinand given at intro}
\end{align}

\bigskip{}

The appearance of the modified Bessel function and the divisor functions
on the infinite series on the left side of (\ref{Guinand given at intro})
is somewhat remindful of the Fourier expansion of the non-holomorphic
Eisenstein series or of the Epstein zeta function attached to a binary
quadratic form. In fact, the first proof of (\ref{Guinand given at intro}),
due to Guinand \cite{guinand_rapidly_convergent}, used a formula of Watson \cite{watson_reciprocal}, which states
that
\begin{equation}
\sum_{n=1}^{\infty}\frac{1}{\left(n^{2}+x^{2}\right)^{s}}=\frac{\sqrt{\pi}\,x^{1-2s}}{2\Gamma(s)}\Gamma\left(s-\frac{1}{2}\right)-\frac{x^{-2s}}{2}+\frac{2\,\pi^{s}x^{\frac{1}{2}-s}}{\Gamma(s)}\,\sum_{n=1}^{\infty}n^{s-\frac{1}{2}}\,K_{s-\frac{1}{2}}(2\pi nx).\label{Watson Formula intro}
\end{equation}
for every $x>0$ and $\text{Re}(s)>\frac{1}{2}$. Due to the fact
that Guinand's proof seems to be the first appearing in the literature,
(\ref{Guinand given at intro}) is usually called ``Ramanujan-Guinand
formula''.

As remarked by Berndt, Lee and Sohn {[}\cite{Guinand_Ramanujan}, p.
23{]}, Guinand's proof of (\ref{Guinand given at intro}) via (\ref{Watson Formula intro})
is ``completely independent of any considerations of nonanalytic
Eisenstein series''.

\bigskip{}

However, proving (\ref{Watson Formula intro}) is in fact the first step in
almost all the proofs of the aforementioned Fourier expansion of Epstein
zeta functions. This expansion was firstly considered by A. Selberg
and S. Chowla in 1949, who started with the Epstein zeta function \cite{Selberg_Chowla I}
\begin{equation}
\zeta_{Q}(s)=\sum_{m,n\neq0}\frac{1}{\left(am^{2}+bmn+cn^{2}\right)^{s}},\,\,\,\,\,\text{Re}(s)>1,\label{Epstein def}
\end{equation}
(where $m,\,n\neq0$ here means that only the term $m=n=0$ is omitted
from the sum) and announced the following formula, valid in the entire
complex plane,
\begin{align}
a^{s}\Gamma(s)\,\zeta_{Q}(s) & =2\Gamma(s)\zeta(2s)+2k^{1-2s}\pi^{1/2}\Gamma\left(s-\frac{1}{2}\right)\zeta(2s-1)\nonumber \\
 & \,\,\,\,\,\,+\,8k^{1/2-s}\pi^{s}\sum_{n=1}^{\infty}n^{s-1/2}\sigma_{1-2s}(n)\,\cos\left(n\pi b/a\right)K_{s-1/2}\left(2\pi k\,n\right).\label{Selberg Chowla Formula}
\end{align}
Here, $d:=b^{2}-4ac$ is the discriminant of the quadratic form, $k^{2}:=|d|/4a^{2}$
and $\sigma_{\nu}(n)=\sum_{d|n}d^{\nu}$. Also, $Q(x,y)=ax^{2}+bx\,y+cy^{2}$ denotes
a real and positive definite quadratic form.

It is now clear that the appearance of the divisor function in (\ref{Selberg Chowla Formula})
is in some form connected to the identity of Ramanujan (\ref{Guinand given at intro})
and if one analyzes the underlying structure of the proofs of (\ref{Selberg Chowla Formula}), they tend to rely in one way or another on Watson's formula (\ref{Watson Formula intro}).

\bigskip{}

By using (\ref{Selberg Chowla Formula}), Selberg and Chowla \cite{selberg_chowla}
established several new results and reproved others by a new method. For example, using their formula, 
they were able to derive a new proof of the functional equation for $\zeta_{Q}(s)$. They also obtained from (\ref{Selberg Chowla Formula}) a very simple proof of Kronecker's limit formula, 
\begin{equation}
\zeta_{Q}\left(s\right)=\frac{2\pi}{\sqrt{|d|}}\,\frac{1}{s-1}+\frac{2\pi}{\sqrt{|d|}}\,\left(2\gamma-\log\left(\frac{|d|}{a}\right)-4\log\,(|\eta(\tau)|)\right)+O(s-1),\label{Kronecker limit formula intro}
\end{equation}
where $\gamma$ is the Euler-Mascheroni constant, $\tau:=\frac{b+i\sqrt{|d|}}{2a}\in\mathbb{H}$
and $\eta(\tau)$ denotes the Dedekind $\eta-$function,
\begin{equation}
\eta(\tau)=e^{\frac{\pi i\tau}{12}}\,\prod_{m=1}^{\infty}\left(1-e^{2\pi im\tau}\right),\,\,\,\,\,\text{Im}(\tau)>0.\label{Dedekind classic}
\end{equation}

In particular, when $Q(m,n)=m^{2}+cn^{2}$, $c>0$, the particular
case of (\ref{Kronecker limit formula intro}) holds
\begin{equation}
\zeta_{Q}\left(s\right):=\zeta_{2}(s,c)=\frac{\pi}{\sqrt{c}}\,\frac{1}{s-1}+\frac{\pi}{\sqrt{c}}\,\left(2\gamma-\log\left(4c\right)-4\log\left(\left|\eta\left(i\sqrt{c}\right)\right|\right)\right)+O(s-1).\label{Kronecker limit formula diagonal form}
\end{equation}
\\

The function defined by (\ref{Dedekind classic}) also played an important
role in Ramanujan's lost notebook, right after the statement (\ref{Guinand given at intro}).
He completes page 253 with two corollaries, the first of them being
equivalent to a transformation formula for (\ref{Dedekind classic}).
\\

\textbf{Entry 3.3.2. (p. 253)} Let $\alpha$ and $\beta$ be two positive numbers such that $\alpha\beta=\pi^{2}$.
Then
\begin{equation}
\sum_{n=1}^{\infty}\sigma_{-1}(n)\,e^{-2n\alpha}-\sum_{n=1}^{\infty}\sigma_{-1}(n)\,e^{-2n\beta}=\frac{\beta-\alpha}{12}+\frac{1}{4}\log\left(\frac{\alpha}{\beta}\right).\label{transformation formula logarithm Dedekind}
\end{equation}

\bigskip{}

Finally, Ramanujan uses (\ref{Guinand given at intro}) to derive
the most famous formula attributed to Koshliakov \cite{koshliakov_Voronoi}. Preceding  Koshliakov's contributions by more than a decade, Ramanujan states: 
\\

\textbf{Entry 3.3.3 (p. 253):} Let $\alpha$ and $\beta$ denote positive numbers such that $\alpha\beta=\pi^{2}$.
Then the following identity takes place
\begin{equation}
\sqrt{\alpha}\left(\frac{\gamma}{4}-\frac{\log(4\beta)}{4}+\sum_{n=1}^{\infty}d(n)\,K_{0}(2n\alpha)\right)=\sqrt{\beta}\left(\frac{\gamma}{4}-\frac{\log(4\alpha)}{4}+\sum_{n=1}^{\infty}d(n)\,K_{0}(2n\beta)\right).\label{Koshliakov given intro}
\end{equation}

\bigskip{}

It is the purpose of the present paper to extend Ramanujan's entries (\ref{Ramanujan formula def entry 8}),
(\ref{Guinand given at intro}), (\ref{transformation formula logarithm Dedekind})
and (\ref{Koshliakov given intro}). In pursuing (\ref{Guinand given at intro})
and (\ref{transformation formula logarithm Dedekind}) we were led
to the study of generalizations of Watson's formula (\ref{Watson Formula intro}), as well as extensions of Kronecker's limit formula (\ref{Kronecker limit formula intro}) (see sections 4 and 5 below). 
\\

We now explain the nature of the generalizations here proposed. These
have to do with a manuscript of N. S. Koshliakov which has recently found a renewed interest. 
In a very long paper \cite{KOSHLIAKOV}, Koshliakov introduced the zeta
function {[}\cite{KOSHLIAKOV}, p. 6{]}
\begin{equation}
\zeta_{p}(s)=\sum_{n=1}^{\infty}\frac{p^{2}+\lambda_{n}^{2}}{p\left(p+\frac{1}{\pi}\right)+\lambda_{n}^{2}}\,\frac{1}{\lambda_{n}^{s}},\,\,\,\,\,\text{Re}(s)>1,\label{Koshliakov zeta function definition}
\end{equation}
where $\left(\lambda_{n}\right)_{n\in\mathbb{N}}$ is the sequence
defined by the positive roots of the transcendental equation
\begin{equation}
\tan(\pi y)=-\frac{y}{p},\,\,\,\,\,p>0.\label{Transcendental equation}
\end{equation}

From the fact that $n-\frac{1}{2}<\lambda_{n}<n$, the series (\ref{Koshliakov zeta function definition})
converges absolutely in the half-plane $\text{Re}(s)>1$. In particular,
$\zeta_{p}(s)$ generalizes the Riemann zeta function. From the structure
of the equation (\ref{Transcendental equation}), one may see that
\[
\lim_{p\rightarrow\infty}\zeta_{p}(s)=\zeta(s),\,\,\,\,\,\lim_{p\rightarrow0^{+}}\zeta_{p}(s)=(2^{s}-1)\,\zeta(s).
\]

Koshliakov connected the study of the zeta function with a second
generalized zeta function, namely, $\eta_{p}(s)$, being defined by
\begin{equation}
\eta_{p}(s):=\sum_{k=1}^{\infty}\frac{(s,2\pi pk)_{k}}{k^{s}},\,\,\,\,\,\text{Re}(s)>1,\label{secondary zeta function Koshliakov sense}
\end{equation}
where 
\begin{equation}
\left(s,\nu k\right)_{k}:=\frac{1}{\Gamma(s)}\,\intop_{0}^{\infty}x^{s-1}e^{-x}\left(\frac{k\nu-x}{k\nu+x}\right)^{k}\,dx.\label{Mellin definition intro}
\end{equation}

On page 20 of his paper \cite{KOSHLIAKOV}, it is possible to look at
an argument concerning the uniform and absolute convergence of the
series (\ref{secondary zeta function Koshliakov sense}) for $\text{Re}(s)>1$.
This property was also rederived on page 26, formula (57), where it
was proved that
\begin{equation}
\lim_{k\rightarrow\infty}(s,\lambda k)_{k}=\frac{1}{\left(1+\frac{2}{\lambda}\right)^{s}},\label{asymptotic kosh sequence end}
\end{equation}
for every $\lambda>0$.

\bigskip{}

Besides studying the analytic continuations of (\ref{Koshliakov zeta function definition})
and (\ref{secondary zeta function Koshliakov sense}), Koshliakov
proved that $\zeta_{p}(s)$ and $\eta_{p}(s)$ are connected through
the functional equation
\begin{equation}
\zeta_{p}(1-s)=\frac{2\,\cos\left(\frac{\pi s}{2}\right)\Gamma(s)}{(2\pi)^{s}}\eta_{p}(s).\label{functional equation Kosh}
\end{equation}

After developing an analytic structure for the functions $\zeta_{p}(s)$
and $\eta_{p}(s)$, Koshliakov enters the realm of summation formulas,
proving new analogues of Poisson's summation formula, introducing
new generalizations of the digamma function and Bernoulli numbers (building on a previous work \cite{Koshliakov_Bernoulli})
and even extending an integral formula of Ramanujan (see [\cite{KOSHLIAKOV},
p. 150, eq. (19)], where a generalization of the main formula from Ramanujan's paper \cite{ramanujan_integrals} is derived). 
\\

It is not an overstatement to declare that  this manuscript of Koshliakov is truly a masterpiece and unfortunately was
kept in obscurity since its publication in 1949. Thanks to the recent outstanding
work of A. Dixit and R. Gupta \cite{DG}, Koshliakov's main results
were examined and extended in several directions.

\bigskip{}

Besides making Koshliakov's ideas accessible to a modern mathematical
community, Dixit and Gupta also took the setting of Koshliakov in order to obtain new particular formulas. For example (among many others, which include other relations that can be found in the Lost Notebook) they extended a famous
formula of Ramanujan for $\zeta(2n+1)$ [\cite{DG}, p. 13, Theorem 4.1.] \footnote{see also  \cite{Berndt_Straub} for a very nice survey about this formula of Ramanujan.} and from this extension they could achieve new identities by taking a suitable choice of the parameter $p$. 

\bigskip{}

The remarkable work of Dixit and Gupta was then followed by a paper \cite{Ramanujan_meets} (written jointly with Berndt and Zaharescu)
where new analogues of the Abel-Plana formula were considered.

\bigskip{}

It is in the line of reasoning of the contributions given in \cite{DG} and \cite{Ramanujan_meets} that our main results are here presented. Our paper is organized as follows: in section 2 we establish a generalization (in Koshliakov's setting) of Ramanujan's formula for $\zeta(1/2)$ (\ref{Ramanujan formula def entry 8}). In sections 3 and 4 we generalize Watson's formula (\ref{Watson Formula intro}), as well as the Kronecker limit formula for 'diagonal quadratic forms', (\ref{Kronecker limit formula diagonal form}), in two directions.

\bigskip{}

At last, we devote the fifth section of our paper to the generalization of Entries 3.3.1, 3.3.2 and 3.3.3 stated above. 

\bigskip{}

Before going to the main sections of this paper, we remark some additional
facts about the zeta function (\ref{secondary zeta function Koshliakov sense}),
which will be useful in the sequel. Since $\eta_{p}(s)$ is not a
Dirichlet series in the classical sense, we need some particular properties
that can be extracted from the Mellin representation (\ref{Mellin definition intro}).
In fact, Koshliakov {[}\cite{KOSHLIAKOV}, p. 25, eq. (48){]} could
write (\ref{Mellin definition intro}) as a combination of fractional
integrals of the form
\begin{equation}
\Gamma(s)\,(s,\lambda)_{k}=\Gamma(s)(-1)^{k}+\Gamma(s)\,k^{s}e^{\lambda}\,\sum_{\ell=1}^{k}\left(\begin{array}{c}
k\\
\ell
\end{array}\right)\left(\frac{2\lambda}{k}\right)^{\ell}(-1)^{k-\ell}\,\intop_{k}^{\infty}t^{-s}e^{-\frac{\lambda}{k}t}\,\frac{(t-k)^{\ell-1}}{(\ell-1)!}\,dt,\label{representation Koshliakov given in poage 24}
\end{equation}
which, when used in the definition (\ref{secondary zeta function Koshliakov sense}),
gives the expression for $\eta_{p}(s)$
\begin{equation}
\eta_{p}(s):=\sum_{k=1}^{\infty}\frac{(s,2\pi pk)_{k}}{k^{s}}=\left(2^{1-s}-1\right)\zeta(s)+\sum_{k=1}^{\infty}e^{2\pi pk}\sum_{\ell=1}^{k}\left(\begin{array}{c}
k\\
\ell
\end{array}\right)\left(4\pi p\right)^{\ell}(-1)^{k-\ell}\,\intop_{k}^{\infty}t^{-s}e^{-2\pi pt}\,\frac{(t-k)^{\ell-1}}{(\ell-1)!}\,dt,\label{secondary zeta function representation as integrals}
\end{equation}
being valid when $\text{Re}(s)>1$. Representation (\ref{secondary zeta function representation as integrals})
will prove to be very useful in section 3 of this paper, where two generalizations of Watson's formula (\ref{Watson Formula intro}) will be given. 

\bigskip{}

Throughout this paper we shall also employ the notation introduced
by Koshliakov: if $p,p^{\prime}\in\mathbb{R}^{+}$, we define the
functions $\sigma(t)$ and $\sigma^{\prime}(t)$ as being {[}\cite{KOSHLIAKOV},
p. 6{]}
\begin{equation}
\sigma(t):=\frac{p+t}{p-t},\,\,\,\sigma^{\prime}(t):=\frac{p^{\prime}+t}{p^{\prime}-t}.\label{definition sigma p sigma p'}
\end{equation}

By [\cite{KOSHLIAKOV}, p. 17, eq. (18)], we know that the function $\frac{1}{\sigma(t)\,e^{2\pi t}-1}$ is associated with
the analytic continuation of $\eta_{p}(s)$. The analogue of this
function attached to $\zeta_{p}(s)$ is\footnote{In this paper we will try to preserve Koshliakov's notation, even when it differs from standard notation (see (\ref{Incomplete Gamma function Kosh sense}) below). So we urge the reader to not misinterpret $\sigma_{p}(z)$ as the standard divisor function, $\sigma_{k}(n)$, which plays an important role in formula (\ref{Guinand given at intro}). The only section of the paper where we will need the divisor function will be Section 5, but fortunately in this section we will not use Koshliakov's function $\sigma_{p}(z)$.} 
 [\cite{KOSHLIAKOV}, p. 44, eq. (33)],
\begin{equation}
\sigma_{p}(z)=\sum_{n=1}^{\infty}\frac{p^{2}+\lambda_{n}^{2}}{p\left(p+\frac{1}{\pi}\right)+\lambda_{n}^{2}}\,e^{-\lambda_{n}z},\,\,\,\,\,\,\text{Re}(z)>0.\label{sigma p definition on Koshliakov}
\end{equation}
\\

For additional standard facts about the zeta functions (\ref{Koshliakov zeta function definition})
and (\ref{secondary zeta function Koshliakov sense}) we refer to
Koshliakov's own manuscript \cite{KOSHLIAKOV} and to Dixit and Gupta's
paper \cite{DG}.

\section{Generalization of Entry 8.3.1., p. 332 of Ramanujan's Lost Notebook}

We are now ready to establish a general version of Entry 8.3.1.
\begin{theorem}
Let $p,p^{\prime}\in \mathbb{R}_{+}$ and  $\lambda_{n}$,
$\lambda_{n}^{\prime}$ be the sequences of numbers satisfying
 respectively  the transcendental equations
\begin{equation}
\tan(\pi y)=-\frac{y}{p},\,\,\,\,\tan(\pi y)=-\frac{y}{p^{\prime}}.\label{beginning of statement transcendental}
\end{equation}

Moreover, let $\sigma(t),\,\sigma^{\prime}(t)$ be defined by (\ref{definition sigma p sigma p'}).
Then, for every $x>0$, the following identity takes place  
\begin{align}
\sum_{n=1}^{\infty}\frac{p^{2}+\lambda_{n}^{2}}{p\left(p+\frac{1}{\pi}\right)+\lambda_{n}^{2}}\cdot\frac{1}{\sigma^{\prime}\left(\frac{\lambda_{n}^{2}x}{2\pi}\right)e^{\lambda_{n}^{2}x}-1} & =\frac{1}{4\left(1+\frac{1}{\pi p}\right)}+\frac{\pi^{2}}{6x}\,\frac{1+\frac{3}{\pi p}(1+\frac{1}{\pi p})}{\left(1+\frac{1}{\pi p^{\prime}}\right)\left(1+\frac{1}{\pi p}\right)^{2}}+\frac{1}{2}\sqrt{\frac{\pi}{x}}\zeta_{p^{\prime}}\left(\frac{1}{2}\right)+\nonumber \\
 & +\frac{1}{2}\sqrt{\frac{\pi}{x}}\,\sum_{n=1}^{\infty}\frac{p^{\prime2}+\lambda_{n}^{\prime2}}{p^{\prime}\left(p^{\prime}+\frac{1}{\pi}\right)+\lambda_{n}^{\prime2}}\,\frac{1}{\sqrt{\lambda_{n}^{\prime}}}\,G_{p}\left(\frac{2\pi\lambda_{n}^{\prime}}{x}\right),\label{Analogue final in the statement}
\end{align}
where, for $a>0$, $G_{p}(a)$ is explicitly given by
\begin{equation}
G_{p}(a)=\frac{\left(p^{2}-a\right)\left\{ \cos(\pi\sqrt{2a})-\sin(\pi\sqrt{2a})\right\} -e^{-\pi\sqrt{2a}}\left(p^{2}-\sqrt{2a}\,p+a\right)-\sqrt{2a}\,p\left\{ \cos(\pi\sqrt{2a})+\sin(\pi\sqrt{2a})\right\} }{p^{2}\left(\cosh(\pi\sqrt{2a})-\cos(\pi\sqrt{2a})\right)+\sqrt{2a}\,p\left(\sinh(\pi\sqrt{2a})+\sin(\pi\sqrt{2a})\right)+a\left(\cosh(\pi\sqrt{2a})+\cos(\pi\sqrt{2a})\right)}.\label{definition of Qp (a) intro}
\end{equation}

\end{theorem}

\begin{proof}
By {[}\cite{KOSHLIAKOV}, p. 32, Chapter 1, eq. (74){]}, for $\text{Re}(s)>1$
one has the representation
\begin{equation}
\zeta_{p^{\prime}}(1-s)=2\cos\left(\frac{\pi s}{2}\right)\intop_{0}^{\infty}\frac{x^{s-1}}{\sigma^{\prime}(x)\,e^{2\pi x}-1}\,dx.\label{representation as starting point}
\end{equation}

Therefore, from Mellin's inversion formula and the absolute convergence
of the series (\ref{Koshliakov zeta function definition}) for $\text{Re}(s)>\mu>1$,
\begin{equation}
\sum_{n=1}^{\infty}\frac{p^{2}+\lambda_{n}^{2}}{p\left(p+\frac{1}{\pi}\right)+\lambda_{n}^{2}}\cdot\frac{1}{\sigma^{\prime}\left(\frac{\lambda_{n}^{2}x}{2\pi}\right)e^{\lambda_{n}^{2}x}-1}=\frac{1}{2\pi i}\,\intop_{\mu-i\infty}^{\mu+i\infty}\frac{\zeta_{p^{\prime}}(1-s)\,\zeta_{p}(2s)}{2\cos\left(\frac{\pi s}{2}\right)}\,\left(\frac{x}{2\pi}\right)^{-s}\,ds.\label{Starting point series}
\end{equation}
Now we shift the line of integration in (\ref{Starting point series})
to $\text{Re}(s)=\frac{1}{2}-\mu$. To do this we just need to integrate
along a positively oriented rectangular contour $\mathscr{R}_{\mu}(T)$
with vertices $\mu\pm iT$ and $\frac{1}{2}-\mu\pm iT$: by an application
of Cauchy's residue Theorem, we have that

\begin{equation}
\left\{ \intop_{\mu-iT}^{\mu+iT}+\intop_{\mu+iT}^{\frac{1}{2}-\mu+iT}+\intop_{\frac{1}{2}-\mu+iT}^{\frac{1}{2}-\mu-iT}+\intop_{\frac{1}{2}-\mu-iT}^{\mu-iT}\right\} \,\frac{\zeta_{p^{\prime}}(1-s)\,\zeta_{p}(2s)}{2\cos\left(\frac{\pi s}{2}\right)}\,\left(\frac{x}{2\pi}\right)^{-s}\,ds=2\pi i\,R_{p,p^{\prime}}(x),\label{Application residue theorem}
\end{equation}
where $R_{p,p^{\prime}}(x)$ denotes the sum of the residues of the
integrand inside $\mathscr{R}_{\mu}(T)$. Since $\zeta_{p}(-2k)=0$
for every $k\in\mathbb{N}$ {[}\cite{KOSHLIAKOV}, p. 22, eq. (37){]},
we know that the points $s=2k+1$, $k\in\mathbb{Z}\setminus\{0\}$
are removable singularities of the integrated function. Thus, the integrand above has only three
simple poles, which are located at the points $s=0,\,\frac{1}{2},\,1$.
It is simple to see that the residual function $R_{p,p^{\prime}}(x)$
can be explicitly written as
\begin{equation}
R_{p,p^{\prime}}(x)=\frac{1}{4\left(1+\frac{1}{\pi p}\right)}+\frac{1}{2}\sqrt{\frac{\pi}{x}}\,\zeta_{p^{\prime}}\left(\frac{1}{2}\right)+\frac{\pi^{2}}{6x}\,\frac{1+\frac{3}{\pi p}(1+\frac{1}{\pi p})}{\left(1+\frac{1}{\pi p^{\prime}}\right)\left(1+\frac{1}{\pi p}\right)^{2}},\label{explicit formula residues p and p'}
\end{equation}
because $\zeta_{p}(0)=-\frac{1}{2}\,\frac{1}{1+\frac{1}{\pi p}}$
and $\zeta_{p}(2)=\frac{\pi^{2}}{6}\,\frac{1+\frac{3}{\pi p}(1+\frac{1}{\pi p})}{\left(1+\frac{1}{\pi p}\right)^{2}}$
{[}\cite{KOSHLIAKOV}, p. 22, eq. (34), (39){]}.

We now estimate in an elementary way the integrals along the horizontal
segments $\left[\frac{1}{2}-\mu\pm iT,\,\mu\pm iT\right]$ when $T\rightarrow\infty$:
indeed, 

\begin{equation}
\intop_{\frac{1}{2}-\mu\pm iT}^{\mu\pm iT}\left|\frac{\zeta_{p}(1-s)\,\zeta_{p}(2s)}{2\cos\left(\frac{\pi s}{2}\right)}\,\left(\frac{x}{2\pi}\right)^{-s}\right|\,|ds|\leq\frac{T^{A}e^{-\frac{\pi}{2}T}}{\log\left(\frac{x}{2\pi}\right)}\,\left[\left(\frac{x}{2\pi}\right)^{\mu-\frac{1}{2}}-\left(\frac{x}{2\pi}\right)^{-\mu}\right]\rightarrow0\label{behavior integrals}
\end{equation}
as $T\rightarrow\infty$. Here, $A:=A(\mu)$ is a positive number
depending on $\mu$: we can write it explicitly by using the convex
estimates for $\zeta_{p}(s)$ found by Koshliakov {[}\cite{KOSHLIAKOV},
p. 24, eq. (44){]}, 
\begin{equation}
\zeta_{p}(\sigma+it)=\begin{cases}
O\left(|t|^{1-\sigma}\log|t|\right) & \frac{1}{2}\leq\sigma\leq1,\\
O\left(|t|^{\frac{1}{2}}\log|t|\right) & 0<\sigma\leq\frac{1}{2},\\
O\left(|t|^{\frac{1}{2}-\sigma}\log|t|\right) & \sigma\leq0.
\end{cases}\,\,\,\,\,|t|\rightarrow\infty.\label{Koshliakov Phragmen Analogue}
\end{equation}

\medskip{}

Combining (\ref{Starting point series}) with (\ref{Application residue theorem})
and using the previous bounds (\ref{behavior integrals}) for the
integrals along the horizontal segments, we are able to obtain 
\begin{align}
\sum_{n=1}^{\infty}\frac{p^{2}+\lambda_{n}^{2}}{p\left(p+\frac{1}{\pi}\right)+\lambda_{n}^{2}}\cdot\frac{1}{\sigma^{\prime}\left(\frac{\lambda_{n}^{2}x}{2\pi}\right)e^{\lambda_{n}^{2}x}-1} & =\frac{1}{2\pi i}\,\intop_{\mu-i\infty}^{\mu+i\infty}\frac{\zeta_{p}(1-2s)\,\zeta_{p^{\prime}}(s+\frac{1}{2})}{\sqrt{2}\,\left[\cos\left(\frac{\pi s}{2}\right)+\sin\left(\frac{\pi s}{2}\right)\right]}\,\left(\frac{x}{2\pi}\right)^{s-\frac{1}{2}}ds+\nonumber \\
+\frac{1}{4\left(1+\frac{1}{\pi p}\right)} & +\frac{1}{2}\sqrt{\frac{\pi}{x}}\,\zeta_{p^{\prime}}\left(\frac{1}{2}\right)+\frac{\pi^{2}}{6x}\,\frac{1+\frac{3}{\pi p}(1+\frac{1}{\pi p})}{\left(1+\frac{1}{\pi p^{\prime}}\right)\left(1+\frac{1}{\pi p}\right)^{2}},\label{integral series again}
\end{align}
where in the last step we took the change of variables $s\leftrightarrow\frac{1}{2}-s$
and the explicit expression (\ref{explicit formula residues p and p'}).
Since we have chosen $\mu>1$, we can use the representation of $\zeta_{p^{\prime}}\left(s+\frac{1}{2}\right)$
as a Dirichlet series to write the integral on the right side of (\ref{integral series again})
as 
\begin{equation}
\sum_{n=1}^{\infty}\frac{p^{\prime2}+\lambda_{n}^{\prime2}}{p^{\prime}\left(p^{\prime}+\frac{1}{\pi}\right)+\lambda_{n}^{\prime2}}\,\sqrt{\frac{\pi}{\lambda_{n}^{\prime}x}}\,\frac{1}{2\pi i}\,\intop_{\mu-i\infty}^{\mu+i\infty}\frac{\zeta_{p}(1-2s)}{\cos\left(\frac{\pi s}{2}\right)+\sin\left(\frac{\pi s}{2}\right)}\,\left(\frac{x}{2\pi\lambda_{n}^{\prime}}\right)^{s}ds.\label{evaluation crucial integral after functional}
\end{equation}

Finally, we evaluate the remaining integral by using the functional
equation for $\zeta_{p}\left(s\right)$ (\ref{functional equation Kosh})
and making some elementary reductions, from which we obtain the expression
\begin{equation}
\frac{1}{2\pi i}\,\intop_{\mu-i\infty}^{\mu+i\infty}\frac{\zeta_{p}(1-2s)}{\cos\left(\frac{\pi s}{2}\right)+\sin\left(\frac{\pi s}{2}\right)}\,\left(\frac{x}{2\pi\lambda_{n}^{\prime}}\right)^{s}ds=\sum_{m=1}^{\infty}\,\frac{1}{2\pi i}\,\intop_{2\mu-i\infty}^{2\mu+i\infty}\left\{ \cos\left(\frac{\pi s}{4}\right)-\sin\left(\frac{\pi s}{4}\right)\right\} \,\Gamma(s)\,\left(s,\,2\pi pm\right)_{m}\left(\frac{x}{8\pi^{3}m^{2}\lambda_{n}^{\prime}}\right)^{s/2}\,ds\label{point we have left}
\end{equation}
in its turn being justified by the absolute convergence of $\eta_{p}(s)$ in the half-plane
$\text{Re}(s)>2\mu$. Combining (\ref{representation Koshliakov given in poage 24})
with Stirling's formula for $\Gamma(s)$, we have for $s=\sigma+it$,
$a\leq\sigma\leq b$ and $z\in\mathbb{C}$,
\[
\left|\Gamma(s)\,(s,\lambda)_{k}\,z^{-s}\right|\leq e^{-\sigma\log|z|}\,C(\sigma,k,\lambda)\,|t|^{\sigma-\frac{1}{2}}e^{-\frac{\pi}{2}|t|+\arg(z)\,t},\,\,\,\,\,\,|t|\rightarrow\infty,
\]
where $C(\sigma,k,\lambda)$ is a positive constant only depending
on these fixed parameters.

\bigskip{}

Since $\Gamma(s)\,(s,\lambda)_{k}\,z^{-s}$ defines an analytic function
for $\text{Re}(s)>0$, we have from (\ref{Mellin definition intro}),
Mellin's inversion formula and analytic continuation with respect
to $z$ in the half-plane $\text{Re}(z)>0$, that the following integral representation is valid
\begin{equation}
f_{k}(z,\lambda):=\frac{1}{2\pi i}\intop_{\sigma-i\infty}^{\sigma+i\infty}\Gamma(s)\,(s,\lambda)_{k}\,z^{-s}ds=\left(\frac{\lambda-z}{\lambda+z}\right)^{k}\,e^{-z},\,\,\,\,\sigma>0,\,\,\text{Re}(z)>0.\label{final evaluation of fk(z,lambda)}
\end{equation}

Applying (\ref{evaluation crucial integral after functional})
and using (\ref{point we have left}) and (\ref{final evaluation of fk(z,lambda)}), we immediately obtain 
\begin{align}
\sum_{m=1}^{\infty}\,\frac{1}{2\pi i}\,\intop_{2\mu-i\infty}^{2\mu+i\infty}\left\{ \cos\left(\frac{\pi s}{4}\right)-\sin\left(\frac{\pi s}{4}\right)\right\} \,\Gamma(s)\,\left(s,\,2\pi pm\right)_{m}\left(\frac{x}{8\pi^{3}m^{2}\lambda_{n}^{\prime}}\right)^{s/2} & =\nonumber \\
=\sqrt{2}\,\text{Im}\left\{ e^{i\frac{\pi}{4}}\,\sum_{m=1}^{\infty}\exp\left(-\sqrt{\frac{8\pi^{3}m^{2}\lambda_{n}^{\prime}}{x}}\,e^{i\frac{\pi}{4}}\right)\,\left(\frac{p-\sqrt{\frac{2\pi\lambda_{n}^{\prime}}{x}}\,e^{i\frac{\pi}{4}}}{p+\sqrt{\frac{2\pi\lambda_{n}^{\prime}}{x}}\,e^{i\frac{\pi}{4}}}\right)^{m}\right\} \nonumber \\
=\sqrt{2}\,\text{Im}\left\{ \frac{e^{i\frac{\pi}{4}}}{\sigma\left(\sqrt{\frac{2\pi\lambda_{n}^{\prime}}{x}}\,e^{i\frac{\pi}{4}}\right)\,e^{2\pi\sqrt{\frac{2\pi\lambda_{n}^{\prime}}{x}}\,e^{i\frac{\pi}{4}}}-1}\right\} ,\label{almost final result with Im}
\end{align}
where in the last step we have used the geometric series (recall once more (\ref{definition sigma p sigma p'}))
\begin{equation}
\frac{1}{\sigma(z)\,e^{2\pi z}-1}=\sum_{m=1}^{\infty}\left(\frac{p-z}{p+z}\right)^{m}e^{-2\pi mz},\,\,\,\,\,\text{Re}(z)>0.\label{power series def}
\end{equation}

Returning to (\ref{integral series again}) and to (\ref{evaluation crucial integral after functional})
and collecting (\ref{almost final result with Im}), we are able to
derive the identity
\begin{align}
\sum_{n=1}^{\infty}\frac{p^{2}+\lambda_{n}^{2}}{p\left(p+\frac{1}{\pi}\right)+\lambda_{n}^{2}}\cdot\frac{1}{\sigma^{\prime}\left(\frac{\lambda_{n}^{2}x}{2\pi}\right)e^{\lambda_{n}^{2}x}-1} & =\frac{1}{4\left(1+\frac{1}{\pi p}\right)}+\frac{1}{2}\sqrt{\frac{\pi}{x}}\,\zeta_{p^{\prime}}\left(\frac{1}{2}\right)+\frac{\pi^{2}}{6x}\,\frac{1+\frac{3}{\pi p}(1+\frac{1}{\pi p})}{\left(1+\frac{1}{\pi p^{\prime}}\right)\left(1+\frac{1}{\pi p}\right)^{2}}+\nonumber \\
+\sqrt{\frac{2\pi}{x}\,}\sum_{n=1}^{\infty}\frac{p^{\prime2}+\lambda_{n}^{\prime2}}{p^{\prime}\left(p^{\prime}+\frac{1}{\pi}\right)+\lambda_{n}^{\prime2}}\,\frac{1}{\sqrt{\lambda_{n}^{\prime}}} & \,\text{Im}\left\{ \frac{e^{i\frac{\pi}{4}}}{\sigma\left(\sqrt{\frac{2\pi\lambda_{n}^{\prime}}{x}}\,e^{i\frac{\pi}{4}}\right)\,e^{2\pi\sqrt{\frac{2\pi\lambda_{n}^{\prime}}{x}}\,e^{i\frac{\pi}{4}}}-1}\right\} .\label{at most compact in the end}
\end{align}

\bigskip{}

Now our proof is almost finished, with the only step remaining being to
simplify the imaginary part of the term in the braces. This comes
after straightforward manipulations:  it is very simple to show that
the denominator inside the braces above can be simplified to ($a:=\sqrt{2\pi\lambda_{n}^{\prime}/x}$)
\begin{equation}
\frac{e^{\pi\sqrt{2a}}}{p^{2}-\sqrt{2a}\,p+a}\left\{ \cos\left(\pi\sqrt{2a}\right)\left(p^{2}-a\right)-\sqrt{2a}\,p\,\sin\left(\pi\sqrt{2a}\right)+i\left[\sqrt{2a}\,p\,\cos\left(\pi\sqrt{2a}\right)+\left(p^{2}-a\right)\,\sin\left(\pi\sqrt{2a}\right)\right]\right\} -1,\label{literal expression elementary-1}
\end{equation}
and from this it is very simple to deduce the expression for $\left|\sigma\left(\sqrt{a}\,e^{i\frac{\pi}{4}}\right)\,e^{2\pi\sqrt{a}\,e^{i\frac{\pi}{4}}}-1\right|^{2}$
as
\begin{align}
\frac{e^{2\pi\sqrt{2a}}}{(p^{2}-\sqrt{2a}\,p+a)^{2}}\left(p^{4}+a^{2}\right)-2\,\frac{e^{\pi\sqrt{2a}}}{p^{2}-\sqrt{2a}\,p+a}\left\{ \cos(\pi\sqrt{2a})\,\left(p^{2}-a\right)-\sqrt{2a}\,p\,\sin(\pi\sqrt{2a})\right\} +1\label{modulus squared-1}\\
=\frac{e^{2\pi\sqrt{2a}}\left(p^{2}+\sqrt{2a}\,p+a\right)}{p^{2}-\sqrt{2a}\,p+a}-2\,\frac{e^{\pi\sqrt{2a}}}{p^{2}-\sqrt{2a}\,p+a}\left\{ \cos(\pi\sqrt{2a})\,\left(p^{2}-a\right)-\sqrt{2a}\,p\,\sin(\pi\sqrt{2a})\right\} +1
\end{align}

Taking the conjugate of (\ref{literal expression elementary-1}) and
using (\ref{modulus squared-1}), we have after simple reductions
that $F_{p}(a)$ can be written as
\[
F_{p}(a):=\text{Im}\left\{ \frac{e^{i\frac{\pi}{4}}}{\sigma\left(\sqrt{a}\,e^{i\frac{\pi}{4}}\right)\,e^{2\pi\sqrt{a}\,e^{i\frac{\pi}{4}}}-1}\right\} =\frac{G_{p}(a)}{2\sqrt{2}},
\]
where $G_{p}(a)$ is given by (\ref{definition of Qp (a) intro}).
This completes the proof of Ramanujan's formula.
\end{proof}

We now derive particular cases of Ramanujan-type formulas, recovering
some results previously known but also gaining new insight. Since
the proofs of all the corollaries are nothing but a specialization
of (\ref{Analogue final in the statement}) as $p$ or $p^{\prime}$ tend
to $0^{+}$ or $\infty$, we will omit them.

\bigskip{}

We start with the following result, which can be obtained by letting
$p^{\prime}\rightarrow\infty$ in the previous Theorem. It provides
an infinite number of representations for $\zeta(1/2)$.

\begin{corollary}
For every $p\in\mathbb{R}^{+}$ and $x>0$, we have the infinitely
many representations for $\zeta(1/2)$,
\begin{align*}
\sum_{n=1}^{\infty}\frac{p^{2}+\lambda_{n}^{2}}{p\left(p+\frac{1}{\pi}\right)+\lambda_{n}^{2}}\,\frac{1}{e^{\lambda_{n}^{2}x}-1}=\frac{1}{4\left(1+\frac{1}{\pi p}\right)}+\frac{1}{2}\,\sqrt{\frac{\pi}{x}}\,\zeta\left(\frac{1}{2}\right)+\frac{\pi^{2}}{6x}\,\frac{1+\frac{3}{\pi p}(1+\frac{1}{\pi p})}{\left(1+\frac{1}{\pi p}\right)^{2}}+\frac{1}{2}\sqrt{\frac{\pi}{x}}\,\sum_{n=1}^{\infty}\frac{1}{\sqrt{n}}\times\\
\times\frac{\left(p^{2}-\frac{2\pi n}{x}\right)\left\{ \cos(2\pi\sqrt{\frac{\pi n}{x}})-\sin(2\pi\sqrt{\frac{\pi n}{x}})\right\} -e^{-2\pi\sqrt{\frac{\pi n}{x}}}\left(p^{2}-2\sqrt{\frac{\pi n}{x}}\,p+\frac{2\pi n}{x}\right)-2\sqrt{\frac{\pi n}{x}}\,p\left\{ \cos(2\pi\sqrt{\frac{\pi n}{x}})+\sin(2\pi\sqrt{\frac{\pi n}{x}})\right\} }{p^{2}\left(\cosh(2\pi\sqrt{\frac{\pi n}{x}})-\cos(2\pi\sqrt{\frac{\pi n}{x}})\right)+2\sqrt{\frac{\pi n}{x}}\,p\left(\sinh(2\pi\sqrt{\frac{\pi n}{x}})+\sin(2\pi\sqrt{\frac{\pi n}{x}})\right)+\frac{2\pi n}{x}\left(\cosh(2\pi\sqrt{\frac{\pi n}{x}})+\cos(2\pi\sqrt{\frac{\pi n}{x}})\right)}.
\end{align*}

In particular, Ramanujan's formula (\ref{Ramanujan formula def entry 8}) holds, together with
\begin{equation}
\sum_{n=1}^{\infty}\frac{1}{e^{(2n-1)^{2}x}-1}=\frac{1}{4}\,\sqrt{\frac{\pi}{x}}\,\zeta\left(\frac{1}{2}\right)+\frac{\pi^{2}}{8x}-\frac{1}{4}\sqrt{\frac{\pi}{x}}\,\sum_{n=1}^{\infty}\frac{1}{\sqrt{n}}\,\frac{\cos\left(\pi\sqrt{\frac{\pi n}{x}}\right)-\sin\left(\pi\sqrt{\frac{\pi n}{x}}\right)+e^{-\pi\sqrt{\frac{\pi n}{x}}}}{\cosh\left(\pi\sqrt{\frac{\pi n}{x}}\right)+\,\cos\left(\pi\sqrt{\frac{\pi n}{x}}\right)},\,\,\,x>0.\label{First example with odd integers-1}
\end{equation}
\end{corollary}

\bigskip{}

If, instead, we take $p\rightarrow\infty$ and fix $p^{\prime}$,
we are able to derive the result.
\begin{corollary}
For every $p^{\prime}\in\mathbb{R}^{+}$ and any $x>0$, the following
identity of Ramanujan type holds
\begin{align*}
\sum_{n=1}^{\infty}\frac{1}{\sigma^{\prime}\left(\frac{n^{2}x}{2\pi}\right)e^{n^{2}x}-1} & =\frac{1}{4}+\frac{1}{2}\sqrt{\frac{\pi}{x}}\,\zeta_{p^{\prime}}\left(\frac{1}{2}\right)+\frac{\pi^{2}}{6x}\,\frac{1}{1+\frac{1}{\pi p^{\prime}}}+\\
+\frac{1}{2}\sqrt{\frac{\pi}{x}}\,\sum_{n=1}^{\infty}\frac{p^{\prime2}+\lambda_{n}^{\prime2}}{p^{\prime}\left(p^{\prime}+\frac{1}{\pi}\right)+\lambda_{n}^{\prime2}} & \,\frac{1}{\sqrt{\lambda_{n}^{\prime}}}\cdot\left\{ \frac{\cos\left(2\pi\sqrt{\frac{\pi\lambda_{n}^{\prime}}{x}}\right)-\sin\left(2\pi\sqrt{\frac{\pi\lambda_{n}^{\prime}}{x}}\right)-e^{-2\pi\sqrt{\frac{\pi\lambda_{n}^{\text{\ensuremath{\prime}}}}{x}}}}{\cosh\left(2\pi\sqrt{\frac{\pi\lambda_{n}^{\prime}}{x}}\right)-\cos\left(2\pi\sqrt{\frac{\pi\lambda_{n}^{\prime}}{x}}\right)}\right\} .
\end{align*}

In particular,
\[
-\sum_{n=1}^{\infty}\frac{1}{e^{n^{2}x}+1}=\frac{1}{4}+\frac{1}{2}\sqrt{\frac{\pi}{x}}\,\left(\sqrt{2}-1\right)\,\zeta\left(\frac{1}{2}\right)+\sqrt{\frac{\pi}{2x}}\,\sum_{n=1}^{\infty}\frac{1}{\sqrt{2n-1}}\cdot\left\{ \frac{\cos\left(\pi\sqrt{\frac{\pi(4n-2)}{x}}\right)-\sin\left(\pi\sqrt{\frac{\pi(4n-2)}{x}}\right)-e^{-\pi\sqrt{\frac{\pi(4n-2)}{x}}}}{\cosh\left(\pi\sqrt{\frac{\pi(4n-2)}{x}}\right)-\cos\left(\pi\sqrt{\frac{\pi(4n-2)}{x}}\right)}\right\}.
\]

\end{corollary}

If we let $p^{\prime}\rightarrow0^{+}$ and fix $p\in\mathbb{R}^{+}$
in our Theorem 2.1, we are still able to find alternative representations
for $\zeta\left(\frac{1}{2}\right)$. This gives the following Corollary.

\begin{corollary}
For every $p\in\mathbb{R}^{+}$ and $x>0$, we have the representations
for $\zeta(1/2)$, 
\[
-\sum_{n=1}^{\infty}\frac{p^{2}+\lambda_{n}^{2}}{p\left(p+\frac{1}{\pi}\right)+\lambda_{n}^{2}}\cdot\frac{1}{e^{\lambda_{n}^{2}x}+1}=\frac{1}{4\left(1+\frac{1}{\pi p}\right)}+\frac{1}{2}\sqrt{\frac{\pi}{x}}\left(\sqrt{2}-1\right)\zeta\left(\frac{1}{2}\right)+\sqrt{\frac{\pi}{2x}}\,\sum_{n=1}^{\infty}\frac{G_{p}\left(\frac{\pi(2n-1)}{x}\right)}{\sqrt{2n-1}},
\]
where $G_{p}(a)$ is given by (\ref{definition of Qp (a) intro}). In particular,
the following relation takes place 
\[
\sum_{n=1}^{\infty}\frac{1}{e^{(2n-1)^{2}x}+1}=\frac{1}{4}\sqrt{\frac{\pi}{x}}\left(1-\sqrt{2}\right)\zeta\left(\frac{1}{2}\right)+\frac{1}{4}\sqrt{\frac{2\pi}{x}}\,\sum_{n=1}^{\infty}\frac{1}{\sqrt{2n-1}}\cdot\frac{\cos\left(\pi\sqrt{\frac{\pi(2n-1)}{2x}}\right)-\sin\left(\pi\sqrt{\frac{\pi(2n-1)}{2x}}\right)+e^{-\pi\sqrt{\frac{\pi(2n-1)}{2x}}}}{\cosh\left(\pi\sqrt{\frac{\pi(2n-1)}{2x}}\right)+\cos\left(\pi\sqrt{\frac{\pi(2n-1)}{2x}}\right)}.
\]

\end{corollary}

Finally, we state the corollary that we obtain once we take $p\rightarrow0^{+}$
in Theorem 2.1 above.

\begin{corollary}
For every $p^{\prime}\in\mathbb{R}^{+}$ and $x>0$, the following
identity takes place
\begin{align*}
\sum_{n=1}^{\infty}\frac{1}{\sigma^{\prime}\left(\frac{(2n-1)^{2}x}{2\pi}\right)e^{(2n-1)^{2}x}-1} & =\frac{\pi^{2}}{8x}\cdot\frac{1}{1+\frac{1}{\pi p^{\prime}}}+\frac{1}{4}\sqrt{\frac{\pi}{x}}\zeta_{p^{\prime}}\left(\frac{1}{2}\right)+\\
+\frac{1}{4}\sqrt{\frac{\pi}{x}}\,\sum_{n=1}^{\infty}\frac{p^{\prime2}+\lambda_{n}^{\prime2}}{p^{\prime}\left(p^{\prime}+\frac{1}{\pi}\right)+\lambda_{n}^{\prime2}}\cdot\frac{1}{\sqrt{\lambda_{n}^{\prime}}} & \cdot\frac{\sin\left(\pi\sqrt{\frac{\pi\lambda_{n}^{\prime}}{x}}\right)-\cos\left(\pi\sqrt{\frac{\pi\lambda_{n}^{\prime}}{x}}\right)-e^{-\pi\sqrt{\frac{\pi\lambda_{n}^{\prime}}{x}}}}{\cosh\left(\pi\sqrt{\frac{\pi\lambda_{n}^{\prime}}{x}}\right)+\cos\left(\pi\sqrt{\frac{\pi\lambda_{n}^{\prime}}{x}}\right)}.
\end{align*}
   
\end{corollary}

\section{A Generalization of Watson's formula}

\subsection{First Analogue of Watson's Formula}
In order to generalize Watson's result (\ref{Watson Formula intro}) in Koshliakov's setting, we first need to generalize the series on the left side of (\ref{Watson Formula intro}). For $\text{Re}(x)>0$ and $\text{Re}(s)>\frac{1}{2}$, we shall consider
\begin{equation}
\varphi_{p}(s,x):=\sum_{n=1}^{\infty}\frac{p^{2}+\lambda_{n}^{2}}{p\left(p+\frac{1}{\pi}\right)+\lambda_{n}^{2}}\,\frac{1}{\left(\lambda_{n}^{2}+x^{2}\right)^{s}},\,\,\,\,\text{Re}(s)>\frac{1}{2}.\label{Watson series}
\end{equation}

This constitutes an analogue the infinite series appearing in (\ref{Watson Formula intro}), which is obtained in the limiting case $p\rightarrow \infty$. We now extend Watson's formula.
\begin{theorem}
Let $x>0$ and $\text{Re}(s)>\frac{1}{2}$. Then the following
generalization of Watson's formula (\ref{Watson Formula intro}) takes
place
\begin{align}
\sum_{n=1}^{\infty}\frac{p^{2}+\lambda_{n}^{2}}{p\left(p+\frac{1}{\pi}\right)+\lambda_{n}^{2}}\,\frac{1}{\left(\lambda_{n}^{2}+x^{2}\right)^{s}} & =\frac{\sqrt{\pi}\,x^{1-2s}}{2\Gamma(s)}\Gamma\left(s-\frac{1}{2}\right)-\frac{1}{2}\,\frac{x^{-2s}}{1+\frac{1}{\pi p}}+\frac{2\pi^{s}x^{\frac{1}{2}-s}}{\Gamma(s)}\,\sum_{m=1}^{\infty}(-1)^{m}m^{s-\frac{1}{2}}K_{s-\frac{1}{2}}(2\pi mx)+\nonumber \\
+\frac{2\,\pi^{s}x^{\frac{1}{2}-s}}{\Gamma(s)}\,\sum_{m=1}^{\infty}e^{2\pi pm}m^{s-\frac{1}{2}}\,\sum_{\ell=1}^{m}\left(\begin{array}{c}
m\\
\ell
\end{array}\right) & (-1)^{m-\ell}\left(4\pi mp\right)^{\ell}\intop_{1}^{\infty}t^{s-\frac{1}{2}}\,e^{-2\pi mpt}\,\frac{(t-1)^{\ell-1}}{(\ell-1)!}\,K_{s-\frac{1}{2}}\left(2\pi xmt\right)\,dt.\label{formula valid for all s}
\end{align}

In particular, for $\frac{1}{2}<\text{Re}(s)<1$, one has the integral
representation
\begin{align}
\sum_{n=1}^{\infty}\frac{p^{2}+\lambda_{n}^{2}}{p\left(p+\frac{1}{\pi}\right)+\lambda_{n}^{2}}\,\frac{1}{\left(\lambda_{n}^{2}+x^{2}\right)^{s}} & =\frac{\sqrt{\pi}\,x^{1-2s}}{2\Gamma(s)}\Gamma\left(s-\frac{1}{2}\right)-\frac{1}{2}\,\frac{x^{-2s}}{1+\frac{1}{\pi p}}+\nonumber \\
+2^{2-2s}\,x^{1-2s}\sin(\pi s)\, & \intop_{0}^{\infty}\,\frac{y^{-s}(y+1)^{-s}}{\sigma\left((2y+1)\,x\right)e^{2\pi(2y+1)x}-1}dy,\label{formula in a region compact to be extended}
\end{align}
with the right-hand side of the previous expression being the analytic
continuation of the series (\ref{Watson series}) to the half-plane
$\text{Re}(s)<\frac{1}{2}$
\end{theorem}

\begin{proof}
Let $\mu>\frac{1}{2}$ and consider a fixed $s\in\mathbb{C}$ such
that $\text{Re}(s)>\mu$. Under this condition, the following integral
representation holds {[}\cite{ERDELIY}, p. 348, eq. 7.3.15{]}
\begin{equation}
\frac{1}{2\pi i}\intop_{\mu-i\infty}^{\mu+i\infty}\Gamma(z)\,\Gamma(s-z)\,x^{-2z}dz=\frac{\Gamma(s)}{(1+x^{2})^{s}},\,\,\,\,x>0.\label{Beta integral beginning}
\end{equation}

Using this integral representation in the series (\ref{Watson series}),
we have that 
\begin{equation}
\sum_{n=1}^{\infty}\frac{p^{2}+\lambda_{n}^{2}}{p\left(p+\frac{1}{\pi}\right)+\lambda_{n}^{2}}\,\frac{1}{\left(\lambda_{n}^{2}+x^{2}\right)^{s}}=\frac{x^{-2s}}{\Gamma(s)}\,\frac{1}{2\pi i}\,\intop_{\mu-i\infty}^{\mu+i\infty}\Gamma(z)\,\Gamma(s-z)\,\zeta_{p}(2z)\,x^{2z}dz\label{integral representation first series}
\end{equation}
where the interchange of the infinite series with the contour integral
(\ref{Beta integral beginning}) is justified by absolute convergence
of $\zeta_{p}(2z)$ in the half-plane $\text{Re}(z)>\mu>\frac{1}{2}$
and Stirling's formula $|\left|\Gamma\left(\mu+it\right)\Gamma\left(s-\mu-it\right)\right|=O\left(|t|^{\text{Re}(s)-1}e^{-\pi|t|}\right)$,
as $|t|\rightarrow\infty$. 

\bigskip{}

As in the proof of Theorem 2.1., let us now move the line of integration
to $\text{Re}(z)=\frac{1}{2}-\mu$ and integrate along a positively
oriented rectangular contour $\mathscr{R}_{\mu}(T)$ containing the
vertices $\mu\pm iT$ and $\frac{1}{2}-\mu\pm iT$, $T>0$. Due to the trivial zeros of $\zeta_{p}(2z)$
and the choice $\text{Re}(s)>\mu>\frac{1}{2}$, the only singularities
of the integrand inside $\mathscr{R}_{\mu}(T)$ are located at $z=\frac{1}{2}$
and at $z=0$. Like in (\ref{behavior integrals}), we can estimate trivially the integrals along
the horizontal segments $\left[\frac{1}{2}-\mu\pm iT,\,\mu\pm iT\right]$
when $T\rightarrow\infty$ and we can show that they vanish. An application of Cauchy's residue theorem gives
\begin{equation}
\frac{1}{2\pi i}\,\intop_{\mu-i\infty}^{\mu+i\infty}\Gamma(z)\,\Gamma(s-z)\,\zeta_{p}(2z)\,x^{2z}dz=\frac{1}{2\pi i}\,\intop_{\frac{1}{2}-\mu-i\infty}^{\frac{1}{2}-\mu+i\infty}\Gamma(z)\,\Gamma(s-z)\,\zeta_{p}(2z)\,x^{2z}dz-\frac{1}{2}\,\frac{\Gamma(s)}{1+\frac{1}{\pi p}}+\frac{\sqrt{\pi}\,x}{2}\Gamma\left(s-\frac{1}{2}\right).\label{formula after Cauchy Watson}
\end{equation}

Invoking the functional equation for $\zeta_{p}(z)$ (\ref{functional equation Kosh}),
we can simplify the integral on the right-hand side of (\ref{formula after Cauchy Watson})
in the form
\[
\frac{1}{2\pi i}\,\intop_{\frac{1}{2}-\mu-i\infty}^{\frac{1}{2}-\mu+i\infty}\Gamma(z)\,\Gamma(s-z)\,\zeta_{p}(2z)\,x^{2z}dz=\sqrt{\pi}\,x\sum_{m=1}^{\infty}\,\frac{1}{2\pi i}\,\intop_{\mu-i\infty}^{\mu+i\infty}\,\Gamma\left(z\right)\,\Gamma\left(s+z-\frac{1}{2}\right)\,\left(2z,2\pi pm\right)_{m}\,(\pi xm)^{-2z}\,dz,
\]
because the representation of $\eta_{p}(2z)$, (\ref{secondary zeta function Koshliakov sense}), 
converges absolutely in the half-plane $\text{Re}(z)>\frac{1}{2}$.
\\

For each fixed $m\in\mathbb{N}$, we now evaluate the integral
\[
I_{m,p}(s,x):=\frac{1}{2\pi i}\,\intop_{\mu-i\infty}^{\mu+i\infty}\left(2z,2\pi pm\right)_{m}\,\Gamma\left(z\right)\,\Gamma\left(s+z-\frac{1}{2}\right)\,(\pi\,x\,m)^{-2z}\,dz.
\]

This is possible by using representation (\ref{representation Koshliakov given in poage 24}),
which gives
\[
I_{m,p}(s,x)=I_{m,p}^{(1)}(s,x)+I_{m,p}^{(2)}(s,x),
\]
with 
\begin{equation}
I_{m,p}^{(1)}(s,x)=\frac{(-1)^{m}}{2\pi i}\,\intop_{\mu-i\infty}^{\mu+i\infty}\Gamma\left(z\right)\,\Gamma\left(s+z-\frac{1}{2}\right)\,(\pi\,x\,m)^{-2z}\,dz=2\,(-1)^{m}\,\left(\pi x\,m\right)^{s-\frac{1}{2}}\,K_{s-\frac{1}{2}}\left(2\pi x\,m\right) \label{first partition in the proof}
\end{equation}
and
\begin{align*}
I_{m,p}^{(2)}(s,x) & =\frac{1}{2\pi i}\,\intop_{\mu-i\infty}^{\mu+i\infty}e^{2\pi pm}\,\sum_{\ell=1}^{m}\left(\begin{array}{c}
m\\
\ell
\end{array}\right)\left(4\pi p\right)^{\ell}(-1)^{m-\ell}\,\intop_{m}^{\infty}t^{-2z}e^{-2\pi p\,t}\,\frac{(t-m)^{\ell-1}}{(\ell-1)!}\,\Gamma\left(z\right)\,\Gamma\left(s+z-\frac{1}{2}\right)\,\left(\pi x\right)^{-2z}\,dt\,dz\\
 & =e^{2\pi pm}\,\sum_{\ell=1}^{m}\left(\begin{array}{c}
m\\
\ell
\end{array}\right)\left(4\pi p\right)^{\ell}(-1)^{m-\ell}\intop_{m}^{\infty}e^{-2\pi pt}\,\frac{(t-m)^{\ell-1}}{(\ell-1)!}\,\frac{1}{2\pi i}\,\intop_{\mu-i\infty}^{\mu+i\infty}\Gamma\left(z\right)\,\Gamma\left(s+z-\frac{1}{2}\right)\,\left(\pi xt\right)^{-2z}dz\,dt\\
 & =2\,(\pi x)^{s-\frac{1}{2}}e^{2\pi pm}m^{s-\frac{1}{2}}\,\sum_{\ell=1}^{m}\left(\begin{array}{c}
m\\
\ell
\end{array}\right)(-1)^{m-\ell}\,\left(4\pi pm\right)^{\ell}\intop_{1}^{\infty}t^{s-\frac{1}{2}}\,e^{-2\pi mpt}\,\frac{(t-1)^{\ell-1}}{(\ell-1)!}\,K_{s-\frac{1}{2}}\left(2\pi xmt\right)\,dt,
\end{align*}
where the interchange of the operations is possible due to absolute
convergence, which in its turn is justified by Stirling's formula for the product $\Gamma(z)\,\Gamma(s+z-\frac{1}{2})$. On the third equality above, as well as in (\ref{first partition in the proof}),
we have used the well-known integral representation for the Macdonald
function [\cite{ERDELIY}, p. 349, 7.3 (17)]
\begin{equation}
K_{\nu}(x)=\frac{1}{2\pi i}\,\intop_{\mu-i\infty}^{\mu+i\infty}2^{s-2}\Gamma\left(\frac{s-\nu}{2}\right)\Gamma\left(\frac{s+\nu}{2}\right)\,x^{-s}ds,\,\,\,\,\,x>0,\,\,\,\mu:=\text{Re}(s)>\max\{0,\text{Re}(\nu)\}. \label{Mellin barnes Macdonald}
\end{equation}
Combining all the expressions given above yields the formula  
\begin{align}
\sum_{n=1}^{\infty}\frac{p^{2}+\lambda_{n}^{2}}{p\left(p+\frac{1}{\pi}\right)+\lambda_{n}^{2}}\,\frac{1}{\left(\lambda_{n}^{2}+x^{2}\right)^{s}} & =\frac{\sqrt{\pi}\,x^{1-2s}}{2\Gamma(s)}\Gamma\left(s-\frac{1}{2}\right)-\frac{1}{2}\,\frac{x^{-2s}}{1+\frac{1}{\pi p}}+\frac{2\pi^{s}\,x^{\frac{1}{2}-s}}{\Gamma(s)}\,\sum_{m=1}^{\infty}(-1)^{m}\,m^{s-\frac{1}{2}}\,K_{s-\frac{1}{2}}\left(2\pi xm\right)\nonumber \\
+\frac{2\,\pi^{s}\,x^{\frac{1}{2}-s}}{\Gamma(s)}\,\sum_{m=1}^{\infty}e^{2\pi pm}m^{s-\frac{1}{2}}\,\sum_{\ell=1}^{m}\left(\begin{array}{c}
m\\
\ell
\end{array}\right) & (-1)^{m-\ell}\left(4\pi mp\right)^{\ell}\intop_{1}^{\infty}t^{s-\frac{1}{2}}\,e^{-2\pi mpt}\,\frac{(t-1)^{\ell-1}}{(\ell-1)!}\,K_{s-\frac{1}{2}}\left(2\pi xmt\right)\,dt.\label{Watson formula with explicit bessel}
\end{align}

\bigskip{}

Although (\ref{Watson formula with explicit bessel}) constitutes
already a generalization of the classical case (\ref{Watson Formula intro}), we will now write
the latter series in a more elegant form, which is well-defined in the half-plane $\text{Re}(s)\leq\frac{1}{2}$. We now claim that the
right-hand side of (\ref{Watson formula with explicit bessel}) constitutes
the analytic continuation of the series $\varphi_{p}(s,x)$, (\ref{Watson series}),
to the entire complex plane. We prove this claim by showing that the
series on the right defines an entire function of $s\in\mathbb{C}$.
Indeed, it is enough to check this for any $s$ in the half-plane
$\text{Re}(s)\leq\frac{1}{2}$ because for $\text{Re}(s)>\frac{1}{2}$,
the absolute convergence of the left-hand side of (\ref{Watson formula with explicit bessel})
already assures this.

\bigskip{}

From now on, we analyze the infinite series on the left of (\ref{Watson formula with explicit bessel})
under the hypothesis that $\text{Re}(s)<1$, which obviously includes
the remaining case $\text{Re}(s)\leq\frac{1}{2}$. To that end, recall
the well-known integral formula {[}\cite{NIST}, p. 252, relation 10.32.8{]}
\begin{equation}
K_{\nu}(z)=\frac{\sqrt{\pi}z^{\nu}}{2^{\nu}\Gamma\left(\nu+\frac{1}{2}\right)}\,\intop_{1}^{\infty}e^{-zt}\left(t^{2}-1\right)^{\nu-\frac{1}{2}}dt,\,\,\,\,\,\text{Re}(\nu)>-\frac{1}{2},\,\,\,|\arg(z)|<\frac{\pi}{2}.\label{First representation Re-1}
\end{equation}

From the fact that $K_{-\nu}(z)=K_{\nu}(z)$, (\ref{First representation Re-1})
can be rewritten in the equivalent form
\begin{equation}
\left(\frac{z}{2}\right)^{\nu}\,K_{\nu}(z)=\frac{\sqrt{\pi}}{\Gamma\left(\frac{1}{2}-\nu\right)}\,\intop_{1}^{\infty}e^{-zt}\left(t^{2}-1\right)^{-\nu-\frac{1}{2}}dt,\,\,\,\text{Re}(\nu)<\frac{1}{2},\,\,\,|\arg(z)|<\frac{\pi}{2},\label{Mehler type-1}
\end{equation}
that can be simplified to
\begin{equation}
\left(\frac{z}{2}\right)^{\nu}K_{\nu}(z)=\frac{2^{-2\nu}\,\sqrt{\pi}\,e^{-z}}{\Gamma\left(\frac{1}{2}-\nu\right)}\,\intop_{0}^{\infty}e^{-2zt}t^{-\nu-\frac{1}{2}}(t+1)^{-\nu-\frac{1}{2}}dt,\,\,\,\,\text{Re}(\nu)<\frac{1}{2},\,\,\,\,|\arg(z)|<\frac{\pi}{2}.\label{as laplace traaansform}
\end{equation}

Assuming that $s$ lies in the region $\text{Re}(s)\leq\frac{1}{2}$ and using
(\ref{as laplace traaansform}), we see that the fractional integral
appearing on the right-hand side of (\ref{Watson formula with explicit bessel})
can be simplified in the following manner, 
\begin{align*}
J_{p,\ell}(m,x) & :=\intop_{1}^{\infty}t^{s-\frac{1}{2}}\,e^{-2\pi mpt}\,\frac{(t-1)^{\ell-1}}{(\ell-1)!}\,K_{s-\frac{1}{2}}\left(2\pi xmt\right)\,dt\\
=\frac{2^{1-2s}\,\sqrt{\pi}\,}{\Gamma\left(1-s\right)}\,\left(\pi xm\right)^{\frac{1}{2}-s} & \intop_{0}^{\infty}\exp\left(-2\pi m\left(p+(2y+1)\,x\right)\right)\left(2\pi m\left(p+(2y+1)\,x\right)\right)^{-\ell}\frac{dy}{y^{s}\left(y+1\right)^{s}}
\end{align*}
where the interchange of the orders of integration comes from the
absolute convergence of both integrals under the hypothesis $\text{Re}(s)\leq\frac{1}{2}<1$.
Note now that the integral in the last expression converges absolutely, which allows
us to take the finite sum over $\ell\in\{0,...,m\}$, resulting in 
\begin{align}
\sum_{m=1}^{\infty}e^{2\pi pm}m^{s-\frac{1}{2}}\,\sum_{\ell=1}^{m}\left(\begin{array}{c}
m\\
\ell
\end{array}\right)(-1)^{m-\ell}\left(4\pi mp\right)^{\ell}\,J_{p,\ell}(m,x) & =\nonumber \\
=\frac{2^{1-2s}\,\pi^{1-s}x^{\frac{1}{2}-s}}{\Gamma\left(1-s\right)}\,\sum_{m=1}^{\infty}e^{2\pi pm}\,\intop_{0}^{\infty}\left\{ \left(\frac{p-(2y+1)x}{p+(2y+1)x}\right)^{m}+(-1)^{m-1}\right\}  & \frac{e^{-2\pi m\left(p+(2y+1)x\right)}}{y^{s}\left(y+1\right)^{s}}\,dy\nonumber \\
=\frac{2^{1-2s}\,\pi^{1-s}x^{\frac{1}{2}-s}}{\Gamma\left(1-s\right)}\intop_{0}^{\infty}\frac{\,y^{-s}(y+1)^{-s}}{\sigma\left((2y+1)\,x\right)e^{2\pi(2y+1)x}-1}dy+\frac{2^{1-2s}\,\pi^{1-s}x^{\frac{1}{2}-s}}{\Gamma\left(1-s\right)} & \,\sum_{m=1}^{\infty}(-1)^{m-1}\intop_{0}^{\infty}\frac{e^{-2\pi m(2y+1)x}}{y^{s}(y+1)^{s}}\,dy\label{computations leading to final expreeeessssssiiiiooooon}
\end{align}
where in the last step we have used once more absolute convergence and
the geometric series (\ref{power series def}). Clearly, the former 
integral converges absolutely and uniformly for every $s$ in the
half-plane $\text{Re}(s)<1$, as the elementary bound shows
\begin{equation}
\left|\intop_{0}^{\infty}\frac{y^{-s}(y+1)^{-s}}{\sigma\left((2y+1)\,x\right)e^{2\pi(2y+1)x}-1}\,dy\right|\leq\intop_{0}^{\infty}\frac{y^{-\sigma}(y+1)^{-\sigma}}{e^{2\pi(2y+1)x}-1}dy.\label{bound proof entire part Watson}
\end{equation}

This proves that the infinite series on the right of (\ref{Watson formula with explicit bessel})
defines an entire function of $s\in\mathbb{C}$ {[}\cite{whittaker_watson},
p. 92{]} and then (\ref{formula valid for all s}) constitutes the
analytic continuation of the series to the half-plane $\text{Re}(s)<\frac{1}{2}$.
Finally, recalling once more (\ref{as laplace traaansform}), we note
that the second series on the last equation of (\ref{computations leading to final expreeeessssssiiiiooooon})
is precisely the symmetric of the infinite series appearing on the
top of the right-hand side of (\ref{Watson formula with explicit bessel}).
Joining (\ref{computations leading to final expreeeessssssiiiiooooon})
with (\ref{Watson formula with explicit bessel}) leads to (\ref{formula in a region compact to be extended})
as an equivalent continuation to $\text{Re}(s)<\frac{1}{2}$. 
\end{proof}

\begin{remark}
It is also possible to write the series on the right-hand side (\ref{formula valid for all s})
as involving integrals of the Whittaker function. Expressions like
this are also helpful to have a grasp on the analytic continuation
of the series (\ref{Watson series}). We could also replace the integral
in (\ref{formula in a region compact to be extended}) by an integral
with respect to a Hankel contour. This would allow a representation
valid for every $s\in\mathbb{C}$. For our purposes, however, it suffices to use the above representation (\ref{formula in a region compact to be extended}). 
\end{remark}

By letting $p\rightarrow\infty$ in (\ref{formula valid for all s})
or in (\ref{formula in a region compact to be extended}), we can
recover Watson's formula (\ref{Watson Formula intro}) and by letting
$p\rightarrow0^{+}$, we can derive an analogue of it. We remark that our next
formula (\ref{Watson for odd sequence}) could be also achieved by
the general methods given in {[}\cite{rysc_I}, Example 5.1, eq. (5.2){]}.
\begin{corollary}
Let $s$ be any complex number such that $\text{Re}(s)>\frac{1}{2}$
and $x>0$. Then Watson's formula (\ref{Watson Formula intro}) holds. Moreover, one has the companion identity 
\begin{equation}
\sum_{n=1}^{\infty}\frac{1}{\left(\left(2n-1\right)^{2}+x^{2}\right)^{s}}=\frac{\sqrt{\pi}\,x^{1-2s}}{4\Gamma(s)}\Gamma\left(s-\frac{1}{2}\right)+\frac{2^{\frac{1}{2}-s}\,x^{\frac{1}{2}-s}\,\pi^{s}}{\Gamma(s)}\,\sum_{m=1}^{\infty}(-1)^{m}m^{s-\frac{1}{2}}K_{s-\frac{1}{2}}(\pi xm).\label{Watson for odd sequence}
\end{equation}
\end{corollary}
\begin{proof}
Since the proofs of (\ref{Watson for odd sequence}) and (\ref{Watson Formula intro}) are analogous, we will only see that (\ref{formula in a region compact to be extended})
reduces to (\ref{Watson Formula intro}) as $p\rightarrow\infty$. For this, it suffices to analyze the continuation of the integral
in (\ref{formula in a region compact to be extended}) to the region
$\text{Re}(s)>1$ in this limit. Assume that $\text{Re}(s)<1$: by
an application of the dominated convergence Theorem, we can take $p\rightarrow\infty$
on the right-hand side of (\ref{formula in a region compact to be extended}),
giving
\begin{align}
\frac{\sqrt{\pi}\,x^{1-2s}}{2\Gamma(s)}\Gamma\left(s-\frac{1}{2}\right)-\frac{x^{-2s}}{2}+2^{1-2s}\,x^{1-2s}\sin(\pi s)\,\intop_{0}^{\infty}\frac{y^{-s}(y+1)^{-s}}{e^{2\pi(2y+1)x}-1}dy\nonumber \\
=\frac{\sqrt{\pi}\,x^{1-2s}}{2\Gamma(s)}\Gamma\left(s-\frac{1}{2}\right)-\frac{x^{-2s}}{2}+2^{1-2s}\,x^{1-2s}\sin(\pi s)\sum_{n=1}^{\infty}e^{-2\pi nx}\intop_{0}^{\infty}y^{-s}(y+1)^{-s}e^{-4\pi nx\,y}dy\nonumber \\
=\frac{\sqrt{\pi}\,x^{1-2s}}{2\Gamma(s)}\Gamma\left(s-\frac{1}{2}\right)-\frac{x^{-2s}}{2}+\frac{\pi^{s}x^{\frac{1}{2}-s}}{\Gamma(s)}\,\sum_{n=1}^{\infty}n^{s-\frac{1}{2}}\,K_{s-\frac{1}{2}}(2\pi nx),\,\,\,\,\text{Re}(s)<\frac{1}{2},\label{series obtained}
\end{align}
where the second step is justified by invoking the geometric series
and arguing by absolute convergence (recall that $K_\nu(x)=O(e^{-x}/\sqrt{x}),\,\,x\rightarrow\infty$) and the last equality comes from
the integral representation (\ref{as laplace traaansform}). Since the series on the right-hand side of (\ref{series obtained})
converges absolutely and uniformly for every $s\in\mathbb{C}$ and
its summands are entire functions of $s$, we see that the expression
(\ref{series obtained}) is valid for any $s\in\mathbb{C}\setminus\{\frac{1}{2}-k,\,k\in\mathbb{N}_{0}\}$.
Therefore, (\ref{series obtained}) and the series (\ref{Watson series})
must coincide for $\text{Re}(s)>\frac{1}{2}$. 
\end{proof}

We now obtain a result in which the restriction $\frac{1}{2}<\text{Re}(s)<1$ in our Theorem 3.1.
is relaxed. Our proof is completely analogous to Watson's {[}\cite{watson_reciprocal},
p. 300{]}.

\begin{corollary}
Let $N>0$ be an integer. For every $x>0$ and every $s$ satisfying
the condition $-N+\frac{1}{2}<\text{Re}(s)<1$, we have the identity
\begin{align}
\sum_{n=1}^{\infty}\frac{p^{2}+\lambda_{n}^{2}}{p\left(p+\frac{1}{\pi}\right)+\lambda_{n}^{2}}\,\left\{ \frac{1}{\left(\lambda_{n}^{2}+x^{2}\right)^{s}}-\sum_{m=0}^{N-1}\left(\begin{array}{c}
-s\\
m
\end{array}\right)\,\frac{x^{2m}}{\lambda_{n}^{2s+2m}}\right\} +\sum_{m=0}^{N-1}\left(\begin{array}{c}
-s\\
m
\end{array}\right)x^{2m}\zeta_{p}\left(2s+2m\right)=\nonumber \\
=\frac{\sqrt{\pi}\,x^{1-2s}}{2\Gamma(s)}\Gamma\left(s-\frac{1}{2}\right)-\frac{1}{2}\,\frac{x^{-2s}}{1+\frac{1}{\pi p}}+2^{2-2s}\,x^{1-2s}\sin(\pi s)\,\intop_{0}^{\infty}\frac{y^{-s}(y+1)^{-s}}{\sigma\left((2y+1)\,x\right)e^{2\pi(2y+1)x}-1}dy.\label{identity concerning analytic continuation}
\end{align}
\end{corollary}

\begin{proof}
Note that, for $\text{Re}(s)>\frac{1}{2}$, we can write the left-hand
side of (\ref{formula in a region compact to be extended}) in the
following form
\begin{equation}
\sum_{n=1}^{\infty}\frac{p^{2}+\lambda_{n}^{2}}{p\left(p+\frac{1}{\pi}\right)+\lambda_{n}^{2}}\,\left\{ \frac{1}{\left(\lambda_{n}^{2}+x^{2}\right)^{s}}-\sum_{m=0}^{N-1}\left(\begin{array}{c}
-s\\
m
\end{array}\right)\,\frac{x^{2m}}{\lambda_{n}^{2s+2m}}\right\} +\sum_{m=0}^{N-1}\left(\begin{array}{c}
-s\\
m
\end{array}\right)x^{2m}\zeta_{p}\left(2s+2m\right).\label{second extension of Watson series}
\end{equation}

Let $N$ be a positive such that, for $n>N$ the $x/\lambda_{n}<1$.
From the generalized binomial theorem and $n>N$, we have 
\[
\frac{1}{\left(\lambda_{n}^{2}+x^{2}\right)^{s}}=\sum_{m=0}^{\infty}\left(\begin{array}{c}
-s\\
m
\end{array}\right)\,\frac{x^{2m}}{\lambda_{n}^{2m+2s}}.
\]

Hence, (\ref{second extension of Watson series}) can be actually
expressed as
\begin{align}
\sum_{n=1}^{N}\frac{p^{2}+\lambda_{n}^{2}}{p\left(p+\frac{1}{\pi}\right)+\lambda_{n}^{2}}\,\left\{ \frac{1}{\left(\lambda_{n}^{2}+x^{2}\right)^{s}}-\sum_{m=0}^{N-1}\left(\begin{array}{c}
-s\\
m
\end{array}\right)\,\frac{x^{2m}}{\lambda_{n}^{2s+2m}}\right\} +\nonumber \\
\sum_{n=N+1}^{\infty}\frac{p^{2}+\lambda_{n}^{2}}{p\left(p+\frac{1}{\pi}\right)+\lambda_{n}^{2}}\left\{ \sum_{m=N}^{\infty}\left(\begin{array}{c}
-s\\
m
\end{array}\right)\,\frac{x^{2m}}{\lambda_{n}^{2s+2m}}\right\} +\sum_{m=0}^{N-1}\left(\begin{array}{c}
-s\\
m
\end{array}\right)x^{2m}\zeta_{p}\left(2s+2m\right).\label{rewriting once more}
\end{align}

Note that the general term of the infinite series $\sum_{m\geq N+1}$
is $O\left(\lambda_{n}^{-2\sigma-2N}\right)$ for large $n$. This
means that this series converges absolutely and uniformly with respect
to $s$ in the region $\text{Re}(s)>-N+\frac{1}{2}$ (because $\lambda_{n}\geq n-\frac{1}{2}$
for every $p\in\mathbb{R}_{+}$). Since any of its terms are analytic
functions of $s$, an application of Weierstrass' M-test shows that
(\ref{second extension of Watson series}) represents an analytic
function in the region $\text{Re}(s)>-N+\frac{1}{2}$. By our argument
above, we know that the right-hand side of (\ref{identity concerning analytic continuation})
represents also an analytic function in the half-plane $\text{Re}(s)<1$.
Thus, by the principle of analytic continuation, we have that (\ref{identity concerning analytic continuation})
must hold.
\end{proof}

\bigskip{}

From the previous result, we can obtain the following identity.

\begin{corollary}
For every $p\in\mathbb{R}_{+}$ and $x>0$, the following formula holds
\begin{equation*}
\sum_{n=1}^{\infty}\frac{p^{2}+\lambda_{n}^{2}}{p\left(p+\frac{1}{\pi}\right)+\lambda_{n}^{2}}\left\{ \frac{1}{\sqrt{\lambda_{n}^{2}+x^{2}}}-\frac{1}{\lambda_{n}}\right\} +C_{p}^{(1)}+\log\left(\frac{x}{2}\right)+\frac{1}{2x}\cdot\frac{1}{1+\frac{1}{\pi p}} =2\,\intop_{0}^{\infty}\frac{1}{\sqrt{y^{2}+y}}\cdot\frac{dy}{\sigma\left((2y+1)\,x\right)e^{2\pi(2y+1)x}-1},\label{meromorphic expansion continuation Kosh}
\end{equation*}
where $C_{p}^{(1)}$ denotes Koshliakov's generalization of Euler's
constant $\gamma$ {[}\cite{KOSHLIAKOV}, p. 46, eq. (46){]},
\begin{equation}
C_{p}^{(1)}:=\lim_{n\rightarrow\infty}\left\{ \sum_{j=1}^{n-1}\frac{p^{2}+\lambda_{j}^{2}}{p\left(p+\frac{1}{\pi}\right)+\lambda_{j}^{2}}\cdot\frac{1}{\lambda_{j}}-\log\left(\lambda_{n}\right)\right\} .\label{Euler Mascheroni Koshliakov sense}
\end{equation}
\end{corollary}

\begin{proof}

Let us take $N=1$ in the previous Corollary and take the limit $s\rightarrow\frac{1}{2}$
on both sides of (\ref{identity concerning analytic continuation}): this yields
\begin{align*}
\sum_{n=1}^{\infty}\frac{p^{2}+\lambda_{n}^{2}}{p\left(p+\frac{1}{\pi}\right)+\lambda_{n}^{2}}\left\{ \frac{1}{\sqrt{\lambda_{n}^{2}+x^{2}}}-\frac{1}{\lambda_{n}}\right\} +\lim_{s\rightarrow\frac{1}{2}}\left[\zeta_{p}(2s)-\frac{\sqrt{\pi}\,x^{1-2s}}{2\Gamma(s)}\Gamma\left(s-\frac{1}{2}\right)\right]\\
=-\frac{1}{2x}\cdot\frac{1}{1+\frac{1}{\pi p}}+2\,\intop_{0}^{\infty}\frac{1}{\sqrt{y^{2}+y}}\cdot\frac{dy}{\sigma\left((2y+1)\,x\right)e^{2\pi(2y+1)x}-1}.
\end{align*}

To find the limit as $s\rightarrow\frac{1}{2}$, we use the well-known Laurent
expansions around this point,
\begin{equation}
\frac{1}{\Gamma(s)}=\frac{1}{\sqrt{\pi}}+\frac{\gamma+2\log(2)}{\sqrt{\pi}}\left(s-\frac{1}{2}\right)+O\left(\left(s-\frac{1}{2}\right)^{2}\right),\label{first take meromorphic}
\end{equation}
\begin{equation}
x^{-2s}=x^{-1}\left(1-2\log(x)\,\left(s-\frac{1}{2}\right)+O\left(\left(s-\frac{1}{2}\right)^{2}\right)\right),\label{meromorphic x}
\end{equation}
\begin{equation}
\Gamma\left(s-\frac{1}{2}\right)=\frac{1}{s-\frac{1}{2}}-\gamma+O\left(s-\frac{1}{2}\right)\label{meromorphic Gamma}
\end{equation}
and finally,
\begin{equation}
\zeta_{p}(2s)=\frac{1}{2s-1}+C_{p}^{(1)}+O\left(s-\frac{1}{2}\right),\label{Koshliakov meromorphic expansion}
\end{equation}
with the last expression coming from {[}\cite{KOSHLIAKOV}, p. 48, eq. (51){]}.
Deducing our corollary from (\ref{first take meromorphic})-(\ref{Koshliakov meromorphic expansion})
is now immediate.
\end{proof}

\bigskip{}

By letting $p\rightarrow\infty$ and $p\rightarrow0^{+}$ we obtain a formula also obtained by Watson and an analogue of it, which seems to be new.
\begin{corollary}
For $x>0$, the following formulas are valid
\begin{equation}
\sum_{n=1}^{\infty}\left\{ \frac{1}{\sqrt{n^{2}+x^{2}}}-\frac{1}{n}\right\} +\frac{1}{2x}+\gamma+\log\left(\frac{x}{2}\right)=2\,\sum_{n=1}^{\infty}K_{0}(2\pi nx),\label{Analytic continuation Watson formula p infinity case}
\end{equation}
\begin{equation}
\sum_{n=1}^{\infty}\left\{ \frac{1}{\sqrt{\left(2n-1\right)^{2}+x^{2}}}-\frac{1}{2n-1}\right\} +\frac{\gamma}{2}+\frac{\log\left(x\right)}{2}=\sum_{n=1}^{\infty}(-1)^{n}\,K_{0}(\pi nx).\label{analytic continuation Watson p zero particular case}
\end{equation}
\end{corollary}

\bigskip{}

As a particular case of our generalization of Watson's result, we
can evaluate the series (\ref{Watson series}) for integer argument
$s$. We will do this for $s=1,\,2$. We remark that there are other
ways of proving both formulas of our next corollary, for instance involving residue theory. In fact,
Koshliakov gave a direct proof of this result on page 34 of his paper
\cite{KOSHLIAKOV}. We only give a proof of the first formula, the second one 
being analogous.
\begin{corollary}
Let $p\in\mathbb{R}_{+}$ and $x>0$. Then the following identities take
place 
\begin{equation}
\sum_{n=1}^{\infty}\frac{p^{2}+\lambda_{n}^{2}}{p\left(p+\frac{1}{\pi}\right)+\lambda_{n}^{2}}\,\frac{1}{x^{2}+\lambda_{n}^{2}}=\frac{\pi}{2x}-\frac{1}{1+\frac{1}{\pi p}}\cdot\frac{1}{2x^{2}}+\frac{\pi}{x}\,\frac{1}{\sigma(x)\,e^{2\pi x}-1},\label{formula koshliakov expansion}
\end{equation}
\begin{align}
\sum_{n=1}^{\infty}\frac{p^{2}+\lambda_{n}^{2}}{p\left(p+\frac{1}{\pi}\right)+\lambda_{n}^{2}}\,\frac{1}{\left(\lambda_{n}^{2}+x^{2}\right)^{2}} & =\frac{\pi}{4x^{3}}-\frac{1}{2x^{4}}\cdot\frac{1}{1+\frac{1}{\pi p}}+\nonumber \\
+\frac{\pi^{2}}{x^{2}\left(\sigma(x)\,e^{2\pi x}-1\right)}\cdot & \left[\frac{1}{2\pi x}+\left(1+\frac{p}{\pi\left(p^{2}-x^{2}\right)}\right)\,\frac{\sigma(x)\,e^{2\pi x}}{\sigma(x)\,e^{2\pi x}-1}\right].\label{formula even integer 2 Koshliakov Epstein}
\end{align}
\end{corollary}

\begin{proof}
Taking $s=1$ in formula (\ref{formula valid for all s}) and use
the special value for the Macdonald function,
\begin{equation}
K_{\frac{1}{2}}(x)=\sqrt{\frac{\pi}{2x}}\,e^{-x},\,\,\,\,\,x>0,\label{particular case of bessel!}
\end{equation}
we deduce  
\begin{align}
\sum_{n=1}^{\infty}\frac{p^{2}+\lambda_{n}^{2}}{p\left(p+\frac{1}{\pi}\right)+\lambda_{n}^{2}}\,\frac{1}{\lambda_{n}^{2}+x^{2}} & =\frac{\pi}{2x}-\frac{1}{1+\frac{1}{\pi p}}\cdot\frac{1}{2x^{2}}+\frac{\pi}{x}\,\sum_{m=1}^{\infty}(-1)^{m}\,e^{-2\pi mx}+\nonumber \\
+\frac{\pi}{x}\,\sum_{m=1}^{\infty}e^{2\pi pm}\,\sum_{\ell=1}^{m}\left(\begin{array}{c}
m\\
\ell
\end{array}\right) & (-1)^{m-\ell}\left(4\pi mp\right)^{\ell}\intop_{1}^{\infty}\,e^{-2\pi mt(x+p)}\,\frac{(t-1)^{\ell-1}}{(\ell-1)!}\,dt=\nonumber \\
=\frac{\pi}{2x}-\frac{1}{1+\frac{1}{\pi p}}\cdot\frac{1}{2x^{2}}+\frac{\pi}{x}\,\sum_{m=1}^{\infty} & e^{-2\pi mx}\,\sum_{\ell=0}^{m}\left(\begin{array}{c}
m\\
\ell
\end{array}\right)(-1)^{m-\ell}\left(4\pi mp\right)^{\ell}\left(2\pi m(x+p)\right)^{-\ell}\label{almost in the step}\\
=\frac{\pi}{2x}-\frac{1}{1+\frac{1}{\pi p}}\cdot\frac{1}{2x^{2}}+\frac{\pi}{x}\,\sum_{m=1}^{\infty}\left(\frac{p-x}{p+x}\right)^{m} & e^{-2\pi mx}=\frac{\pi}{2x}-\frac{1}{1+\frac{1}{\pi p}}\cdot\frac{1}{2x^{2}}+\frac{\pi}{x}\,\frac{1}{\sigma(x)\,e^{2\pi x}-1}.\label{proof expression of Koshliakov}
\end{align}
\end{proof}

\begin{remark}
By invoking (\ref{formula valid for all s}) and to the representation
via the Bessel polynomials
\[
K_{n-\frac{1}{2}}(x)=\sqrt{\frac{\pi}{2x}}\,e^{-x}\sum_{k=0}^{n-1}\frac{(n-1+k)!}{k!\,(n-1-k)!\,(2x)^{k}},\,\,\,\,n\in\mathbb{N},
\]
it is possible to compute the values of (\ref{Watson series}) for
every integer $s$. Moreover, we may connect the resulting values 
with identities involving $\zeta_{p}(2n-1)$, reminiscent of Terras'
representations \cite{terras_zeta(2n+1)}  representations (c.f. \cite{kanemitsu_rapidly}).
\end{remark} 

\subsection{Second Analogue of Watson's formula}
We now develop a second analogue of Watson's formula which will be
useful to derive the functional equation for a generalized Epstein
zeta function. Another way of considering an analogue of (\ref{Watson Formula intro}) in Koshliakov's setting is by starting with its right-hand side, this is, to consider
\[
\tilde{\varphi}_{p}(x):=\sum_{n=1}^{\infty}\frac{p^{2}+\lambda_{n}^{2}}{p\left(p+\frac{1}{\pi}\right)+\lambda_{n}^{2}}\,\lambda_{n}^{s-\frac{1}{2}}\,K_{s-\frac{1}{2}}(2\pi\lambda_{n}x),\,\,\,\,x>0.
\]

We will now establish the following result, which acts as a complementary case to Theorem 3.1.
\begin{theorem}
Let $\text{Re}(s)>1$ and $x>0$.
Then the following generalization of Watson's formula takes place
\begin{align}
\sum_{n=1}^{\infty}\frac{p^{2}+\lambda_{n}^{2}}{p\left(p+\frac{1}{\pi}\right)+\lambda_{n}^{2}}\,\lambda_{n}^{s-\frac{1}{2}}\,K_{s-\frac{1}{2}}(2\pi\lambda_{n}x)\nonumber \\
=-\frac{\pi^{\frac{1}{2}-s}x^{\frac{1}{2}-s}}{4}\,\frac{\Gamma\left(s-\frac{1}{2}\right)}{1+\frac{1}{\pi p}}+\frac{\Gamma(s)\pi^{-s}x^{-s-\frac{1}{2}}}{4}+\pi\,\intop_{0}^{\infty}\,\frac{y^{s-\frac{1}{2}}J_{s-\frac{1}{2}}(2\pi xy)}{\sigma\left(y\right)e^{2\pi y}-1}\,dy,\label{second analogue Watson}
\end{align}
where $J_{\nu}(x)$ denotes the Bessel function of the first kind. 

\end{theorem}

\begin{proof}
Let $s\in\mathbb{C}$ and choose $\mu>\text{Re}(s)+\frac{1}{2}$.
Using the integral representation (\ref{Mellin barnes Macdonald}), arguing by absolute convergence of the Dirichlet series $\zeta_{p}(s)$ for $\text{Re}(s)>1$ and taking into account our choice of $\mu$, we have 
\begin{align}
\sum_{n=1}^{\infty}\frac{p^{2}+\lambda_{n}^{2}}{p\left(p+\frac{1}{\pi}\right)+\lambda_{n}^{2}}\,\lambda_{n}^{s-\frac{1}{2}}K_{s-\frac{1}{2}}(2\pi\lambda_{n}x) & =\frac{1}{8\pi i}\,\intop_{\mu-i\infty}^{\mu+i\infty}\Gamma\left(\frac{w+\frac{1}{2}-s}{2}\right)\Gamma\left(\frac{w+s-\frac{1}{2}}{2}\right)\,\zeta_{p}\left(w+\frac{1}{2}-s\right)\,(\pi x)^{-w}\,dw\nonumber \\
 & =\frac{(\pi x)^{\frac{1}{2}-s}}{4\pi i}\,\intop_{\sigma-i\infty}^{\sigma+i\infty}\Gamma\left(z\right)\Gamma\left(s+z-\frac{1}{2}\right)\,\zeta_{p}\left(2z\right)\,(\pi x)^{-2z}\,dz\nonumber \\
 & =\frac{(\pi x)^{\frac{1}{2}-s}}{2}\,\frac{1}{2\pi i}\,\intop_{\sigma-i\infty}^{\sigma+i\infty}\Gamma\left(z\right)\Gamma\left(s+z-\frac{1}{2}\right)\,\zeta_{p}\left(2z\right)\,(\pi x)^{-2z}\,dz.\label{starting with Mellin representation}
\end{align}
where $\sigma>\frac{1}{2}$. We will now assume that $\text{Re}(s)>\sigma>\frac{1}{2}$ and change the line of integration in (\ref{starting with Mellin representation})
to $\text{Re}(z)=\frac{1}{2}-\sigma$: doing so we pass by two poles
located at $z=\frac{1}{2}$ and $z=0$, since $\text{Re}(s)>\sigma$
by hypothesis. An application of Cauchy's residue Theorem, combined with the convex estimates (\ref{Koshliakov Phragmen Analogue}), gives
\begin{align}
\frac{1}{2\pi i}\,\intop_{\sigma-i\infty}^{\sigma+i\infty}\Gamma\left(z\right)\Gamma\left(s+z-\frac{1}{2}\right)\,\zeta_{p}\left(2z\right)\,(\pi x)^{-2z}\,dz & =\frac{1}{2\pi i}\intop_{\frac{1}{2}-\sigma-i\infty}^{\frac{1}{2}-\sigma+i\infty}\Gamma\left(z\right)\Gamma\left(s+z-\frac{1}{2}\right)\,\zeta_{p}\left(2z\right)\,(\pi x)^{-2z}\,dz\nonumber \\
 & -\frac{1}{2}\cdot\frac{\Gamma\left(s-\frac{1}{2}\right)}{1+\frac{1}{\pi p}}+\frac{\Gamma\left(s\right)}{2x\sqrt{\pi}}.\label{using thing at the end}
\end{align}

Finally, we may use the functional equation for the Koshliakov zeta
function (\ref{functional equation Kosh}) to deduce that
\begin{align}
\frac{1}{2\pi i}\intop_{\frac{1}{2}-\sigma-i\infty}^{\frac{1}{2}-\sigma+i\infty}\Gamma\left(z\right)\Gamma\left(s+z-\frac{1}{2}\right)\,\zeta_{p}\left(2z\right)\,(\pi x)^{-2z}\,dz=\frac{1}{2\pi i}\,\frac{x^{-1}}{\sqrt{\pi}}\,\intop_{\sigma-i\infty}^{\sigma+i\infty}x^{2z}\,\Gamma(z)\,\eta_{p}(2z)\,\Gamma(s-z)\,dz\nonumber \\
=\frac{x^{-1}}{\sqrt{\pi}}\sum_{m=1}^{\infty}\frac{1}{2\pi i}\,\intop_{\sigma-i\infty}^{\sigma+i\infty}\Gamma(z)\,\Gamma(s-z)\,(2z,2\pi pm)_{m}\,\left(\frac{x}{m}\right)^{2z}\,dz=\frac{x^{-1}}{\sqrt{\pi}}\,\sum_{m=1}^{\infty}I_{m,p}^{\star}(s,x).\label{after using the functional equation}
\end{align}

As in the proof of the first analogue of Watson's formula (see Theorem 3.1. above), if we use
the representation (\ref{representation Koshliakov given in poage 24}),
we see that
\begin{equation}
I_{m,p}^{\star}(s,x):=\frac{1}{2\pi i}\,\intop_{\sigma-i\infty}^{\sigma+i\infty}\Gamma(z)\,\Gamma(s-z)\,(2z,2\pi pm)_{m}\,\left(\frac{x}{m}\right)^{2z}\,dz=I_{m,p}^{\star(1)}(s,x)+I_{m,p}^{\star(2)}(s,x),\label{deifnition Im,p}
\end{equation}
where, due to the hypothesis $\text{Re}(s)>\sigma>\frac{1}{2}$ and the integral representation (\ref{Beta integral beginning}),
\begin{equation}
I_{m,p}^{\star(1)}(s,x)=\frac{(-1)^{m}}{2\pi i}\intop_{\sigma-i\infty}^{\sigma+i\infty}\Gamma(z)\,\Gamma(s-z)\,\left(\frac{x}{m}\right)^{2z}dz=\frac{(-1)^{m}\Gamma(s)x^{2s}}{\left(x^{2}+m^{2}\right)^{s}}.\label{Im,p(1) in the second analogue}
\end{equation}

The second integral,
$I_{m,p}^{\star(2)}(s,x)$, has a similar representation, given as 
\begin{align}
I_{m,p}^{\star(2)}(s,x) & =\frac{1}{2\pi i}\,\intop_{\sigma-i\infty}^{\sigma+i\infty}e^{2\pi pm}\,\sum_{\ell=1}^{m}\left(\begin{array}{c}
m\\
\ell
\end{array}\right)\left(4\pi p\right)^{\ell}(-1)^{m-\ell}\,\intop_{m}^{\infty}t^{-2z}e^{-2\pi p\,t}\,\frac{(t-m)^{\ell-1}}{(\ell-1)!}\,\Gamma(z)\,\Gamma(s-z)\,x^{2z}\,dt\,dz\nonumber \\
 & =e^{2\pi pm}\sum_{\ell=1}^{m}\left(\begin{array}{c}
m\\
\ell
\end{array}\right)\left(4\pi p\right)^{\ell}(-1)^{m-\ell}\,\intop_{m}^{\infty}e^{-2\pi p\,t}\,\frac{(t-m)^{\ell-1}}{(\ell-1)!}\,\frac{1}{2\pi i}\,\intop_{\sigma-i\infty}^{\sigma+i\infty}\Gamma(z)\,\Gamma(s-z)\,t^{-2z}\,x^{2z}dz\,dt\nonumber \\
 & =e^{2\pi pm}\sum_{\ell=1}^{m}\left(\begin{array}{c}
m\\
\ell
\end{array}\right)\left(4\pi p\right)^{\ell}(-1)^{m-\ell}\,\intop_{m}^{\infty}e^{-2\pi p\,t}\,\frac{(t-m)^{\ell-1}}{(\ell-1)!}\,\frac{\Gamma(s)x^{2s}}{\left(x^{2}+t^{2}\right)^{s}}\,dt,\label{evaluation Imp2}
\end{align}
where all the steps are justified via absolute convergence. Now, we use
the integral representation [\cite{ryzhik}, p. 702, eq. (6.623.1){]}
\begin{equation}
\frac{1}{\left(x^{2}+t^{2}\right)^{s}}=\frac{\sqrt{\pi}\,2^{\frac{1}{2}-s}}{\Gamma(s)x^{s-\frac{1}{2}}}\intop_{0}^{\infty}y^{s-\frac{1}{2}}J_{s-\frac{1}{2}}(xy)\,e^{-ty}\,dy,\label{Integral representation Bessel laplace}
\end{equation}
valid for every $x,t>0$ and $\text{Re}(s)>0$. We find an alternative
representation for $I_{m,p}^{\star(1)}(s,x)$ in the following form
\begin{equation}
I_{m,p}^{\star(1)}(s,x)=(-1)^{m}\sqrt{\pi}2^{\frac{1}{2}-s}x^{s+\frac{1}{2}}\,\intop_{0}^{\infty}y^{s-\frac{1}{2}}J_{s-\frac{1}{2}}(xy)\,e^{-my}dy\label{alternative representaion Im,p (1) star}
\end{equation}

It is simple to argue via absolute convergence and the hypothesis $\text{Re}(s)>\frac{1}{2}$ that the orders of
integration with respect to $t$ and $y$ can be reversed
and this gives the representation for $I_{m,p}^{\star(2)}(s,x)$,
\begin{align}
I_{m,p}^{\star(2)}(s,x) & =\sqrt{\pi}2^{\frac{1}{2}-s}x^{s+\frac{1}{2}}\,e^{2\pi pm}\sum_{\ell=1}^{m}\left(\begin{array}{c}
m\\
\ell
\end{array}\right)\left(4\pi p\right)^{\ell}(-1)^{m-\ell}\intop_{0}^{\infty}y^{s-\frac{1}{2}}J_{s-\frac{1}{2}}(xy)\,\intop_{m}^{\infty}e^{-2\pi pt}\,\frac{(t-m)^{\ell-1}}{(\ell-1)!}\,e^{-ty}\,dt\,dy\nonumber \\
 & =\sqrt{\pi}2^{\frac{1}{2}-s}x^{s+\frac{1}{2}}\intop_{0}^{\infty}y^{s-\frac{1}{2}}J_{s-\frac{1}{2}}(xy)\,\sum_{\ell=1}^{m}\left(\begin{array}{c}
m\\
\ell
\end{array}\right)\left(4\pi p\right)^{\ell}(-1)^{m-\ell}\,\left(2\pi p+y\right)^{-\ell}e^{-m\,y}\,dy\nonumber \\
 & =\sqrt{\pi}2^{\frac{1}{2}-s}x^{s+\frac{1}{2}}\intop_{0}^{\infty}y^{s-\frac{1}{2}}J_{s-\frac{1}{2}}(xy)\,\left\{ \left(\frac{4\pi p}{2\pi p+y}-1\right)^{m}-(-1)^{m}\right\} e^{-my}\,dy\nonumber \\
 & =\sqrt{\pi}2^{\frac{1}{2}-s}x^{s+\frac{1}{2}}\intop_{0}^{\infty}y^{s-\frac{1}{2}}J_{s-\frac{1}{2}}(xy)\,\left\{ \left(\frac{2\pi p-y}{2\pi p+y}\right)^{m}-(-1)^{m}\right\} e^{-my}\,dy.\label{Im, p (2) second analogue Koshliakov}
\end{align}

Finally, a combination of (\ref{Im,p(1) in the second analogue}) and (\ref{Im, p (2) second analogue Koshliakov})
yields
\begin{equation}
I_{m,p}^{\star}(s,x)=\sqrt{\pi}2^{\frac{1}{2}-s}x^{s+\frac{1}{2}}\intop_{0}^{\infty}y^{s-\frac{1}{2}}J_{s-\frac{1}{2}}(xy)\,\left(\frac{2\pi p-y}{2\pi p+y}\right)^{m}\,e^{-my}dy\label{one of the final rperesentarions almost similar lips}
\end{equation}

Summing over $m$, using the well-known uniform bound $|J_\nu(x)|\leq C_{\nu}/\sqrt{x}$, $x>0$, and the hypothesis that $\text{Re}(s)>1$, we 
find from (\ref{one of the final rperesentarions almost similar lips}) and (\ref{power series def}) that
\begin{equation}
\sum_{m=1}^{\infty}I_{m,p}^{\star}(s,x)=\sqrt{\pi}2^{\frac{1}{2}-s}x^{s+\frac{1}{2}}\intop_{0}^{\infty}\,\frac{y^{s-\frac{1}{2}}J_{s-\frac{1}{2}}(xy)}{\sigma\left(\frac{y}{2\pi}\right)e^{y}-1}\,dy.\label{evaluation of sum of Imp}
\end{equation}

Returning to (\ref{after using the functional equation}) and (\ref{starting with Mellin representation})
we obtain immediately the second analogue of Watson's formula (\ref{second analogue Watson}).

\end{proof}

As before, by letting $p\rightarrow\infty$ or $p\rightarrow 0^{+}$ in (\ref{second analogue Watson}), we shall obtain Watson's formula (\ref{Watson Formula intro}) and an analogue of it. 
\begin{corollary}
For every $x>0$ and $\text{Re}(s)>\frac{1}{2}$, the classical Watson's
formula (\ref{Watson Formula intro}) holds. Moreover, an analogue
of (\ref{Watson for odd sequence}) takes place
\begin{equation}
\sum_{n=1}^{\infty}\frac{(-1)^{n}}{(x^{2}+n^{2})^{s}}=-\frac{x^{-2s}}{2}+\frac{2^{\frac{3}{2}-s}\pi^{s}x^{\frac{1}{2}-s}}{\Gamma(s)}\sum_{n=1}^{\infty}\left(2n-1\right)^{s-\frac{1}{2}}\,K_{s-\frac{1}{2}}\left(\pi(2n-1)x\right).\label{yet another analogue Watson from second}
\end{equation}
\end{corollary}

\begin{proof}
Assume first that $\text{Re}(s)>1$: letting $p\rightarrow\infty$ in (\ref{second analogue Watson}),
we find that
\begin{align*}
\sum_{n=1}^{\infty}n^{s-\frac{1}{2}}\,K_{s-\frac{1}{2}}(2\pi nx) & =-\frac{\pi^{\frac{1}{2}-s}x^{\frac{1}{2}-s}}{4}\,\Gamma\left(s-\frac{1}{2}\right)+\frac{\Gamma(s)\pi^{-s}x^{-s-\frac{1}{2}}}{4}+\pi\intop_{0}^{\infty}\,\frac{y^{s-\frac{1}{2}}J_{s-\frac{1}{2}}(2\pi xy)}{e^{2\pi y}-1}\,dy\\
 & =-\frac{\pi^{\frac{1}{2}-s}x^{\frac{1}{2}-s}}{4}\,\Gamma\left(s-\frac{1}{2}\right)+\frac{\Gamma(s)\pi^{-s}x^{-s-\frac{1}{2}}}{4}+\pi\sum_{n=1}^{\infty}\intop_{0}^{\infty}\,y^{s-\frac{1}{2}}J_{s-\frac{1}{2}}(2\pi xy)\,e^{-2\pi ny}\,dy,
\end{align*}
where the interchange of the orders of the integral and the series
comes from the well-known bound for $J_{\nu}(x)$, $|J_{\nu}(x)|\leq C_{\nu}/\sqrt{x}$, $x>0$,
which gives
\begin{equation}
\sum_{n=1}^{\infty}\intop_{0}^{\infty}\,y^{\text{Re}(s)-\frac{1}{2}}|J_{s-\frac{1}{2}}(2\pi xy)|\,e^{-2\pi ny}\,dy\leq C_{s}\,\sum_{n=1}^{\infty}\intop_{0}^{\infty}\,y^{\text{Re}(s)-1}\,e^{-2\pi ny}\,dy\leq D_{s}\,\sum_{n=1}^{\infty}\frac{1}{n^{\text{Re}(s)}}<\infty,\label{set of bounds justification}
\end{equation}
by the hypothesis $\text{Re}(s)>1$.

Therefore, using (\ref{Integral representation Bessel laplace}) with
$x$ replaced by $2\pi x$ and $t$ by $2\pi n$, we deduce immediately
Watson's formula (\ref{Watson Formula intro}) when $\text{Re}(s)>1$.
Since both sides of (\ref{Watson Formula intro}) represent analytic
functions on the half-plane $\text{Re}(s)>\frac{1}{2}$ (note the
uniform and absolute convergence of the series on the left in this
region), Watson's formula can be proved for every $\text{Re}(s)>\frac{1}{2}$
by the principle of analytic continuation.
\\

Letting $p\rightarrow0^{+}$ and using the same justification as in (\ref{set of bounds justification}),
we find that (\ref{second analogue Watson}) implies
\begin{equation*}
2^{\frac{1}{2}-s}\sum_{n=1}^{\infty}\left(2n-1\right)^{s-\frac{1}{2}}\,K_{s-\frac{1}{2}}\left(\pi(2n-1)x\right)
=\frac{\Gamma(s)\pi^{-s}x^{-s-\frac{1}{2}}}{4}+\pi\sum_{n=1}^{\infty}(-1)^{n}\,\intop_{0}^{\infty}\,y^{s-\frac{1}{2}}J_{s-\frac{1}{2}}(2\pi xy)\,e^{-2\pi ny}\,dy,\,\,\,\,\text{Re}(s)>1.
\end{equation*}

Thus, the use of the integral representation (\ref{Integral representation Bessel laplace})
also proves the desired formula (\ref{yet another analogue Watson from second})
for $\text{Re}(s)>1$. Finally, the extension of (\ref{yet another analogue Watson from second})
to the region $\text{Re}(s)>\frac{1}{2}$ follows from analytic continuation.

\end{proof}

\bigskip{}

By using an integration by parts on the integral (\ref{one of the final rperesentarions almost similar lips}), it can be actually proved that Watson's formula (\ref{second analogue Watson}) holds for $\text{Re}(s)>\frac{1}{2}$ (see (\ref{justifying the passage}) below for a justification of a particular case of this). Like in Corollary 3.2., we now establish the analytic continuation
of the second analogue of Watson's formula (\ref{second analogue Watson}).

\begin{corollary}
Let $N>0$ be an integer. For every $x>0$ and every $s$ satisfying the condition, $\text{Re}(s)>-N-\frac{1}{2}$, one has the identity
\begin{align}
\sum_{n=1}^{\infty}\frac{p^{2}+\lambda_{n}^{2}}{p\left(p+\frac{1}{\pi}\right)+\lambda_{n}^{2}}\,\lambda_{n}^{s-\frac{1}{2}}\,K_{s-\frac{1}{2}}(2\pi\lambda_{n}x) & =-\frac{\pi^{\frac{1}{2}-s}x^{\frac{1}{2}-s}}{4}\,\frac{\Gamma\left(s-\frac{1}{2}\right)}{1+\frac{1}{\pi p}}+\frac{\Gamma(s)\pi^{-s}x^{-s-\frac{1}{2}}}{4}\nonumber \\
+\frac{x^{s-\frac{1}{2}}}{2\pi^{s}}\sum_{k=0}^{N}\frac{\left(-1\right)^{k}x^{2k}}{k!}\,\Gamma(s+k)\,\eta_{p}(2s+2k)+\pi\intop_{0}^{\infty} & \frac{y^{s-\frac{1}{2}}}{\sigma\left(y\right)e^{2\pi y}-1}\left\{ J_{s-\frac{1}{2}}(2\pi xy)-\sum_{k=0}^{N}\frac{\left(-1\right)^{k}\left(\pi xy\right)^{s+2k-\frac{1}{2}}}{k!\Gamma\left(s+k+\frac{1}{2}\right)}\right\} \,dy.\label{Analytic continuation for every s arbitrary N}
\end{align}   
\end{corollary}

\begin{proof}
For every $x>0$ and $N\in\mathbb{N}$, let us consider the integral
\[
\mathcal{J}_{N}(x,s):=\intop_{0}^{\infty}\,\frac{y^{s-\frac{1}{2}}\left\{ J_{s-\frac{1}{2}}(2\pi xy)-\sum_{k=0}^{N}\frac{\left(-1\right)^{k}\left(\pi xy\right)^{s+2k-\frac{1}{2}}}{k!\Gamma\left(s+k+\frac{1}{2}\right)}\right\} }{\sigma\left(y\right)e^{2\pi y}-1}\,dy
\]
which converges absolutely and uniformly for every $\text{Re}(s)>-N-\frac{1}{2}$.
By the power series of the Bessel function of the first kind $J_{\nu}(z)$
{[}\cite{NIST}, p. 217, eq. 10.2.2.{]}, for every $\text{Re}(s)>\frac{1}{2}$
we know that the following equality must take place 
\begin{align}
\intop_{0}^{\infty}\,\frac{y^{s-\frac{1}{2}}J_{s-\frac{1}{2}}(2\pi xy)}{\sigma\left(y\right)e^{2\pi y}-1}\,dy & =\mathcal{J}_{N}(x,s)+\sum_{k=0}^{N}\frac{\left(-1\right)^{k}\left(\pi x\right)^{s+2k-\frac{1}{2}}}{k!\Gamma\left(s+k+\frac{1}{2}\right)}\,\intop_{0}^{\infty}\frac{y^{2s+2k-1}}{\sigma\left(y\right)e^{2\pi y}-1}dy\nonumber \\
 & =\mathcal{J}_{N}(x,s)+\frac{x^{s-\frac{1}{2}}}{2\pi^{s+1}}\sum_{k=0}^{N}\frac{\left(-1\right)^{k}x^{2k}}{k!}\,\Gamma(s+k)\,\eta_{p}(2s+2k)\label{combination final}
\end{align}
where we have used the integral representation (\ref{representation as starting point}),
together with the functional equation for $\zeta_{p}(s)$ as well
as Legendre's duplication formula for the Gamma function. Taking now $\text{Re}(s)>1$ in 
(\ref{combination final}) and combining it with (\ref{second analogue Watson}), we may now invoke the principle of analytic continuation to deduce (\ref{Analytic continuation for every s arbitrary N}),
completing the proof.

\end{proof}

\bigskip{}

From the previous result, we can obtain a nice generalization of Watson's
identity (\ref{Analytic continuation Watson formula p infinity case}),
which is analogous to (\ref{meromorphic expansion continuation Kosh}).

\begin{corollary}
For every $p\in\mathbb{R}_{+}$ and $x>0$, the following identity
takes place
\begin{equation}
\sum_{n=1}^{\infty}\frac{p^{2}+\lambda_{n}^{2}}{p\left(p+\frac{1}{\pi}\right)+\lambda_{n}^{2}}\,K_{0}(2\pi\lambda_{n}x)=\frac{1}{4x}+\frac{C_{p}^{(2)}}{2}-\frac{e^{2\pi p}Q_{2\pi p}(0)}{1+\frac{1}{\pi p}}+\frac{\log\left(\frac{x}{2}\right)}{2\left(1+\frac{1}{\pi p}\right)}+\pi\,\intop_{0}^{\infty}\,\frac{J_{0}(2\pi xy)-1}{\sigma\left(y\right)e^{2\pi y}-1}\,dy.\label{Dont know which one is true}
\end{equation}
\end{corollary}

\begin{proof}
We prove (\ref{Dont know which one is true}) by taking $N=1$ and letting $s\rightarrow\frac{1}{2}$ in (\ref{Analytic continuation for every s arbitrary N}). The conclusion of the proof follows after invoking the
Laurent expansion {[}\cite{KOSHLIAKOV}, p. 49, eq. (53){]}, 
\begin{equation}
\eta_{p}(s)=\frac{1}{1+\frac{1}{\pi p}}\,\frac{1}{s-1}+C_{p}^{(2)}-\frac{2e^{2\pi p}Q_{2\pi p}(0)}{1+\frac{1}{\pi p}}+O\left(s-1\right),\label{Meromorphic expansion around zero Koshliakov sense}
\end{equation}
together with (\ref{meromorphic x}), (\ref{meromorphic Gamma}),
as well as
\begin{equation}
\Gamma\left(s\right)=\sqrt{\pi}-\sqrt{\pi}\left(2\log(2)+\gamma\right)\left(s-\frac{1}{2}\right)+O\left(s-\frac{1}{2}\right)^{2}.\label{another Laurent expansion}
\end{equation}
\end{proof}

\begin{corollary}
Watson's formula (\ref{Analytic continuation Watson formula p infinity case})
holds. Furthermore, we have the identity
\begin{equation}
\sum_{n=1}^{\infty}\frac{(-1)^{n}}{\sqrt{x^{2}+n^{2}}}=-\frac{1}{2x}+2\sum_{n=1}^{\infty}K_{0}\left(\pi(2n-1)x\right),\,\,\,\,\,x>0,\label{Watson formula limit case p zero}
\end{equation}
which is analogous to (\ref{analytic continuation Watson p zero particular case}).
\end{corollary}

\begin{proof}
Taking $p\rightarrow\infty$ on (\ref{Dont know which one is true}), 
we obtain
\begin{align*}
\sum_{n=1}^{\infty}K_{0}(2\pi nx) & =\frac{1}{4x}+\frac{\gamma}{2}+\frac{\log\left(\frac{x}{2}\right)}{2}+\pi\,\intop_{0}^{\infty}\,\frac{J_{0}(2\pi xy)-1}{e^{2\pi y}-1}\,dy=\frac{1}{4x}+\frac{\gamma}{2}+\frac{\log\left(\frac{x}{2}\right)}{2}+\pi\,\sum_{n=1}^{\infty}\intop_{0}^{\infty}\,\left(J_{0}(2\pi xy)-1\right)e^{-2\pi ny}\,dy\\
 & =\frac{1}{4x}+\frac{\gamma}{2}+\frac{\log\left(\frac{x}{2}\right)}{2}+\frac{1}{2}\,\sum_{n=1}^{\infty}\left\{ \frac{1}{\sqrt{n^{2}+x^{2}}}-\frac{1}{n}\right\} ,
\end{align*}
where the last step comes from (\ref{Integral representation Bessel laplace}). In the second step, the interchange of the orders of the integral
and the geometric series can be justified by absolute convergence: indeed, it follows from an integration by parts that
\begin{equation}
\sum_{n=1}^{\infty}\intop_{0}^{\infty}\,\left(J_{0}(2\pi xy)-1\right)e^{-2\pi ny}\,dy=-\,\sum_{n=1}^{\infty}\frac{x}{n}\,\intop_{0}^{\infty}e^{-2\pi ny}\,J_{1}(2\pi xy)\,dy,\label{integration by parts}
\end{equation}
so that, by using once more the uniform bound $|J_{\nu}(x)|\leq C_{\nu}/\sqrt{x}$, $x>0$, 
we have for absolute constants $D$ and $D^{\text{\ensuremath{\prime}}}$, 
\begin{equation}
\sum_{n=1}^{\infty}\frac{1}{n}\,\intop_{0}^{\infty}e^{-2\pi ny}\,|J_{1}(2\pi xy)|\,dy\leq\frac{D}{\sqrt{2\pi x}}\,\sum_{n=1}^{\infty}\frac{1}{n}\,\intop_{0}^{\infty}\frac{e^{-2\pi ny}}{\sqrt{y}}\,dy\leq\frac{D^{\prime}}{\sqrt{x}}\sum_{n=1}^{\infty}\frac{1}{n^{3/2}},\label{justifying the passage}
\end{equation}
justifying the passage. To prove (\ref{Watson formula limit case p zero}),
we let $p\rightarrow0^{+}$ on (\ref{Dont know which one is true})
and use the same kind of justifications as in (\ref{justifying the passage}) in order to get
\begin{align*}
\sum_{n=1}^{\infty}K_{0}\left(\pi\left(2n-1\right)x\right) & =\frac{1}{4x}-\frac{\log(2)}{2}+\pi\,\sum_{n=1}^{\infty}(-1)^{n}\intop_{0}^{\infty}\left(J_{0}(2\pi xy)-1\right)\,e^{-2\pi ny}dy\\
=\frac{1}{4x}-\frac{\log(2)}{2} & +\frac{1}{2}\,\sum_{n=1}^{\infty}\left\{ \frac{(-1)^{n}}{\sqrt{n^{2}+x^{2}}}-\frac{(-1)^{n}}{n}\right\} =\frac{1}{4x}+\frac{1}{2}\sum_{n=1}^{\infty}\frac{(-1)^{n}}{\sqrt{n^{2}+x^{2}}}.
\end{align*}
\end{proof}

\bigskip{}

Like in Corollary 3.5., it is now possible to derive as a particular
case from our analogue of Watson's formula (\ref{second analogue Watson}) a formula obtained by Koshliakov himself {[}\cite{KOSHLIAKOV}, p. 44, eq.
(36){]}, which truly complements (\ref{formula koshliakov expansion}).

\begin{corollary}
Let $\sigma_{p}(z)$ be defined by
\begin{equation}
\sigma_{p}(z)=\sum_{n=1}^{\infty}\frac{p^{2}+\lambda_{n}^{2}}{p\left(p+\frac{1}{\pi}\right)+\lambda_{n}^{2}}\,e^{-\lambda_{n}z},\,\,\,\,\,\,\text{Re}(z)>0,\label{sigma p Koshliakov setting}
\end{equation}
(c.f. (\ref{sigma p definition on Koshliakov}) above). Then the following
formula holds
\begin{equation}
\sigma_{p}(2\pi x)=-\frac{1}{2}\,\frac{1}{1+\frac{1}{\pi p}}+\frac{1}{2\pi x}+2\,\intop_{0}^{\infty}\frac{\sin(xy)}{\sigma(y)\,e^{2\pi y}-1}\,dy.\label{analogue abel plana koshliakov sense}
\end{equation}
\end{corollary}

\begin{proof}
Take $s=1$ in our second analogue of Watson's formula (\ref{second analogue Watson}).
By appealing to the well-known formula (\ref{particular case of bessel!})
and to $J_{1/2}(x)=\sqrt{\frac{2}{\pi x}}\,\sin(x)$, we see that 
\begin{align*}
\frac{1}{2\sqrt{x}}\sigma_{p}(2\pi x)&=\frac{1}{2\sqrt{x}}\,\sum_{n=1}^{\infty}\frac{p^{2}+\lambda_{n}^{2}}{p\left(p+\frac{1}{\pi}\right)+\lambda_{n}^{2}}\,e^{-2\pi\lambda_{n}x}=-\frac{1}{4\sqrt{x}}\,\frac{1}{1+\frac{1}{\pi p}}+\frac{1}{4\pi x^{3/2}}+\pi\,\intop_{0}^{\infty}\,\frac{y^{1/2}J_{1/2}(2\pi xy)}{\sigma\left(y\right)e^{2\pi y}-1}\,dy\\
 &=-\frac{1}{4\sqrt{x}}\,\frac{1}{1+\frac{1}{\pi p}}+\frac{1}{4\pi x^{3/2}}+\frac{1}{\sqrt{x}}\intop_{0}^{\infty}\frac{\sin(2\pi xy)}{\sigma(y)\,e^{2\pi y}-1}\,dy,
\end{align*}
which is equivalent to (\ref{analogue abel plana koshliakov sense}).
\end{proof}

\subsection{Another proof and a generalization of Watson's formula}

In this subsection we give an alternative proof and a new expression respectively for (\ref{formula in a region compact to be extended}) and (\ref{second analogue Watson}). We start with a new proof of (\ref{formula in a region compact to be extended}). 

\begin{proof}[2\ts{nd} Proof of (\ref{formula in a region compact to be extended})]
We evaluate
\begin{equation}
\varphi_{p}(s,x):=\sum_{n=1}^{\infty}\frac{p^{2}+\lambda_{n}^{2}}{p\left(p+\frac{1}{\pi}\right)+\lambda_{n}^{2}}\,\frac{1}{\left(\lambda_{n}^{2}+x^{2}\right)^{s}},\,\,\,\,\text{Re}(s)>\frac{1}{2},\,\,\,\,x>0,\label{varphi p definition once again}
\end{equation}
by appealing to the representation (\ref{Integral representation Bessel laplace}). By absolute convergence, we have 
\begin{equation}
\sum_{n=1}^{\infty}\frac{p^{2}+\lambda_{n}^{2}}{p\left(p+\frac{1}{\pi}\right)+\lambda_{n}^{2}}\,\frac{1}{\left(\lambda_{n}^{2}+x^{2}\right)^{s}}=\frac{\sqrt{\pi}\,2^{\frac{1}{2}-s}}{\Gamma(s)x^{s-\frac{1}{2}}}\intop_{0}^{\infty}y^{s-\frac{1}{2}}J_{s-\frac{1}{2}}(xy)\,\sigma_{p}(y)\,dy,\,\,\,\,\text{Re}(s)>\frac{1}{2}.\label{intermediate expression}
\end{equation}
The justification of the previous formula is similar to (\ref{justifying the passage})
and we may use an integration by parts and Fubini's Theorem to complete it. 
Invoking the Basset type representation for $J_{\nu}(z)$ {[}\cite{NIST},
p. 224, eq. (10.9.12){]},
\begin{equation}
J_{\nu}(x)=\frac{2^{\nu+1}x^{-\nu}}{\sqrt{\pi}\Gamma\left(\frac{1}{2}-\nu\right)}\,\intop_{1}^{\infty}\frac{\sin(xt)}{(t^{2}-1)^{\nu+\frac{1}{2}}}\,dt,\,\,\,\,\,|\text{Re}(\nu)|<\frac{1}{2},\,x>0,\label{Mehler Basset type for Bessel 1st kind}
\end{equation}
we find that (\ref{intermediate expression}) implies
\begin{align}
\sum_{n=1}^{\infty}\frac{p^{2}+\lambda_{n}^{2}}{p\left(p+\frac{1}{\pi}\right)+\lambda_{n}^{2}}\,\frac{1}{\left(\lambda_{n}^{2}+x^{2}\right)^{s}} & =\frac{2\sin(\pi s)}{\pi x^{2s-1}}\,\intop_{0}^{\infty}\intop_{1}^{\infty}\sigma_{p}(y)\,\sin(xyt)\,\frac{dt\,dy}{(t^{2}-1)^{s}}\nonumber \\
 & =\frac{2\sin(\pi s)}{\pi x^{2s-1}}\,\intop_{1}^{\infty}\intop_{0}^{\infty}\sigma_{p}(y)\,\sin(xyt)\,dy\,\frac{dt}{(t^{2}-1)^{s}},\,\,\,\,\,\frac{1}{2}<\text{Re}(s)<1\label{integral representation on the way}
\end{align}
with the last step being justified by Fubini's theorem and the hypothesis
that $\frac{1}{2}<\text{Re}(s)<1$. 

Now, by {[}\cite{KOSHLIAKOV}, p. 45, eq. (40){]}\footnote{this formula is a corollary of Koshliakov's generalizations of the
Abel-Plana formula. See \cite{Ramanujan_meets} for even more generalizations.
This formula is in fact equivalent to (\ref{formula koshliakov expansion}),
also due to Koshliakov {[}\cite{KOSHLIAKOV}, p. 34, eq. (76){]}.}, we know that
\begin{equation}
\frac{1}{\sigma(x)\,e^{2\pi x}-1}=-\frac{1}{2}+\frac{1}{2\pi x}\cdot\frac{1}{1+\frac{1}{\pi p}}+\frac{1}{\pi}\,\intop_{0}^{\infty}\sin(xt)\,\sigma_{p}(t)\,dt.
\end{equation}

Using this formula, we are able to evaluate the first integral on
(\ref{integral representation on the way}) and we can derive that
\[
\sum_{n=1}^{\infty}\frac{p^{2}+\lambda_{n}^{2}}{p\left(p+\frac{1}{\pi}\right)+\lambda_{n}^{2}}\,\frac{1}{\left(\lambda_{n}^{2}+x^{2}\right)^{s}}=\frac{2\sin(\pi s)}{x^{2s-1}}\,\intop_{1}^{\infty}\left\{ \frac{1}{\sigma(xt)e^{2\pi xt}-1}+\frac{1}{2}-\frac{1}{2\pi xt}\cdot\frac{1}{1+\frac{1}{\pi p}}\right\} \,\frac{dt}{(t^{2}-1)^{s}}.
\]

By using some elementary relations for Euler's beta function,
\begin{equation}
\intop_{1}^{\infty}\frac{1}{(t^{2}-1)^{s}}\,\frac{dt}{t}=\frac{\pi}{2\sin(\pi s)},\,\,\,\,\,\,0<\text{Re}(s)<1,\label{First evaluation beta type}
\end{equation}
\begin{equation}
\intop_{1}^{\infty}\frac{dt}{(t^{2}-1)^{s}}=\frac{\Gamma\left(1-s\right)\Gamma\left(s-\frac{1}{2}\right)}{2\sqrt{\pi}},\,\,\,\frac{1}{2}<\text{Re}(s)<1,\label{second evaluation beta type}
\end{equation}
we find, for $\frac{1}{2}<\text{Re}(s)<1$ and $x>0$,
\begin{equation}
\sum_{n=1}^{\infty}\frac{p^{2}+\lambda_{n}^{2}}{p\left(p+\frac{1}{\pi}\right)+\lambda_{n}^{2}}\,\frac{1}{\left(\lambda_{n}^{2}+x^{2}\right)^{s}}=\frac{\sqrt{\pi}x^{1-2s}}{2\Gamma(s)}\Gamma\left(s-\frac{1}{2}\right)-\frac{1}{2}\,\frac{x^{-2s}}{1+\frac{1}{\pi p}}+2\sin(\pi s)x^{1-2s}\,\intop_{1}^{\infty}\frac{(t^{2}-1)^{-s}}{\sigma(xt)e^{2\pi xt}-1}dt,\label{final expression tird proof first analogue}
\end{equation}
from which (\ref{formula in a region compact to be extended}) can
now be easily derived, after taking a simple change of variable in
the integral on the right side of (\ref{final expression tird proof first analogue}).
\end{proof}
\bigskip{}
\bigskip{}
\bigskip{}

There is still a plethora of different scenarios where we may generalize
Watson's formula (\ref{Watson Formula intro}). Here we give yet another generalization that is connected to Koshliakov's first analogue of Poisson's summation formula [\cite{KOSHLIAKOV}, p. 58, eq. (V)].
\begin{theorem}
Let $\text{Re}(s)>\frac{1}{2}$ and $x>0$.
Then the following generalization of Watson's formula (\ref{Watson Formula intro})
holds
\begin{align}
\frac{2^{1-2s}x^{s-\frac{1}{2}}\pi^{-s}}{\Gamma(s)}\sum_{m=1}^{\infty}\intop_{0}^{\infty}\intop_{0}^{1}y^{2s-1}e^{-xy}\left(1-u^{2}\right)^{s-1}\cos\left(myu+2m\,\arctan\left(\frac{yu}{2\pi p}\right)\right)\,du\,dy\nonumber \\
=\sum_{n=1}^{\infty}\frac{p^{2}+\lambda_{n}^{2}}{p\left(p+\frac{1}{\pi}\right)+\lambda_{n}^{2}}\,\lambda_{n}^{s-\frac{1}{2}}K_{s-\frac{1}{2}}(2\pi\lambda_{n}x)+\frac{(\pi x)^{\frac{1}{2}-s}}{4}\cdot\frac{\Gamma\left(s-\frac{1}{2}\right)}{1+\frac{1}{\pi p}}-\frac{\pi^{-s}x^{-s-\frac{1}{2}}\Gamma(s)}{4}.\label{second of second analogue of Watson formula}
\end{align}
\end{theorem}

\begin{proof}
In the proof Theorem 3.2., we have used the representation (\ref{Integral representation Bessel laplace})
in order to evaluate the integrals $I_{m,p}^{\star(1)}(s,x)$ and
$I_{m,p}^{\star(2)}(s,x)$. Since we can reverse the roles of the
variables $x$ and $t$ in the representation (\ref{Integral representation Bessel laplace}),
we may write an alternative version
\[
\frac{1}{\left(x^{2}+t^{2}\right)^{s}}=\frac{\sqrt{\pi}\,2^{\frac{1}{2}-s}}{\Gamma(s)t^{s-\frac{1}{2}}}\intop_{0}^{\infty}y^{s-\frac{1}{2}}J_{s-\frac{1}{2}}(ty)\,e^{-xy}\,dy,
\]
and so, returning to the expression (\ref{evaluation Imp2}), we obtain
for $I_{m,p}^{\star(2)}(s,x)$, 
\begin{align*}
I_{m,p}^{\star(2)}(s,x) & =e^{2\pi pm}\sum_{\ell=1}^{m}\left(\begin{array}{c}
m\\
\ell
\end{array}\right)\left(4\pi p\right)^{\ell}(-1)^{m-\ell}\,\intop_{m}^{\infty}e^{-2\pi p\,t}\,\frac{(t-m)^{\ell-1}}{(\ell-1)!}\,\frac{\Gamma(s)x^{2s}}{\left(x^{2}+t^{2}\right)^{s}}\,dt\\
=e^{2\pi pm}\sum_{\ell=1}^{m} & \left(\begin{array}{c}
m\\
\ell
\end{array}\right)\left(4\pi p\right)^{\ell}(-1)^{m-\ell}\,\intop_{m}^{\infty}e^{-2\pi p\,t}\,\frac{(t-m)^{\ell-1}}{(\ell-1)!}\,\frac{\sqrt{\pi}2^{\frac{1}{2}-s}x^{2s}}{t^{s-\frac{1}{2}}}\,\intop_{0}^{\infty}y^{s-\frac{1}{2}}J_{s-\frac{1}{2}}(ty)\,e^{-xy}\,dy\,dt\\
=e^{2\pi pm}\sqrt{\pi}2^{\frac{1}{2}-s} & x^{2s}\,\sum_{\ell=1}^{m}\,\left(\begin{array}{c}
m\\
\ell
\end{array}\right)\left(4\pi p\right)^{\ell}(-1)^{m-\ell}\intop_{0}^{\infty}y^{s-\frac{1}{2}}\,e^{-xy}\,\intop_{m}^{\infty}e^{-2\pi p\,t}\,\frac{(t-m)^{\ell-1}}{(\ell-1)!}\,\frac{J_{s-\frac{1}{2}}(ty)}{t^{s-\frac{1}{2}}}\,dt\,dy.
\end{align*}

Using the Poisson representation for the Bessel function of the first
kind {[}\cite{NIST}, p. 224, eq. (10.9.4){]}
\begin{equation}
z^{-\nu}J_{\nu}(z)=\frac{2^{1-\nu}}{\sqrt{\pi}\Gamma(\nu+\frac{1}{2})}\,\intop_{0}^{1}\left(1-u^{2}\right)^{\nu-\frac{1}{2}}\,\cos(zu)\,\,du,\,\,\,\,\,\,\text{Re}(\nu)>-\frac{1}{2},\,\,\,z\in\mathbb{C},\label{Poisson representation for first kind bessel}
\end{equation}
we have from absolute convergence that
\begin{align*}
\intop_{m}^{\infty}e^{-2\pi p\,t}\,\frac{(t-m)^{\ell-1}}{(\ell-1)!}\,\frac{J_{s-\frac{1}{2}}(ty)}{t^{s-\frac{1}{2}}}\,dt & =\frac{y^{s-\frac{1}{2}}2^{\frac{3}{2}-s}}{\sqrt{\pi}\Gamma(s)}\,\intop_{0}^{1}\left(1-u^{2}\right)^{s-1}\intop_{m}^{\infty}e^{-2\pi p\,t}\,\frac{(t-m)^{\ell-1}}{(\ell-1)!}\,\cos(yu\,t)\,dt\,du\\
=\frac{2^{\frac{3}{2}-s}y^{s-\frac{1}{2}}}{\sqrt{\pi}\Gamma\left(s\right)}\intop_{0}^{1}\left(1-u^{2}\right)^{s-1} & \text{Re}\left[\left(2\pi p+iyu\right)^{-\ell}e^{-2\pi mp-imyu}\right]\,du.
\end{align*}

Therefore, the summation over $\ell$ yields the result
\begin{align*}
I_{m,p}^{\star(2)}(s,x) & =\sqrt{\pi}2^{\frac{1}{2}-s}\,x^{2s}\,\intop_{0}^{\infty}y^{s-\frac{1}{2}}\,e^{-xy}\frac{2^{\frac{3}{2}-s}y^{s-\frac{1}{2}}}{\sqrt{\pi}\Gamma\left(s\right)}\,\intop_{0}^{1}\left(1-u^{2}\right)^{s-1}\text{Re}\left[\sum_{\ell=1}^{m}\left(\begin{array}{c}
m\\
\ell
\end{array}\right)\left(\frac{4\pi p}{2\pi p+iyu}\right)^{\ell}(-1)^{m-\ell}e^{-imyu}\right]\,du\,dy\\
 & =\sqrt{\pi}2^{\frac{1}{2}-s}\,x^{2s}\,\intop_{0}^{\infty}y^{s-\frac{1}{2}}\,e^{-xy}\,\frac{2^{\frac{3}{2}-s}y^{s-\frac{1}{2}}}{\sqrt{\pi}\Gamma\left(s\right)}\,\intop_{0}^{1}\left(1-u^{2}\right)^{s-1}\text{Re}\left[\left(\frac{p-iy\frac{u}{2\pi}}{p+iy\frac{u}{2\pi}}\right)^{m}e^{-imyu}-(-1)^{m}e^{-imyu}\right]\,du\,dy\\
 & =\text{\ensuremath{\frac{2^{2-2s}}{\Gamma(s)}}}\,x^{2s}\,\intop_{0}^{\infty}y^{2s-1}\,e^{-xy}\,\intop_{0}^{1}\left(1-u^{2}\right)^{s-1}\,\text{Re}\left[\sigma^{m}\left(-\frac{iyu}{2\pi}\right)e^{-imyu}-(-1)^{m}e^{-imyu}\right]\,du\,dy,
\end{align*}
which proves that
\[
\text{\ensuremath{I_{m,p}^{\star}(s,x)=\frac{2^{2-2s}}{\Gamma(s)}}}\,x^{2s}\,\intop_{0}^{\infty}\intop_{0}^{1}\,y^{2s-1}\,e^{-xy}\left(1-u^{2}\right)^{s-1}\,\text{Re}\left[\sigma^{m}\left(-\frac{iyu}{2\pi}\right)e^{-imyu}\right]\,du\,dy.
\]

Finally, returning to (\ref{after using the functional equation})
and to (\ref{using thing at the end}), we obtain 
\begin{align*}
\frac{x^{-1}}{\sqrt{\pi}}\,\sum_{m=1}^{\infty}I_{m,p}^{\star}(s,x) & =\frac{2^{2-2s}x^{2s-1}}{\sqrt{\pi}\Gamma(s)}\sum_{m=1}^{\infty}\intop_{0}^{\infty}\intop_{0}^{1}y^{2s-1}e^{-xy}\left(1-u^{2}\right)^{s-1}\text{Re}\left[\sigma^{m}\left(-\frac{iyu}{2\pi}\right)e^{-imyu}\right]\,du\,dy\\
 & =\frac{2}{(\pi x)^{\frac{1}{2}-s}}\sum_{n=1}^{\infty}\frac{p^{2}+\lambda_{n}^{2}}{p\left(p+\frac{1}{\pi}\right)+\lambda_{n}^{2}}\,\lambda_{n}^{s-\frac{1}{2}}K_{s-\frac{1}{2}}(2\pi\lambda_{n}x)+\frac{1}{2}\cdot\frac{\Gamma\left(s-\frac{1}{2}\right)}{1+\frac{1}{\pi p}}-\frac{\Gamma\left(s\right)}{2x\sqrt{\pi}},
\end{align*}
which implies that
\begin{align*}
\frac{2^{1-2s}x^{s-\frac{1}{2}}\pi^{-s}}{\Gamma(s)}\sum_{m=1}^{\infty}\intop_{0}^{\infty}\intop_{0}^{1}y^{2s-1}e^{-xy}\left(1-u^{2}\right)^{s-1}\text{Re}\left[\sigma^{m}\left(-\frac{iyu}{2\pi}\right)e^{-imyu}\right]\,du\,dy\\
=\sum_{n=1}^{\infty}\frac{p^{2}+\lambda_{n}^{2}}{p\left(p+\frac{1}{\pi}\right)+\lambda_{n}^{2}}\,\lambda_{n}^{s-\frac{1}{2}}K_{s-\frac{1}{2}}(2\pi\lambda_{n}x)+\frac{(\pi x)^{\frac{1}{2}-s}}{4}\cdot\frac{\Gamma\left(s-\frac{1}{2}\right)}{1+\frac{1}{\pi p}}-\frac{\Gamma\left(s\right)(\pi x)^{\frac{1}{2}-s}}{4x\sqrt{\pi}}.
\end{align*}

Watson's formula (\ref{second of second analogue of Watson formula}) is now achieved after using the elementary identity
\begin{equation}
\text{Re}\left[\sigma^{m}\left(-ix\right)e^{-2\pi imx}\right]=\cos\left(2\pi mx+2m\arctan\left(\frac{x}{p}\right)\right).\label{Euler representation almost at the end of the proof of Lemma}
\end{equation} 
\end{proof}

\begin{corollary}
Watson's formulas (\ref{Watson Formula intro}) and (\ref{yet another analogue Watson from second})
hold. 
\end{corollary}

\begin{proof}
We only give details when $p\rightarrow\infty$. In fact, for $\text{Re}(s)>\frac{1}{2}$
and $x>0$, the right-hand side of (\ref{second of second analogue of Watson formula})
reduces to
\[
\sum_{n=1}^{\infty}n^{s-\frac{1}{2}}K_{s-\frac{1}{2}}(2\pi n\,x)+\frac{(\pi x)^{\frac{1}{2}-s}}{4}\Gamma\left(s-\frac{1}{2}\right)-\frac{\pi^{-s}x^{-s-\frac{1}{2}}\Gamma(s)}{4}.
\]

At the same time, if we use the Poisson representation (\ref{Poisson representation for first kind bessel})
and (\ref{Integral representation Bessel laplace}), we see that the left-hand side can be further simplified to
\begin{align*}
\frac{2^{1-2s}x^{s-\frac{1}{2}}\pi^{-s}}{\Gamma(s)}\sum_{m=1}^{\infty}\intop_{0}^{\infty}\intop_{0}^{1}y^{2s-1}e^{-xy}\left(1-u^{2}\right)^{s-1}\cos\left(myu\right)\,du\,dy&=2^{-s-\frac{1}{2}}x^{s-\frac{1}{2}}\pi^{\frac{1}{2}-s}\sum_{m=1}^{\infty}\frac{1}{m^{s-\frac{1}{2}}}\intop_{0}^{\infty}y^{s-\frac{1}{2}}e^{-xy}\,J_{s-\frac{1}{2}}(my)\,dy\\
&=\frac{\Gamma(s)x^{s-\frac{1}{2}}\pi^{-s}}{2}\sum_{m=1}^{\infty}\frac{1}{\left(m^{2}+x^{2}\right)^{s}},\,\,\,\,\,\text{Re}(s)>\frac{1}{2}.
\end{align*}
\end{proof}

\section{Generalizations of the Epstein zeta function}
\subsection{First Analogue of Epstein's zeta function}
Here we will consider a generalization of the Epstein zeta function
(\ref{Epstein def}) when the summation indices are replaced by two
Koshliakov sequences $(\lambda_{m},\lambda_{n})_{m,n\in\mathbb{Z}}$.
Note that if $y$ is a root of the equation (\ref{Transcendental equation})
then $-y$ also is. Therefore, we extend Koshliakov's sequence $\lambda_{n}$
to the nonpositive integers by simply setting $\lambda_{-n}:=-\lambda_{n}$
and $\lambda_{0}:=0$.

\bigskip{}

Although we may proceed with any positive quadratic form, say $Q(x,y)=ax^{2}+bxy+cy^{2}$, 
it is enough for the purposes of our paper to consider the following analogue of
(\ref{Epstein def}),
\begin{equation}
\zeta_{p,p^{\prime}}(s,c):=\sum_{m,n\neq0}\frac{\left(p^{2}+\lambda_{n}^{2}\right)\cdot\left(p^{\prime2}+\lambda_{n}^{\prime2}\right)}{\left(p(p+\frac{1}{\pi})+\lambda_{n}^{2}\right)\cdot\left(p^{\prime}(p^{\prime}+\frac{1}{\pi})+\lambda_{n}^{\prime2}\right)}\,\frac{1}{\left(\lambda_{m}^{2}+c\lambda_{n}^{\prime2}\right)^{s}},\,\,\,\,\text{Re}(s)>1,\,\,c>0,\label{definition Epstein in statement}
\end{equation}
where $m,\,n\neq0$ here means that only the term $m=n=0$ is omitted
from the sum above.

\bigskip{}

It is simple to see that the Dirichlet series defining $\zeta_{p,p^{\prime}}(s,c)$
converges absolutely in the half-plane $\text{Re}(s)>1$. In the next
result we give a formula which provides the continuation of (\ref{definition Epstein in statement})
to $\text{Re}(s)<$1.

\begin{theorem}
For any $c>0$ and $p,p^{\prime}\in\mathbb{R}_{+}$, consider the
analogue of Epstein's zeta function given by (\ref{definition Epstein in statement}).

Then $\zeta_{p,p^{\prime}}(s,c)$ can be continued to the half-plane $\text{Re}(s)<1$
by the following integral formula,
\begin{align}
\left(\frac{\pi}{\sqrt{c}}\right)^{-s}\Gamma(s)\,\zeta_{p,p^{\prime}}(s,c) & =\frac{2c^{s/2}\,\pi^{-s}\Gamma(s)}{1+\frac{1}{\pi p^{\prime}}}\,\zeta_{p}(2s)+2\,c^{\frac{1-s}{2}}\,\pi^{-\left(s-\frac{1}{2}\right)}\Gamma\left(s-\frac{1}{2}\right)\zeta_{p^{\prime}}(2s-1)+\nonumber \\
+\frac{2^{4-2s}\,\pi^{1-s}\,c^{\frac{1-s}{2}}}{\Gamma(1-s)} & \sum_{n=1}^{\infty}\frac{p^{\prime2}+\lambda_{n}^{\prime2}}{p^{\prime}\left(p^{\prime}+\frac{1}{\pi}\right)+\lambda_{n}^{\prime2}}\,\lambda_{n}^{\prime1-2s}\,\intop_{0}^{\infty}\,\frac{y^{-s}(y+1)^{-s}}{\sigma\left(\lambda_{n}^{\prime}\,(2y+1)\,\sqrt{c}\right)e^{2\pi\lambda_{n}^{\prime}(2y+1)\sqrt{c}}-1}dy.\label{representation Re(s)<1}
\end{align}

Furthermore, (\ref{representation Re(s)<1}) provides the continuation of $\zeta_{p,p^{\prime}}(s,c)$ to the entire complex
plane as an analytic function everywhere except at a simple pole located at $s=1$, whose residue is
\begin{equation}
\text{Res}_{s=1}\zeta_{p,p^\prime}(s,c)=\frac{\pi}{\sqrt{c}}. \label{residue First analogue Epstein}
\end{equation}
\end{theorem}
\begin{proof}
The proof comes immediately from our first analogue of Watson's formula (\ref{formula in a region compact to be extended}). We start
by choosing $\text{Re}(s)>\mu>1$ and by writing the generalized Epstein
zeta function (\ref{definition Epstein in statement}) in the following form
\begin{equation}
\zeta_{p,p^{\prime}}(s,c)=\frac{2}{1+\frac{1}{\pi p^{\prime}}}\,\zeta_{p}(2s)+\frac{2c^{-s}}{1+\frac{1}{\pi p}}\,\zeta_{p^{\prime}}(2s)+4\sum_{m,n=1}^{\infty}\frac{\left(p^{2}+\lambda_{n}^{2}\right)\left(p^{\prime2}+\lambda_{n}^{\prime2}\right)}{\left(p(p+\frac{1}{\pi})+\lambda_{n}^{2}\right)\left(p^{\prime}(p^{\prime}+\frac{1}{\pi})+\lambda_{n}^{\prime2}\right)}\,\frac{1}{\left(\lambda_{m}^{2}+c\lambda_{n}^{\prime2}\right)^{s}}.\label{definition Epstein at beginning}
\end{equation}

\bigskip{}

If, on the double series in (\ref{definition Epstein at beginning}),
we fix the variable of summation $n$ and sum over $m$ by using (\ref{formula valid for all s})
with $x$ being replaced by $\sqrt{c}\lambda_{n}^{\prime}$, we are
able to deduce that 

\begin{equation}
\zeta_{p,p^{\prime}}(s,c)=\frac{2}{1+\frac{1}{\pi p^{\prime}}}\,\zeta_{p}(2s)+\frac{2\sqrt{\pi}\,c^{\frac{1}{2}-s}\Gamma\left(s-\frac{1}{2}\right)}{\Gamma(s)}\zeta_{p^{\prime}}(2s-1)+H_{p,p^{\prime}}(s,c),\,\,\,\,\,\text{Re}(s)>1,\label{Analytic Continuation}
\end{equation}
where 
\begin{align}
\frac{\Gamma(s)\,H_{p,p^{\prime}}(s,c)}{8\pi^{s}c^{\frac{1}{4}-\frac{s}{2}}} & =\sum_{m,n=1}^{\infty}\frac{\left(p^{\prime2}+\lambda_{n}^{\prime2}\right)\,(-1)^{m}}{p^{\prime}\left(p^{\prime}+\frac{1}{\pi}\right)+\lambda_{n}^{\prime2}}\,\left(\frac{m}{\lambda_{n}^{\prime}}\right)^{s-\frac{1}{2}}\,K_{s-\frac{1}{2}}\left(2\pi m\,\lambda_{n}^{\prime}\sqrt{c}\right)+\nonumber \\
+\sum_{m,n=1}^{\infty}\frac{p^{\prime2}+\lambda_{n}^{\prime2}}{p^{\prime}\left(p^{\prime}+\frac{1}{\pi}\right)+\lambda_{n}^{\prime2}} \,\left(\frac{m}{\lambda_{n}^{\prime}}\right)^{s-\frac{1}{2}}e^{2\pi pm}\,&\sum_{\ell=1}^{m}\left(\begin{array}{c}
m\\
\ell
\end{array}\right)\,(-1)^{m-\ell}\left(4\pi mp\right)^{\ell}\intop_{1}^{\infty}t^{s-\frac{1}{2}}\,e^{-2\pi mpt}\,\frac{(t-1)^{\ell-1}}{(\ell-1)!}\,K_{s-\frac{1}{2}}\left(2\pi m\sqrt{c}\lambda_{n}^{\prime}t\right)\,dt.\label{Explicit rep entire part}
\end{align}

\bigskip{}

We now claim that the right-hand side of (\ref{Analytic Continuation})
constitutes the analytic continuation of $\zeta_{p,p^{\prime}}(s,\,c)$
to the entire complex plane and when $\text{Re}(s)<1$, it is reduced
to (\ref{representation Re(s)<1}). Since the continuations of $\zeta_{p}(s)$
and $\zeta_{p^{\prime}}(s)$ are assured by Koshliakov's paper \cite{KOSHLIAKOV},
we only need to focus on the continuation of $H_{p,p^{\prime}}(s,c)$.
This is now analogous to the case where we have treated (\ref{Watson formula with explicit bessel})
and it is possible to show that it defines an entire function
of $s\in\mathbb{C}$.

\bigskip{}

Take now the expression defining $H_{p,p^{\prime}}(s,c)$ (\ref{Explicit rep entire part})
and assume that $\text{Re}(s)<1$. We see from the computations leading
to (\ref{computations leading to final expreeeessssssiiiiooooon})
that it can be written in the form
\[
2^{4-2s}c^{\frac{1}{2}-s}\sin(\pi s)\sum_{n=1}^{\infty}\frac{p^{\prime2}+\lambda_{n}^{\prime2}}{p^{\prime}\left(p^{\prime}+\frac{1}{\pi}\right)+\lambda_{n}^{\prime2}}\,\lambda_{n}^{\prime1-2s}\,\intop_{0}^{\infty}\,\frac{y^{-s}(y+1)^{-s}}{\sigma\left(\lambda_{n}^{\prime}\,(2y+1)\,\sqrt{c}\right)e^{2\pi\lambda_{n}^{\prime}(2y+1)\sqrt{c}}-1}dy,
\]
which gives (\ref{representation Re(s)<1}).

\bigskip{}

By a standard verification of the right-hand side of (\ref{Analytic Continuation}), it is easily seen that $\zeta_{p,p^{\prime}}(s,c)$
has removable singularities located at $s=\frac{1}{2}-k$, $k\in\mathbb{N}_{0}$.
Since $\zeta_{p}(2s)$ and $\Gamma\left(s-\frac{1}{2}\right)$ are
analytic in a neighborhood of the point $s=1$, we can conclude that
$\zeta_{p,p^{\prime}}(s,c)$ must have a pole located at $s=1$ (coming
from the function $\zeta_{p^{\prime}}(2s-1)$) with residue $\pi/\sqrt{c}$.
\end{proof}

\bigskip{}

\begin{remark}
It is possible to deform the path of integration on the right-hand
side of (\ref{representation Re(s)<1}) in order to make it valid
for every $s\in\mathbb{C}$. See Remark 3.1. above.
\end{remark}

\begin{remark}
It is clear from the proof of Corollary 3.1. (see the steps leading
to (\ref{series obtained})) that (\ref{representation Re(s)<1})
and (\ref{Explicit rep entire part}) are generalizations of the Selberg-Chowla
formula (\ref{Selberg Chowla Formula}) when $Q(m,n)=m^{2}+cn^{2}$,
$c>0$.
\end{remark}

Having proved that $\zeta_{p,p^{\prime}}(s,c)$ possesses a pole at
$s=1$, we now study how it behaves around this singularity, generalizing
a classical formula due to Kronecker \cite{Siegel_Analytic Number Theory} (see (\ref{Kronecker limit formula diagonal form}) above). 

\begin{corollary}
Let $p,p^{\prime}\in\mathbb{R}_{+}$ and let $\zeta_{p,p^{\prime}}(s,c)$
be the generalized Epstein $\zeta-$function (\ref{definition Epstein in statement}). Moreover, let $\sigma(t)$ be defined by (\ref{definition sigma p sigma p'}). 
Then $\zeta_{p,p^{\prime}}(s,c)$ admits the meromorphic expansion around $s=1$, 
\begin{align}
\zeta_{p,p^{\prime}}(s,\,c) & =\frac{\pi}{\sqrt{c}}\,\frac{1}{s-1}+\frac{\pi^{2}}{3}\,\frac{1+\frac{3}{\pi p}(1+\frac{1}{\pi p})}{\left(1+\frac{1}{\pi p^{\prime}}\right)\left(1+\frac{1}{\pi p}\right)^{2}}+\nonumber \\
+\frac{\pi}{\sqrt{c}} & \left(2C_{p^{\prime}}^{(1)}-\log\left(4c\right)+4\,\sum_{n=1}^{\infty}\frac{p^{\prime2}+\lambda_{n}^{\prime2}}{p^{\prime}\left(p^{\prime}+\frac{1}{\pi}\right)+\lambda_{n}^{\prime2}}\cdot\frac{\lambda_{n}^{\prime-1}}{\sigma\left(\sqrt{c}\lambda_{n}^{\prime}\right)e^{2\pi\sqrt{c}\lambda_{n}^{\prime}}-1}\right)+O(s-1),\label{corrected Kronecker Limit formula}
\end{align}
where $C_{p}^{(1)}$ is Koshliakov's analogue of the Euler-Mascheroni
constant (\ref{Euler Mascheroni Koshliakov sense}).
\end{corollary}

\begin{proof}
The proof consists in evaluating the right-hand side of (\ref{Analytic Continuation})
when $s\rightarrow1$. Let us recall the following expansions around $s=1$,
\begin{equation}
\frac{\sqrt{\pi}\,\Gamma(s-\frac{1}{2})}{\Gamma(s)}=\pi-2\pi\log(2)\,(s-1)+O\left(s-1\right)^{2}\label{First meromorphic Kronecker limit}
\end{equation}
and
\begin{equation}
\zeta_{p^{\prime}}(2s-1)=\frac{1}{2(s-1)}+C_{p^{\prime}}^{(1)}+O\left(s-1\right),\label{Second meromorphic Kronecker limit}
\end{equation}
where $C_{p^{\prime}}^{(1)}$ is given by (\ref{Euler Mascheroni Koshliakov sense}).
  Combining (\ref{First meromorphic Kronecker limit}) and (\ref{Second meromorphic Kronecker limit}),
we see that the second term in (\ref{Analytic Continuation}) can
be written as 
\[
\frac{2\sqrt{\pi}\,c^{\frac{1}{2}-s}\Gamma\left(s-\frac{1}{2}\right)}{\Gamma(s)}\zeta_{p^{\prime}}(2s-1)=\frac{\pi}{\sqrt{c}}\,\frac{1}{s-1}+\frac{\pi}{\sqrt{c}}\left(2C_{p^{\prime}}^{(1)}-\log\left(4c\right)\right)+O\left(s-1\right).
\]
which gives, after the use of (\ref{Analytic Continuation}), 
\begin{equation}
\zeta_{p,p^{\prime}}(s,\,c)=\frac{\pi}{\sqrt{c}}\,\frac{1}{s-1}+\frac{\pi}{\sqrt{c}}\left(2C_{p^{\prime}}^{(1)}-\log\left(4c\right)\right)+\frac{\pi^{2}}{3}\,\frac{1+\frac{3}{\pi p}(1+\frac{1}{\pi p})}{\left(1+\frac{1}{\pi p^{\prime}}\right)\left(1+\frac{1}{\pi p}\right)^{2}}+H_{p,p^{\prime}}(1,c)+O\left(s-1\right)\label{first approximation Kronecker}
\end{equation}
where we have employed the identity $\zeta_{p}(2)=\frac{\pi^{2}}{6}\,\frac{1+\frac{3}{\pi p}(1+\frac{1}{\pi p})}{(1+\frac{1}{\pi p})^{2}}$
{[}\cite{KOSHLIAKOV}, p. 22, eq. (39){]}.   Using once more formula
(\ref{particular case of bessel!}), and doing no more than the
calculations made in (\ref{almost in the step}), we can derive that
\begin{equation}
H_{p,p^{\prime}}(1,c)=\frac{4\pi}{\sqrt{c}}\,\sum_{n=1}^{\infty}\frac{p^{\prime2}+\lambda_{n}^{\prime2}}{p^{\prime}\left(p^{\prime}+\frac{1}{\pi}\right)+\lambda_{n}^{\prime2}}\cdot\frac{\lambda_{n}^{\prime-1}}{\sigma\left(\sqrt{c}\lambda_{n}^{\prime}\right)e^{2\pi\sqrt{c}\lambda_{n}^{\prime}}-1}\label{expression entire part almost end Kronecker.}
\end{equation}
which proves (\ref{corrected Kronecker Limit formula}). 
\end{proof}

\begin{remark}
The infinite
series appearing in (\ref{corrected Kronecker Limit formula}) can be regarded as a generalization of the logarithm of Dedekind's $\eta-$function (\ref{Dedekind classic}).
Taking $p,p^{\prime}\rightarrow\infty$ in our formula above, the classical limit formula
of Kronecker (\ref{Kronecker limit formula intro}) can be derived.
Analogously to section 2 above, we can obtain several new analogues of Kronecker's
limit formula by fixing one of the parameters and varying the other.
Of course, when $p,p^{\prime}$ are both tending to zero or infinity
(not necessary the same limit), the formulas resulting from (\ref{corrected Kronecker Limit formula})
may also be achieved through Kronecker's second limit formula [\cite{Siegel_Analytic Number Theory}, p. 22].
\end{remark}

\bigskip{}
Our first corollary is given when we take $p^{\prime}\rightarrow0^{+}$
while $p$ is kept fixed.

\begin{corollary}
For every $p\in\mathbb{R}_{+}$ and $c>0$, consider the Epstein zeta
function defined by
\[
\zeta_{p,0}(s,c):=\sum_{m,n\neq0}\frac{p^{2}+\lambda_{m}^{2}}{p(p+\frac{1}{\pi})+\lambda_{m}^{2}}\,\frac{1}{\left(\lambda_{m}^{2}+c\,\left(n-\frac{1}{2}\right)^{2}\right)^{s}},\,\,\,\,\text{Re}(s)>1.
\]

Then $\zeta_{p,0}(s,c)$ admits the meromorphic expansion 
\[
\zeta_{p,0}(s,c)=\frac{\pi}{\sqrt{c}}\,\frac{1}{s-1}+\frac{\pi}{\sqrt{c}}\left(2\gamma-\log\left(\frac{c}{4}\right)+8\,\sum_{n=1}^{\infty}\frac{1}{2n-1}\cdot\frac{1}{\sigma\left(\sqrt{c}\lambda_{n}^{\prime}\right)e^{(2n-1)\pi\sqrt{c}}-1}\right)+O(s-1).
\]

In particular, the standard expansions holds
\[
\zeta_{0,0}(s,c)=\frac{\pi}{\sqrt{c}}\,\frac{1}{s-1}+\frac{\pi}{\sqrt{c}}\left(2\gamma-\log\left(\frac{c}{4}\right)-8\,\sum_{n=1}^{\infty}\frac{1}{2n-1}\cdot\frac{1}{e^{(2n-1)\pi\sqrt{c}}+1}\right)+O(s-1).
\]
\end{corollary}

\bigskip{}

Now, letting $p^{\prime}\rightarrow\infty$ and fixing $p\in\mathbb{R}_{+}$, we deduce the following result.
\begin{corollary}
For every $p\in\mathbb{R}_{+}$ and $c>0$, consider the Epstein zeta
function defined by
\[
\zeta_{p,\infty}(s,c)=\sum_{m,n\neq0}\frac{p^{2}+\lambda_{n}^{2}}{p(p+\frac{1}{\pi})+\lambda_{n}^{2}}\,\frac{1}{\left(\lambda_{m}^{2}+cn^{2}\right)^{s}},\,\,\,\,\,\text{Re}(s)>1.
\]

Then $\zeta_{p,\infty}(s,c)$ has the following meromorphic expansion around $s=1$,
\begin{equation*}
\zeta_{p,\infty}(s,\,c)=\frac{\pi}{\sqrt{c}}\,\frac{1}{s-1}+\frac{\pi^{2}}{3}\,\frac{1+\frac{3}{\pi p}(1+\frac{1}{\pi p})}{\left(1+\frac{1}{\pi p}\right)^{2}}
+\frac{\pi}{\sqrt{c}} \left(2\gamma-\log\left(4c\right)+4\,\sum_{n=1}^{\infty}\frac{1}{n}\cdot\frac{1}{\sigma\left(\sqrt{c}n\right)e^{2\pi\sqrt{c}n}-1}\right)+O(s-1).
\end{equation*}

In particular, the usual Kronecker
limit formula (\ref{Kronecker limit formula diagonal form}) takes place. Moreover,
\[
\zeta_{0,\infty}(s,c)=\frac{\pi}{\sqrt{c}}\,\frac{1}{s-1}+\pi^{2}+\frac{\pi}{\sqrt{c}}\left(2\gamma-\log\left(4c\right)-4\,\sum_{n=1}^{\infty}\frac{1}{n}\cdot\frac{1}{e^{2\pi\sqrt{c}n}+1}\right)+O(s-1)
\]
\end{corollary}

\bigskip{}

If we now take $p\rightarrow\infty$ and fix $p^{\prime}\in\mathbb{R}_{+}$, we can derive the following formula.
\begin{corollary}
For every $p^{\prime}\in\mathbb{R}_{+}$ and $c>0$, consider the
Epstein zeta function defined by
\begin{equation*}
\zeta_{\infty,p^{\prime}}(s,c):=\sum_{m,n\neq0}\frac{p^{\prime2}+\lambda_{n}^{\prime2}}{p^{\prime}(p^{\prime}+\frac{1}{\pi})+\lambda_{n}^{\prime2}}\,\frac{1}{\left(m^{2}+c\lambda_{n}^{\prime2}\right)^{s}},\,\,\,\,\text{Re}(s)>1.
\end{equation*}

Then $\zeta_{\infty,p^{\prime}}(s,c)$ has a Laurent expansion around
$s=1$ given by 
\[
\zeta_{\infty,p^{\prime}}(s,\,c)=\frac{\pi}{\sqrt{c}}\,\frac{1}{s-1}+\frac{\pi^{2}}{3}\cdot\frac{1}{1+\frac{1}{\pi p^{\prime}}}+\frac{\pi}{\sqrt{c}}\left(2C_{p^{\prime}}^{(1)}-\log\left(4c\right)+4\,\sum_{n=1}^{\infty}\frac{p^{\prime2}+\lambda_{n}^{\prime2}}{p^{\prime}\left(p^{\prime}+\frac{1}{\pi}\right)+\lambda_{n}^{\prime2}}\cdot\frac{\lambda_{n}^{\prime-1}}{e^{2\pi\sqrt{c}\lambda_{n}^{\prime}}-1}\right)+O\left(s-1\right)
\]

In particular, we have the expansion
\[
\zeta_{\infty,0}(s,\,c)=\frac{\pi}{\sqrt{c}}\,\frac{1}{s-1}+\frac{\pi}{\sqrt{c}}\left(2\gamma-\log\left(\frac{c}{4}\right)+8\,\sum_{n=1}^{\infty}\frac{1}{2n-1}\cdot\frac{1}{e^{\pi\sqrt{c}(2n-1)}-1}\right)+O\left(s-1\right).
\]
\end{corollary}

\bigskip{}

Our final corollary of the Kronecker limit formula now comes from taking $p\rightarrow0^{+}$ on
our generalized Kronecker's limit formula.
\begin{corollary}
For every $p^{\prime}\in\mathbb{R}_{+}$ and $c>0$, consider the
Epstein zeta function defined by
\[
\zeta_{0,p^{\prime}}(s,c):=\sum_{m,n\neq0}\frac{p^{\prime2}+\lambda_{n}^{\prime2}}{p^{\prime}\left(p^{\prime}+\frac{1}{\pi}\right)+\lambda_{n}^{\prime2}}\,\frac{1}{\left(\left(m-\frac{1}{2}\right)^{2}+c\,\lambda_{n}^{\prime2}\right)^{s}},\,\,\,\,\text{Re}(s)>1.
\]
Then $\zeta_{0,p^{\prime}}(s,c)$ admits the meromorphic expansion
around $s=1$,
\[
\zeta_{0,p^{\prime}}(s,\,c)=\frac{\pi}{\sqrt{c}}\,\frac{1}{s-1}+\frac{\pi^{2}}{1+\frac{1}{\pi p^{\prime}}}+\frac{\pi}{\sqrt{c}}\,\left(2C_{p^{\prime}}^{(1)}-\log\left(4c\right)-4\,\sum_{n=1}^{\infty}\frac{p^{\prime2}+\lambda_{n}^{\prime2}}{p^{\prime}\left(p^{\prime}+\frac{1}{\pi}\right)+\lambda_{n}^{\prime2}}\cdot\frac{\lambda_{n}^{\prime-1}}{e^{2\pi\sqrt{c}\lambda_{n}^{\prime}}+1}\right)+O(s-1).
\]
\end{corollary}

\bigskip{}

\bigskip{}

Although it is possible to write a functional equation for (\ref{definition Epstein in statement})
using the general Selberg-Chowla formula (\ref{representation Re(s)<1}),
we shall omit it for this case because there is an asymmetry between
the zeta functions on both sides of it. This asymmetry is due to the
drastic differences between Koshliakov's zeta functions $\zeta_{p}(s)$
and $\eta_{p}(s)$.

\bigskip{}

Nevertheless, we end this section with an identity that $\zeta_{p,p^{\prime}}\left(s,c\right)$
shares with the classical Epstein zeta function (\ref{Epstein def}). This identity
concerns its central value, $\zeta_{p,p^{\prime}}\left(\frac{1}{2},c\right)$
and may be seen as an analogue of our extension of Ramanujan's formula (\ref{Analogue final in the statement}).

\bigskip{}

By using an idea of Selberg and Chowla \cite{selberg_chowla II}, but
closely following the proof by Bateman and Grosswald \cite{Bateman_Epstein}, 
we shall use this formula to prove that, under the condition of large $c$, $\zeta_{p,p^{\prime}}(s,c)$
will have a real zero lying on the interval $\left(\frac{1}{2},1\right)$.

\begin{corollary}
Let $\sigma(t)$ be defined
by (\ref{definition sigma p sigma p'}). Then the following identity
takes place 
\begin{align}
\zeta_{p,p^{\prime}}\left(\frac{1}{2},c\right) & =\frac{2C_{p}^{(1)}+\log\left(\frac{c}{4}\right)-4e^{2\pi p^{\prime}}Q_{2\pi p^{\prime}}(0)}{1+\frac{1}{\pi p^{\prime}}}+2C_{p^{\prime}}^{(2)}-2\log(2\pi)-2\gamma+\nonumber \\
+8\,\sum_{n=1}^{\infty} & \frac{p^{\prime2}+\lambda_{n}^{\prime2}}{p^{\prime}\left(p^{\prime}+\frac{1}{\pi}\right)+\lambda_{n}^{\prime2}}\,\intop_{0}^{\infty}\,\frac{1}{\sigma\left(\lambda_{n}^{\prime}\,(2y+1)\,\sqrt{c}\right)e^{2\pi\lambda_{n}^{\prime}(2y+1)\sqrt{c}}-1}\,\frac{dy}{\sqrt{y^{2}+y}},\label{Identity central value general Epstein!}
\end{align}
where $C_{p}^{(2)}$ denotes Koshliakov's second analogue of the Euler-Mascheroni
constant [\cite{KOSHLIAKOV}, p. 46, eq. (47)],
\begin{equation}
C_{p}^{(2)}:=\lim_{n\rightarrow\infty}\left\{ \sum_{k=1}^{n-1}\frac{(1,2\pi pk)_{k}}{k}-\frac{\log(n)}{1+\frac{1}{\pi p}}\right\} \label{second analogue Euler Mascheroni}
\end{equation}
and (in Koshliakov's notation), 
\begin{equation}
Q_{\mu}(s)=\intop_{\mu}^{\infty}t^{s-1}e^{-t}dt\label{Incomplete Gamma function Kosh sense}
\end{equation}
denotes the incomplete gamma function $\Gamma(s,\mu)$ [\cite{KOSHLIAKOV}, p. 25, eq. (49)]. 
\end{corollary}

\begin{proof}
We use our generalization of the Selberg-Chowla formula (\ref{representation Re(s)<1})
and let $s=\frac{1}{2}$. We obtain,
\begin{align}
\zeta_{p,p^{\prime}}\left(\frac{1}{2},c\right) & =c^{-1/4}\,\lim_{s\rightarrow\frac{1}{2}}\left\{ \frac{2c^{s/2}\,\pi^{-s}\Gamma(s)}{1+\frac{1}{\pi p^{\prime}}}\,\zeta_{p}(2s)+2\,c^{\frac{1-s}{2}}\,\pi^{-\left(s-\frac{1}{2}\right)}\Gamma\left(s-\frac{1}{2}\right)\zeta_{p^{\prime}}(2s-1)\right\} +\nonumber \\
 & +8\,\sum_{n=1}^{\infty}\frac{p^{\prime2}+\lambda_{n}^{\prime2}}{p^{\prime}\left(p^{\prime}+\frac{1}{\pi}\right)+\lambda_{n}^{\prime2}}\,\intop_{0}^{\infty}\,\frac{1}{\sigma\left(\lambda_{n}^{\prime}\,(2y+1)\,\sqrt{c}\right)e^{2\pi\lambda_{n}^{\prime}(2y+1)\sqrt{c}}-1}\,\frac{dy}{\sqrt{y^{2}+y}}.\label{Simple expression at beginning central value}
\end{align}

In order to compute the Laurent expansion of the first term in the braces, we
invoke the well-known expansions around $s=\frac{1}{2}$, (\ref{meromorphic x})
and (\ref{meromorphic Gamma}), as well as 
\begin{equation}
\Gamma\left(s\right)=\sqrt{\pi}-\sqrt{\pi}\left(2\log(2)+\gamma\right)\left(s-\frac{1}{2}\right)+O\left(s-\frac{1}{2}\right)^{2}.\label{expansion gamma function atound 1/2 onde mocre}
\end{equation}

Furthermore, by {[}\cite{KOSHLIAKOV}, pp. 48 and 49, eq. (51) and (54){]},
we have
\begin{equation}
\zeta_{p}(2s)=\frac{1}{2s-1}+C_{p}^{(1)}+O\left(s-\frac{1}{2}\right)
\end{equation}
and
\begin{align}
\zeta_{p}(2s-1) & =\zeta_{p}(0)+2\zeta_{p}^{\prime}\left(0\right)\,\left(s-\frac{1}{2}\right)+O\left(s-\frac{1}{2}\right)^{2}=\nonumber \\
=-\frac{1}{2}\,\frac{1}{1+\frac{1}{\pi p}} & +2\,\left\{ -\frac{\log(2\pi)}{2}+\frac{1}{2}C_{p}^{(2)}-\frac{\gamma}{2}-\frac{e^{2\pi p}}{1+\frac{1}{\pi p}}Q_{2\pi p}(0)\right\} \left(s-\frac{1}{2}\right)+O\left(s-\frac{1}{2}\right)^{2}\label{at beginning but will resurface on Koshliakoc}
\end{align}
where $Q_{\mu}(s)$ is given by (\ref{Incomplete Gamma function Kosh sense}).

\bigskip{}

We start by simplifying the first term in the braces: as one easily checks, 
\begin{align}
\frac{2}{1+\frac{1}{\pi p^{\prime}}}\,\left(\frac{\pi}{\sqrt{c}}\right)^{-s}\Gamma(s)\,\zeta_{p}(2s) & =\frac{2c^{1/4}}{1+\frac{1}{\pi p^{\prime}}}\left[1-\left(\gamma+\log\left(\frac{4\pi}{\sqrt{c}}\right)\right)\left(s-\frac{1}{2}\right)+O\left(s-\frac{1}{2}\right)^{2}\right]\times\nonumber \\
\times\left[\frac{1}{2s-1}+C_{p}^{(1)}+O\left(s-\frac{1}{2}\right)\right] & =\frac{2c^{1/4}}{1+\frac{1}{\pi p^{\prime}}}\left[\frac{1}{2s-1}+C_{p}^{(1)}-\frac{1}{2}\left(\gamma+\log\left(\frac{4\pi}{\sqrt{c}}\right)\right)+O\left(s-\frac{1}{2}\right)\right].\label{Simplification first term zeta (1/2) some details}
\end{align}

Analogously, we find that the second term on the braces can be written as
\begin{align}
2\,c^{\frac{1-s}{2}}\,\pi^{-\left(s-\frac{1}{2}\right)}\Gamma\left(s-\frac{1}{2}\right)\zeta_{p^{\prime}}(2s-1) & =2c^{1/4}\left[\frac{1}{s-\frac{1}{2}}-\gamma-\log\left(\pi\sqrt{c}\right)+O\left(s-\frac{1}{2}\right)\right]\times\nonumber \\
\times\left[-\frac{1}{2}\,\frac{1}{1+\frac{1}{\pi p^{\prime}}}+2\zeta_{p^{\prime}}^{\prime}(0)\,\left(s-\frac{1}{2}\right)+O\left(s-\frac{1}{2}\right)^{2}\right] & =2c^{1/4}\,\left[\frac{\zeta_{p^{\prime}}(0)}{s-\frac{1}{2}}+2\zeta_{p^{\prime}}^{\prime}(0)-\gamma\zeta_{p^{\prime}}(0)-\log(\pi\sqrt{c})\,\zeta_{p^{\prime}}(0)+O\left(s-\frac{1}{2}\right)\right].\label{Simplifications second term central value}
\end{align}

Finally, using the evaluation of $\zeta_{p^{\prime}}^{\prime}(0)$ given in (\ref{at beginning but will resurface on Koshliakoc}) and bringing (\ref{Simplification first term zeta (1/2) some details}) and (\ref{Simplifications second term central value}) together, we find (\ref{Identity central value general Epstein!}) after some straightforward manipulations.
    
\end{proof}

\bigskip{}

\begin{corollary}
For $p^{\prime}\rightarrow\infty$ and arbitrary $p\in \mathbb{R}_{+}$, consider the particular case of the
Epstein zeta function (\ref{definition Epstein in statement})
\[
\zeta_{p,\infty}(s,c):=\sum_{m,n\neq0}\frac{p^{2}+\lambda_{n}^{2}}{p(p+\frac{1}{\pi})+\lambda_{n}^{2}}\cdot\frac{1}{\left(\lambda_{m}^{2}+c\,n^{2}\right)^{s}},\,\,\,\,\text{Re}(s)>1.
\]

Then there exists some $c_{0}>16\pi^{2}e^{-2C_{p}^{(1)}}$ such that,
for every $c\geq c_{0}$, $\zeta_{p,\infty}(s,c)$ has a real zero on the interval $\left(\frac{1}{2},\,1\right)$.
\end{corollary}

\begin{proof}
Like Bateman and Grosswald \cite{Bateman_Epstein}, we bound
in a trivial manner the ``double series'' appearing as the last term on the expression (\ref{Simple expression at beginning central value}) defining $\zeta_{p,p^{\prime}}\left(\frac{1}{2},c\right)$. Indeed, since $\frac{p-t}{p+t}<1$ and $\lambda_{n}^{\prime}>n-\frac{1}{2}$ for every $n\in\mathbb{N}$, we have
\begin{align}
8\sum_{n=1}^{\infty}\left|\frac{p^{\prime2}+\lambda_{n}^{\prime2}}{p^{\prime}\left(p^{\prime}+\frac{1}{\pi}\right)+\lambda_{n}^{\prime2}}\,\intop_{0}^{\infty}\,\frac{1}{\sigma\left(\lambda_{n}^{\prime}\,(2y+1)\,\sqrt{c}\right)e^{2\pi\lambda_{n}^{\prime}(2y+1)\sqrt{c}}-1}\,\frac{dy}{\sqrt{y^{2}+y}}\right| & \leq\nonumber \\
\leq8\sum_{n=1}^{\infty}\intop_{0}^{\infty}\,\frac{1}{e^{2\pi\lambda_{n}^{\prime}(2y+1)\sqrt{c}}-1}\,\frac{dy}{\sqrt{y^{2}+y}}=8\,\sum_{m,n=1}^{\infty}\intop_{0}^{\infty}\frac{e^{-2\pi m\lambda_{n}^{\prime}\left(2y+1\right)\sqrt{c}}}{\sqrt{y^{2}+y}}\,dy\leq\nonumber \\
\leq8\sum_{m,n=1}^{\infty}K_{0}\left(\pi m\left(2n-1\right)\sqrt{c}\right)\leq8\sum_{n=1}^{\infty}d(n)\,K_{0}\left(\pi n\,\sqrt{c}\right).\label{bound for central value essenvial}
\end{align}

Using now a similar argument as in {[}\cite{Bateman_Epstein}, p.
371, eq. (20){]}, we get the simple bound
\begin{equation}
8\sum_{n=1}^{\infty}d(n)\,K_{0}\left(\pi n\,\sqrt{c}\right)<\frac{2^{5/2}\,e^{-\pi\sqrt{c}}}{c^{1/4}}.\label{bound just like Bateman and Grosswald for first part}
\end{equation}

Therefore, for some $-1<\theta<1$, it follows from (\ref{Identity central value general Epstein!})
and (\ref{bound just like Bateman and Grosswald for first part})
that
\[
\zeta_{p,p^{\prime}}\left(\frac{1}{2},c\right)=\frac{2C_{p}^{(1)}+\log\left(\frac{c}{4}\right)-4e^{2\pi p^{\prime}}Q_{2\pi p^{\prime}}(0)}{1+\frac{1}{\pi p^{\prime}}}+2C_{p^{\prime}}^{(2)}-2\log(2\pi)-2\gamma+\frac{2^{5/2}\theta e^{-\pi\sqrt{c}}}{c^{1/4}}.
\]

Thus, if $p^{\prime}\rightarrow\infty$ we see from the previous expression
that
\[
\zeta_{p,\infty}\left(\frac{1}{2},c\right)=2C_{p}^{(1)}+\log\left(\frac{c}{4}\right)-2\log(2\pi)+\theta\,\frac{2^{5/2}e^{-\pi\sqrt{c}}}{c^{1/4}}>0
\]
provided $c\geq c_{0}>16\pi^{2}e^{-2C_{p}^{(1)}}$ and $c_{0}$ large
enough. By (\ref{residue First analogue Epstein}), we know that $\lim_{s\rightarrow1^{-}}\zeta_{p,p^{\prime}}(s,c)=-\pi/\sqrt{c}<0$
and so an immediate application of the intermediate value theorem
shows that there exists some $\sigma_{0}\in\left(\frac{1}{2},\,1\right)$
for which $\zeta_{p,\infty}(\sigma_{0},c)=0$. 
\end{proof}

\subsection{A second Analogue of Epstein's zeta function}
Let $c>0$, $p,\,p^{\prime}>0$ and let $\lambda_{n},\,\lambda_{n}^{\prime}$
be the Koshliakov sequences satisfying the transcendental equations (\ref{beginning of statement transcendental}). For $\text{Re}(s)>1$, we define the second analogue of Epstein's
zeta function as 
\begin{equation}
\tilde{\zeta}_{p,p^{\prime}}(s,c):=\frac{2c^{-s}}{1+\frac{1}{\pi p}}\,\zeta_{p^{\prime}}(2s)+\frac{2}{1+\frac{1}{\pi p^{\prime}}}\,\eta_{p}(2s)+ \frac{8\,\pi^{s+1}c^{\frac{1}{4}-\frac{s}{2}}}{\Gamma(s)}\,\sum_{n=1}^{\infty} \frac{p^{\prime2}+\lambda_{n}^{\prime2}}{p^{\prime}\left(p^{\prime}+\frac{1}{\pi}\right)+\lambda_{n}^{\prime2}}\,\lambda_{n}^{\prime\frac{1}{2}-s}\,\intop_{0}^{\infty}\,\frac{y^{s-\frac{1}{2}}J_{s-\frac{1}{2}}\left(2\pi\lambda_{n}^{\prime}\sqrt{c}\,y\right)}{\sigma\left(y\right)e^{2\pi y}-1}\,dy.\label{second analogue definition definition}
\end{equation}

Although this definition seems to have a drastically different aspect
to the classical Epstein zeta function (\ref{Epstein def}), in view of the second analogue of Watson's formula (\ref{second analogue Watson}), we may expect to recover the classical cases
when $p,p^{\prime}$ tend to zero of infinity. In the next lemma we check that this is the case.

\begin{lemma}
Let $\text{Re}(s)>1$ and $c>0$. Then we have the limiting cases
\begin{equation}
\lim_{p,p^{\prime}\rightarrow\infty}\tilde{\zeta}_{p,p^{\prime}}(s,c)=\sum_{m,n\neq0}\frac{1}{\left(m^{2}+cn^{2}\right)^{s}}\label{first case second}
\end{equation}
and
\begin{equation}
\lim_{p,p^{\prime}\rightarrow0^{+}}\tilde{\zeta}_{p,p^{\prime}}(s,c)=\sum_{m\ne0,n\neq0}\frac{(-1)^{m}}{\left(m^{2}+c\left(n-\frac{1}{2}\right)^{2}\right)^{s}}.\label{second case second}
\end{equation}

Analogously,
\begin{equation}
\lim_{p\rightarrow0^{+},p^{\prime}\rightarrow\infty}\tilde{\zeta}_{p,p^{\prime}}(s,c)=\sum_{m\neq0,\,n\in\mathbb{Z}}\frac{(-1)^{m}}{\left(m^{2}+c\,n^{2}\right)^{s}},\label{thrid case secccond th}
\end{equation}
\begin{equation}
\lim_{p\rightarrow\infty,p^{\prime}\rightarrow0^{+}}\tilde{\zeta}_{p,p^{\prime}}(s,c)=\sum_{m\in\mathbb{Z},n\neq0}\frac{1}{\left(m^{2}+c\,\left(n-\frac{1}{2}\right)^{2}\right)^{s}}.\label{fiourth caase second analouge}
\end{equation}
\end{lemma}

\begin{proof}
It suffices to prove the case (\ref{first case second}), the remaining ones being analogous. By definition (\ref{second analogue definition definition})
and by appealing to the steps given in (\ref{set of bounds justification})
and using once more the Laplace representation (\ref{Integral representation Bessel laplace}),
we obtain after some straightforward simplifications,
\begin{align}
\lim_{p,p^{\prime}\rightarrow\infty}\tilde{\zeta}_{p,p^{\prime}}(s,c) & :=2c^{-s}\,\zeta(2s)+2\,\zeta(2s)+\frac{8\,\pi^{s+1}c^{\frac{1}{4}-\frac{s}{2}}}{\Gamma(s)}\,\sum_{n=1}^{\infty}n^{\frac{1}{2}-s}\,\intop_{0}^{\infty}\,\frac{y^{s-\frac{1}{2}}J_{s-\frac{1}{2}}\left(2\pi n\,\sqrt{c}\,y\right)}{e^{2\pi y}-1}\,dy\nonumber \\
 & =2c^{-s}\,\zeta(2s)+2\,\zeta(2s)+\frac{8\,\pi^{s+1}c^{\frac{1}{4}-\frac{s}{2}}}{\Gamma(s)}\,\sum_{m,n=1}^{\infty}n^{\frac{1}{2}-s}\,\intop_{0}^{\infty}y^{s-\frac{1}{2}}J_{s-\frac{1}{2}}\left(2\pi n\,\sqrt{c}\,y\right)\,e^{-2\pi my}dy\nonumber \\
 & =2c^{-s}\,\zeta(2s)+2\,\zeta(2s)+4\,\sum_{m,n=1}^{\infty}\frac{1}{\left(m^{2}+cn^{2}\right)^{s}}=\sum_{m,n\neq0}\frac{1}{\left(m^{2}+cn^{2}\right)^{s}},\,\,\,\,\,\text{Re}(s)>1.\label{computation limiting cases!}
\end{align}

\end{proof}

Like in the previous section, we study the analytic continuation
of the Epstein zeta function (\ref{second analogue definition definition}).
It is also apparent from the limiting cases (\ref{first case second})-(\ref{fiourth caase second analouge})
that the analytic continuation of (\ref{second analogue definition definition})
should be similar to the analytic continuation of the Classical Epstein
zeta function (\ref{Epstein def}).

\bigskip{}

Our next result gives another generalization of the Selberg-Chowla
formula (\ref{Selberg Chowla Formula}) and the analytic continuation
of $\tilde{\zeta}_{p,p^{\prime}}(s,c)$.

\begin{theorem}
For every $\text{Re}(s)>1$, the following generalization of the Selberg-Chowla
formula takes place
\begin{align}
\tilde{\zeta}_{p,p^{\prime}}(s,c) & =\frac{2}{1+\frac{1}{\pi p^{\prime}}}\,\eta_{p}(2s)+\frac{2\sqrt{\pi}\,c^{\frac{1}{2}-s}}{\Gamma(s)}\cdot\frac{\Gamma\left(s-\frac{1}{2}\right)}{1+\frac{1}{\pi p}}\,\zeta_{p^{\prime}}(2s-1)+\nonumber \\
+\frac{8\,\pi^{s}c^{\frac{1}{4}-\frac{s}{2}}}{\Gamma(s)} & \sum_{m,n=1}^{\infty}\frac{p^{2}+\lambda_{m}^{2}}{p\left(p+\frac{1}{\pi}\right)+\lambda_{m}^{2}}\,\frac{p^{\prime2}+\lambda_{n}^{\prime2}}{p^{\prime}\left(p^{\prime}+\frac{1}{\pi}\right)+\lambda_{n}^{\prime2}}\,\left(\frac{\lambda_{m}}{\lambda_{n}^{\prime}}\right)^{s-\frac{1}{2}}\,K_{s-\frac{1}{2}}\left(2\pi\sqrt{c}\lambda_{m}\lambda_{n}^{\prime}\right).\label{Selberg-Chowla for second Epstein}
\end{align}

Furthermore, (\ref{Selberg-Chowla for second Epstein}) provides the
continuation of $\tilde{\zeta}_{p,p^{\prime}}(s,c)$ as an
analytic function for every $s\in\mathbb{C}$ except at a simple pole
located at $s=1$, whose residue is
\begin{equation}
\text{Res}_{s=1}\tilde{\zeta}_{p,p^{\prime}}(s,c)=\frac{\pi}{\sqrt{c}\left(1+\frac{1}{\pi p}\right)}.\label{Residue pole tilde zeta Koshliakov Epstein}
\end{equation}
\end{theorem}

\begin{proof}
The proof of (\ref{Selberg-Chowla for second Epstein}) comes immediately
from our second analogue of Watson's formula (\ref{second analogue Watson}).
Indeed, if we assume that $\text{Re}(s)>1$ and apply (\ref{second analogue Watson})
to the third term of (\ref{second analogue definition definition}),
we obtain
\begin{align*}
&\tilde{\zeta}_{p,p^{\prime}}(s,c):=\frac{2c^{-s}}{1+\frac{1}{\pi p}}\,\zeta_{p^{\prime}}(2s)+2\eta_{p}(2s)+\frac{8\,\pi^{s}c^{\frac{1}{4}-\frac{s}{2}}}{\Gamma(s)}\,\sum_{m,n=1}^{\infty}\frac{p^{2}+\lambda_{m}^{2}}{p\left(p+\frac{1}{\pi}\right)+\lambda_{m}^{2}}\cdot\frac{p^{\prime2}+\lambda_{n}^{\prime2}}{p^{\prime}\left(p^{\prime}+\frac{1}{\pi}\right)+\lambda_{n}^{\prime2}}\,\left(\frac{\lambda_{m}}{\lambda_{n}^{\prime}}\right)^{s-\frac{1}{2}}\,K_{s-\frac{1}{2}}\left(2\pi\sqrt{c}\lambda_{m}\lambda_{n}^{\prime}\right)\\
 & +\frac{8\pi^{s}c^{1/4-\frac{s}{2}}}{\Gamma(s)}\,\left\{ \frac{\pi^{\frac{1}{2}-s}\left(\sqrt{c}\right)^{\frac{1}{2}-s}}{4}\,\frac{\Gamma\left(s-\frac{1}{2}\right)}{1+\frac{1}{\pi p}}\,\zeta_{p^{\prime}}(2s-1)-\frac{\Gamma(s)\pi^{-s}c^{-\frac{s}{2}-\frac{1}{4}}}{4}\zeta_{p^{\prime}}(2s)\right\} =2\eta_{p}(2s)+\frac{2\sqrt{\pi}\,c^{\frac{1}{2}-s}\Gamma\left(s-\frac{1}{2}\right)\,\zeta_{p^{\prime}}(2s-1)}{\Gamma(s)(1+\frac{1}{\pi p})}+\\
 &+\frac{8\,\pi^{s}c^{\frac{1}{4}-\frac{s}{2}}}{\Gamma(s)}\,\sum_{m,n=1}^{\infty}\frac{p^{2}+\lambda_{m}^{2}}{p\left(p+\frac{1}{\pi}\right)+\lambda_{m}^{2}}\,\frac{p^{\prime2}+\lambda_{n}^{\prime2}}{p^{\prime}\left(p^{\prime}+\frac{1}{\pi}\right)+\lambda_{n}^{\prime2}}\,\left(\frac{\lambda_{m}}{\lambda_{n}^{\prime}}\right)^{s-\frac{1}{2}}\,K_{s-\frac{1}{2}}\left(2\pi\sqrt{c}\lambda_{m}\lambda_{n}^{\prime}\right),
\end{align*}
which proves (\ref{Selberg-Chowla for second Epstein}). To deduce
the analytic continuation of $\tilde{\zeta}_{p,p^{\prime}}(s,c)$
from this, we recall that proving that
\begin{equation}
\tilde{H}_{p,p^{\prime}}(s,c):=\sum_{m,n=1}^{\infty}\frac{p^{2}+\lambda_{m}^{2}}{p\left(p+\frac{1}{\pi}\right)+\lambda_{m}^{2}}\,\frac{p^{\prime2}+\lambda_{n}^{\prime2}}{p^{\prime}\left(p^{\prime}+\frac{1}{\pi}\right)+\lambda_{n}^{\prime2}}\,\left(\frac{\lambda_{m}}{\lambda_{n}^{\prime}}\right)^{s-\frac{1}{2}}\,K_{s-\frac{1}{2}}\left(2\pi\sqrt{c}\lambda_{m}\lambda_{n}^{\prime}\right)\label{definition tilde H entire argument}
\end{equation}
is entire requires no more than repeating the argument given in {[}\cite{Bateman_Epstein},
p. 368, Lemma 2{]}. Therefore, any potential meromorphic behavior of $\tilde{\zeta}_{p,p^{\prime}}(s,c)$ is due to the function
\[
\tilde{G}_{p,p^{\prime}}(s,c):=\frac{2}{1+\frac{1}{\pi p^{\prime}}}\,\eta_{p}(2s)+\frac{2\sqrt{\pi}\,c^{\frac{1}{2}-s}}{\Gamma(s)}\cdot\frac{\Gamma\left(s-\frac{1}{2}\right)}{1+\frac{1}{\pi p}}\,\zeta_{p^{\prime}}(2s-1).
\]

It is straightforward to check that $\tilde{G}_{p,p^{\prime}}(s,c)$
has removable singularities at $s=\frac{1}{2}-k,$ $k\in\mathbb{N}_{0}$.
Since $\eta_{p}(2s)$ and $\Gamma\left(s-\frac{1}{2}\right)$ are
analytic in a neighborhood of $s=1$, we conclude that $\tilde{\zeta}_{p,p^{\prime}}(s,c)$
must have a simple pole at $s=1$ (coming from the contribution of $\zeta_{p^{\prime}}(2s-1)$)
with residue explicitly given by (\ref{Residue pole tilde zeta Koshliakov Epstein}).
\end{proof}

Like in the previous section, we may get a generalization of Kronecker's
limit formula (\ref{Kronecker limit formula diagonal form}) in another
direction. The following corollary provides such generalization.

\begin{corollary}
Let $p,\,p^{\prime}\in\mathbb{R}_{+}$, $c>0$ and let $\tilde{\zeta}_{p,p^{\prime}}(s,c)$
be the generalized Epstein zeta function (\ref{definition Epstein in statement}).
Then $\tilde{\zeta}_{p,p^{\prime}}(s,c)$ admits the following meromorphic expansion
around $s=1$,  
\begin{align}
\zeta_{p,p^{\prime}}(s,\,c) & =\frac{\pi}{\sqrt{c}\left(1+\frac{1}{\pi p}\right)}\,\frac{1}{s-1}+\frac{2}{1+\frac{1}{\pi p^{\prime}}}\,\eta_{p}(2)+\nonumber \\
+\frac{\pi}{\sqrt{c}} & \left(\frac{2C_{p^{\prime}}^{(1)}-\log\left(4c\right)}{1+\frac{1}{\pi p}}+4\,\sum_{n=1}^{\infty}\frac{p^{\prime2}+\lambda_{n}^{\prime2}}{p^{\prime}\left(p^{\prime}+\frac{1}{\pi}\right)+\lambda_{n}^{\prime2}}\,\lambda_{n}^{\prime-1}\,\sigma_{p}\left(2\pi\sqrt{c}\lambda_{n}^{\prime}\right)\right)+O(s-1),\label{Kronecker limit formula for second Epstein zeta function}
\end{align}
where $\sigma_{p}(z)$ is defined by (\ref{sigma p definition on Koshliakov}).
\end{corollary}

\begin{proof}
Like in the proof of Corollary 4.1., we appeal to the Laurent expansions
(\ref{First meromorphic Kronecker limit}) and (\ref{Second meromorphic Kronecker limit}), which show that
\begin{equation}
\frac{2\sqrt{\pi}\,c^{\frac{1}{2}-s}}{\Gamma(s)}\cdot\frac{\Gamma\left(s-\frac{1}{2}\right)}{1+\frac{1}{\pi p}}\,\zeta_{p^{\prime}}(2s-1)=\frac{\pi}{\sqrt{c}\left(1+\frac{1}{\pi p}\right)}\,\frac{1}{s-1}+\frac{\pi}{\sqrt{c}\left(1+\frac{1}{\pi p}\right)}\left(2C_{p^{\prime}}^{(1)}-\log\left(4c\right)\right)+O\left(s-1\right).\label{second meromorphic expansion}
\end{equation}

Therefore, using the generalized Selberg-Chowla formula (\ref{Selberg-Chowla for second Epstein})
and (\ref{particular case of bessel!}) and using the limit $s\rightarrow1$,
we conclude
\begin{align*}
\tilde{\zeta}_{p,p^{\prime}}(s,c)=\frac{2}{1+\frac{1}{\pi p^{\prime}}}\,\eta_{p}(2)+\frac{1}{1+\frac{1}{\pi p}}\left\{ \frac{\pi}{\sqrt{c}}\,\frac{1}{s-1}+\frac{\pi}{\sqrt{c}}\left(2C_{p^{\prime}}^{(1)}-\log\left(4c\right)\right)\right\} \\
+\frac{4\pi}{\sqrt{c}}\,\sum_{m,n=1}^{\infty}\frac{\left(p^{2}+\lambda_{m}^{2}\right)\left(p^{\prime2}+\lambda_{n}^{\prime2}\right)\,\lambda_{n}^{\prime-1}}{\left(p\left(p+\frac{1}{\pi}\right)+\lambda_{m}^{2}\right)\left(p^{\prime}\left(p^{\prime}+\frac{1}{\pi}\right)+\lambda_{n}^{\prime2}\right)}\cdot e^{-2\pi\sqrt{c}\lambda_{m}\lambda_{n}^{\prime}}+O\left(s-1\right).
\end{align*}

One finally obtains (\ref{Kronecker limit formula for second Epstein zeta function}) after using the definition of $\sigma_{p}(z)$ (\ref{sigma p definition on Koshliakov}).  
\end{proof}

\bigskip{}

We may use the previous result to derive more analogues of Corollaries 4.2, 4.3., 4.4. and 4.5. Unlike the previous analogue of the Epstein zeta function, the symmetries
provided by the Selberg-Chowla formula (\ref{Selberg-Chowla for second Epstein})
can be used to derive (as an easy consequence) a symmetric functional equation
for $\tilde{\zeta}_{p,p^{\prime}}(s,c)$.
\begin{corollary}
For every $s\in\mathbb{C}$, the following functional equation holds
\begin{equation}
\left(\frac{\pi}{\sqrt{c}}\right)^{-s}\Gamma(s)\,\tilde{\zeta}_{p,p^{\prime}}\left(s,c\right)=\left(\frac{\pi}{\sqrt{c}}\right)^{-(1-s)}\Gamma\left(1-s\right)\,\tilde{\zeta}_{p^{\prime},p}\left(1-s,c\right).\label{Functional equation Epstein second analogueeee!}
\end{equation}
\end{corollary}

\begin{proof}
Indeed, using the Selberg-Chowla formula (\ref{Selberg-Chowla for second Epstein})
with $s$ being replaced by $1-s$, we are able to see that
\begin{align}
\tilde{\zeta}_{p,p^{\prime}}\left(1-s,c\right) & =\frac{2}{1+\frac{1}{\pi p^{\prime}}}\,\eta_{p}(2-2s)+\frac{2\sqrt{\pi}\,c^{s-\frac{1}{2}}}{\Gamma(1-s)}\cdot\frac{\Gamma\left(\frac{1}{2}-s\right)}{1+\frac{1}{\pi p}}\,\zeta_{p^{\prime}}(1-2s)+\nonumber \\
+\frac{8\,\pi^{1-s}c^{\frac{s}{2}-\frac{1}{4}}}{\Gamma(1-s)} & \sum_{m,n=1}^{\infty}\frac{p^{2}+\lambda_{m}^{2}}{p\left(p+\frac{1}{\pi}\right)+\lambda_{m}^{2}}\,\frac{p^{\prime2}+\lambda_{n}^{\prime2}}{p^{\prime}\left(p^{\prime}+\frac{1}{\pi}\right)+\lambda_{n}^{\prime2}}\,\left(\frac{\lambda_{n}^{\prime}}{\lambda_{m}}\right)^{s-\frac{1}{2}}\,K_{s-\frac{1}{2}}\left(2\pi\sqrt{c}\,\lambda_{m}\lambda_{n}^{\prime}\right).\label{Once again Selberg Chowla for second ananananaanlogue}
\end{align}

Assume now that $\text{Re}(s)>1$: if we start by evaluating the Bessel
series on the right-hand side of (\ref{Once again Selberg Chowla for second ananananaanlogue})
by summing first with respect to the variable of summation $n$, we
get (according to the second analogue of Watson's formula (\ref{second analogue Watson})),
\begin{align}
\sum_{n=1}^{\infty}\frac{p^{\prime2}+\lambda_{n}^{\prime2}}{p^{\prime}\left(p^{\prime}+\frac{1}{\pi}\right)+\lambda_{n}^{\prime2}}\,\lambda_{n}^{\prime s-\frac{1}{2}}\,K_{s-\frac{1}{2}}\left(2\pi\lambda_{m}\sqrt{c}\lambda_{n}^{\prime}\right) & =\nonumber \\
=-\frac{\pi^{\frac{1}{2}-s}(\lambda_{m}\sqrt{c})^{\frac{1}{2}-s}}{4}\,\frac{\Gamma\left(s-\frac{1}{2}\right)}{1+\frac{1}{\pi p^{\prime}}}+\frac{\Gamma(s)\pi^{-s}\left(\lambda_{m}\sqrt{c}\right)^{-s-\frac{1}{2}}}{4}+\pi\,\intop_{0}^{\infty}\,\frac{y^{s-\frac{1}{2}}J_{s-\frac{1}{2}}\left(2\pi\lambda_{m}\sqrt{c}y\right)}{\sigma^{\prime}\left(y\right)e^{2\pi y}-1}\,dy.\label{sum sum sum sum sum}
\end{align}

Sum now the right-hand side of the previous equality with respect to $m$ and use the hypothesis $\text{Re}(s)>1$ to obtain (here $\mathcal{A}:=\tilde{\zeta}_{p,p^{\prime}}\left(1-s,c\right)$)
\begin{align*}
\mathcal{A} & =\frac{2\eta_{p}(2-2s)}{1+\frac{1}{\pi p^{\prime}}}+\frac{2\sqrt{\pi}\,c^{s-\frac{1}{2}}}{\Gamma(1-s)}\cdot\frac{\Gamma\left(\frac{1}{2}-s\right)\zeta_{p^{\prime}}(1-2s)}{1+\frac{1}{\pi p}}+\frac{8\,\pi^{1-s}c^{\frac{s}{2}-\frac{1}{4}}}{\Gamma(1-s)}\left[\frac{\Gamma(s)\pi^{-s}c^{-\frac{s}{2}-\frac{1}{4}}\zeta_{p}(2s)}{4}-\frac{\pi^{\frac{1}{2}-s}\Gamma\left(s-\frac{1}{2}\right)c^{\frac{1}{4}-\frac{s}{2}}\zeta_{p}(2s-1)}{4\left(1+\frac{1}{\pi p^{\prime}}\right)}\right]\\
 & +\frac{8\,\pi^{2-s}c^{\frac{s}{2}-\frac{1}{4}}}{\Gamma(1-s)}\,\sum_{m=1}^{\infty}\frac{p^{2}+\lambda_{m}^{2}}{p\left(p+\frac{1}{\pi}\right)+\lambda_{m}^{2}}\lambda_{m}^{\frac{1}{2}-s}\,\intop_{0}^{\infty}\,\frac{y^{s-\frac{1}{2}}J_{s-\frac{1}{2}}\left(2\pi\lambda_{m}\sqrt{c}y\right)}{\sigma^{\prime}\left(y\right)e^{2\pi y}-1}\,dy.
\end{align*}

Appealing to the functional equation for $\zeta_{p}(2s-1)$, (\ref{functional equation Kosh}),
we are able to simplify the previous expression to:
\begin{align}
\tilde{\zeta}_{p,p^{\prime}}\left(1-s,c\right) & =\frac{2\sqrt{\pi}\,c^{s-\frac{1}{2}}}{\Gamma(1-s)}\cdot\frac{\Gamma\left(\frac{1}{2}-s\right)}{1+\frac{1}{\pi p}}\,\zeta_{p^{\prime}}(1-2s)+\frac{8\,\pi^{1-s}c^{\frac{s}{2}-\frac{1}{4}}}{\Gamma(1-s)}\times\nonumber \\
\times & \left[\frac{\Gamma(s)\pi^{-s}c^{-\frac{s}{2}-\frac{1}{4}}}{4}\zeta_{p}(2s)+\pi\sum_{m=1}^{\infty}\frac{p^{2}+\lambda_{m}^{2}}{p\left(p+\frac{1}{\pi}\right)+\lambda_{m}^{2}}\lambda_{m}^{\frac{1}{2}-s}\,\intop_{0}^{\infty}\,\frac{y^{s-\frac{1}{2}}J_{s-\frac{1}{2}}\left(2\pi\lambda_{m}\sqrt{c}y\right)}{\sigma^{\prime}\left(y\right)e^{2\pi y}-1}\,dy\right].\label{one of the final expressions to deduce!}
\end{align}

Henceforth, we see that (\ref{one of the final expressions to deduce!})
gives
\begin{align}
\left(\frac{\pi}{\sqrt{c}}\right)^{-(1-s)}\Gamma\left(1-s\right)\,\tilde{\zeta}_{p,p^{\prime}}\left(1-s,c\right) & =2\pi^{s-\frac{1}{2}}\,c^{\frac{s}{2}}\cdot\frac{\Gamma\left(\frac{1}{2}-s\right)}{1+\frac{1}{\pi p}}\,\zeta_{p^{\prime}}(1-2s)+2\Gamma(s)\pi^{-s}c^{-\frac{s}{2}}\,\zeta_{p}(2s)\nonumber \\
+8\pi c^{\frac{1}{4}}\sum_{m=1}^{\infty}\frac{p^{2}+\lambda_{m}^{2}}{p\left(p+\frac{1}{\pi}\right)+\lambda_{m}^{2}} & \lambda_{m}^{\frac{1}{2}-s}\,\intop_{0}^{\infty}\,\frac{y^{s-\frac{1}{2}}J_{s-\frac{1}{2}}\left(2\pi\lambda_{m}\sqrt{c}y\right)}{\sigma^{\prime}\left(y\right)e^{2\pi y}-1}\,dy.\label{after several simplifications in the end of proof!}
\end{align}

Invoking now the definition of the second analogue of Epstein's zeta function, (\ref{second analogue definition definition}) (with $p$ and $p^{\prime}$ being switched), we are also able to check that
\begin{align}
\left(\frac{\pi}{\sqrt{c}}\right)^{-s}\Gamma(s)\,\tilde{\zeta}_{p^{\prime},p}(s,c) & :=\frac{2c^{-s/2}\,\pi^{-s}\Gamma(s)\,\zeta_{p}(2s)}{1+\frac{1}{\pi p^{\prime}}}+\frac{2c^{s/2}\,\pi^{-s}\Gamma(s)\eta_{p^{\prime}}(2s)}{1+\frac{1}{\pi p}}+\nonumber \\
 & +8\pi c^{\frac{1}{4}}\,\sum_{m=1}^{\infty}\frac{p^{2}+\lambda_{m}^{2}}{p\left(p+\frac{1}{\pi}\right)+\lambda_{m}^{2}}\lambda_{m}^{\frac{1}{2}-s}\,\intop_{0}^{\infty}\,\frac{y^{s-\frac{1}{2}}J_{s-\frac{1}{2}}\left(2\pi\lambda_{m}\sqrt{c}y\right)}{\sigma^{\prime}\left(y\right)e^{2\pi y}-1}\label{other expression deeeefinition!}
\end{align}
for every $\text{Re}(s)>1$. But the right-hand sides of (\ref{one of the final expressions to deduce!})
and (\ref{other expression deeeefinition!}) coincide, due to the
functional equation for $\zeta_{p^{\prime}}(1-2s)$. Therefore, equating
both sides of (\ref{one of the final expressions to deduce!}) and
(\ref{other expression deeeefinition!}) (and replacing $p$ by $p^{\prime}$)
gives (\ref{Functional equation Epstein second analogueeee!}).
\end{proof}

As a particular case of the previous theorem, we can write a simple
functional equation for the particular case of (\ref{definition Epstein in statement})
studied in Corollary 4.4.

\begin{corollary}
For $c>0$, let $\zeta_{p,p^{\prime}}(s,c)$ be the first analogue
of Epstein's zeta function (\ref{definition Epstein in statement}).
Furthermore, for every $p^{\prime}\in\mathbb{R}_{+}$, let
\[
\zeta_{\infty,p^{\prime}}(s,c):=\sum_{m,n\neq0}\frac{p^{\prime2}+\lambda_{n}^{\prime2}}{p^{\prime}(p^{\prime}+\frac{1}{\pi})+\lambda_{n}^{\prime2}}\,\frac{1}{\left(m^{2}+c\lambda_{n}^{\prime2}\right)^{s}},\,\,\,\,\text{Re}(s)>1.
\]

Then the analytic continuation of $\zeta_{\infty,p^{\prime}}(s,c)$
satisfies the functional equation
\begin{equation}
\left(\frac{\pi}{\sqrt{c}}\right)^{-s}\Gamma(s)\,\zeta_{\infty,p^{\prime}}\left(s,c\right)=\left(\frac{\pi}{\sqrt{c}}\right)^{-(1-s)}\Gamma\left(1-s\right)\,\tilde{\zeta}_{p^{\prime},\infty}\left(1-s,c\right).\label{Functional equation particular case}
\end{equation}
\end{corollary}

\begin{proof}
Indeed, by letting $p\rightarrow\infty$, we see from a simple adaptation
of (\ref{computation limiting cases!}) that
\[
\tilde{\zeta}_{\infty,p^{\prime}}(s,c):=\lim_{p\rightarrow\infty}\tilde{\zeta}_{p,p^{\prime}}(s,c)=\sum_{m,n\neq0}\frac{p^{\prime2}+\lambda_{n}^{\prime2}}{p^{\prime}\left(p^{\prime}+\frac{1}{\pi}\right)+\lambda_{n}^{\prime2}}\frac{1}{\left(m^{2}+c\lambda_{n}^{\prime2}\right)^{s}}=\zeta_{\infty,p^{\prime}}(s,c),\,\,\,\,\,\text{Re}(s)>1.
\]

Therefore, (\ref{Functional equation particular case}) follows from (\ref{Functional equation Epstein second analogueeee!}). 
\end{proof}

We end this section by establishing an analogue of Corollary 4.6.
for $\tilde{\zeta}_{p,p^{\prime}}(s,c)$.
\begin{corollary}
Let $\sigma(t)$ be defined by (\ref{definition sigma p sigma p'}). Then the following identity takes place,
\begin{align}
\tilde{\zeta}_{p,p^{\prime}}\left(\frac{1}{2},c\right) & =\frac{2C_{p}^{(2)}-\gamma-\log\left(\frac{4\pi}{\sqrt{c}}\right)}{1+\frac{1}{\pi p^{\prime}}}+\frac{2C_{p^{\prime}}^{(2)}-2\log(2\pi)-2\gamma}{1+\frac{1}{\pi p}}+\frac{\log(\pi\sqrt{c})+\gamma-4e^{2\pi p}\,Q_{2\pi p}(0)-4e^{2\pi p^{\prime}}Q_{2\pi p^{\prime}}(0)}{\left(1+\frac{1}{\pi p^{\prime}}\right)\left(1+\frac{1}{\pi p}\right)}\nonumber \\
 & +8\,\sum_{m,n=1}^{\infty}\frac{p^{2}+\lambda_{m}^{2}}{p\left(p+\frac{1}{\pi}\right)+\lambda_{m}^{2}}\cdot\frac{p^{\prime2}+\lambda_{n}^{\prime2}}{p^{\prime}\left(p^{\prime}+\frac{1}{\pi}\right)+\lambda_{n}^{\prime2}}\,K_{0}\left(2\pi\sqrt{c}\lambda_{m}\lambda_{n}^{\prime}\right),\label{identity central value for second analogue Epstein}
\end{align}
where $C^{(2)}_{p}$ and $Q_{\mu}(s)$ are respectively defined by (\ref{second analogue Euler Mascheroni}) and (\ref{Incomplete Gamma function Kosh sense}).
\end{corollary}

\begin{proof}
Using the second Selberg-Chowla formula (\ref{Selberg-Chowla for second Epstein}),
we find that, in an analogous way to (\ref{Simple expression at beginning central value}),
\begin{align*}
\tilde{\zeta}_{p,p^{\prime}}\left(\frac{1}{2},c\right) & =c^{-1/4}\,\lim_{s\rightarrow\frac{1}{2}}\left\{ \frac{2c^{s/2}}{1+\frac{1}{\pi p^{\prime}}}\,\pi^{-s}\,\Gamma(s)\,\eta_{p}(2s)+\frac{2c^{\frac{1-s}{2}}}{1+\frac{1}{\pi p}}\,\pi^{-\left(s-\frac{1}{2}\right)}\Gamma\left(s-\frac{1}{2}\right)\,\zeta_{p^{\prime}}(2s-1)\right\} \\
 & +8\,\sum_{m,n=1}^{\infty}\frac{p^{2}+\lambda_{m}^{2}}{p\left(p+\frac{1}{\pi}\right)+\lambda_{m}^{2}}\cdot\frac{p^{\prime2}+\lambda_{n}^{\prime2}}{p^{\prime}\left(p^{\prime}+\frac{1}{\pi}\right)+\lambda_{n}^{\prime2}}\,K_{0}\left(2\pi\sqrt{c}\lambda_{m}\lambda_{n}^{\prime}\right).
\end{align*}

We have seen already (c.f. (\ref{Simplifications second term central value})
above) that the second term on the braces has the Laurent expansion
\begin{align}
\frac{2c^{\frac{1-s}{2}}}{1+\frac{1}{\pi p}}\,\pi^{-\left(s-\frac{1}{2}\right)}\Gamma\left(s-\frac{1}{2}\right)\,\zeta_{p^{\prime}}(2s-1)=\nonumber \\
=\frac{2c^{1/4}}{1+\frac{1}{\pi p}}\left\{ \frac{\zeta_{p^{\prime}}(0)}{s-\frac{1}{2}}+2\zeta_{p^{\prime}}^{\prime}(0)-\gamma\zeta_{p^{\prime}}(0)-\log(\pi\sqrt{c})\,\zeta_{p^{\prime}}(0)+O\left(s-\frac{1}{2}\right)\right\} ,\label{first laurent expansion central value second analogue}
\end{align}
with $\zeta_{p^{\prime}}^{\prime}(0)$ being given in the expression
(\ref{at beginning but will resurface on Koshliakoc}). Therefore,
we just need to find the Laurent expansion of the first term in the
braces. Invoking (\ref{expansion gamma function atound 1/2 onde mocre})
and {[}\cite{KOSHLIAKOV}, p. 49, eq. (53){]},
\begin{equation}
\eta_{p}(2s)=\frac{1}{1+\frac{1}{\pi p}}\cdot\frac{1}{2s-1}+C_{p}^{(2)}-\frac{2e^{2\pi p}}{1+\frac{1}{\pi p}}Q_{2\pi p}(0)+O\left(s-\frac{1}{2}\right),\label{Expansion around 1/2 Eta second Kosliakov}
\end{equation}
we find, just like in (\ref{Simplification first term zeta (1/2) some details}),
\begin{equation}
\left(\frac{\pi}{\sqrt{c}}\right)^{-s}\Gamma(s)\,\eta_{p}(2s)=c^{1/4}\left[\frac{1}{1+\frac{1}{\pi p}}\cdot\frac{1}{2s-1}+C_{p}^{(2)}-\frac{2e^{2\pi p}}{1+\frac{1}{\pi p}}Q_{2\pi p}(0)-\frac{1}{2}\left(\gamma+\log\left(\frac{4\pi}{\sqrt{c}}\right)\right)+O\left(s-\frac{1}{2}\right)\right].\label{first term simplifying second analogue Epstein}
\end{equation}

By combining (\ref{first laurent expansion central value second analogue})
and (\ref{first term simplifying second analogue Epstein}), (\ref{identity central value for second analogue Epstein}) follows.
\end{proof}

\section{Generalizations of Entries 3.3.1 - 3.3.3., pp. 253-254 of Ramanujan's Lost Notebook}

Building on the work of the previous two sections, we now establish
a generalization of the Ramanujan-Guinand formula\footnote{As stated in footnote 1, page 6, of this paper, in this section $\sigma_{k}(n)$ will always denote the usual divisor function and not Koshliakov's function 
 (\ref{sigma p definition on Koshliakov}).} (\ref{Guinand given at intro}), stated at the introduction as "Entry 3.3.1.". However, before stating it, we need a Lemma that motivates the general aspect of the Modified Bessel functions showing up in (\ref{Guinand given at intro}).

\begin{lemma}
For $x>0$ and $\text{Re}(\nu)<\frac{1}{2}$, let
\begin{equation}
\mathcal{K}_{\nu,p}(x)=\intop_{0}^{\infty}\,\frac{y^{-\nu-\frac{1}{2}}(y+1)^{-\nu-\frac{1}{2}}}{\sigma\left(\frac{x}{2\pi}\,(2y+1)\,\right)e^{(2y+1)x}-1}dy,\,\,\,\,\,x>0,\,\,\text{Re}(\nu)<\frac{1}{2},\label{modified Bessel function kosh case}
\end{equation}
where $\sigma(t)$ is defined by (\ref{definition sigma p sigma p'}).

Then one has the limiting cases:
\begin{equation}
\lim_{p\rightarrow\infty}\mathcal{K}_{\nu,p}(x)=\frac{\Gamma\left(\frac{1}{2}-\nu\right)(2x)^{\nu}}{\sqrt{\pi}}\sum_{n=1}^{\infty}n^{\nu}\,K_{\nu}(xn),\label{p infinity limiting case}
\end{equation}
while 
\begin{equation}
\lim_{p\rightarrow0^{+}}\mathcal{K}_{\nu,p}(x)=\frac{\Gamma\left(\frac{1}{2}-\nu\right)(2x)^{\nu}}{\sqrt{\pi}}\sum_{n=1}^{\infty}(-1)^{n}\,n^{\nu}\,K_{\nu}(xn).\label{p zero limiting case almost all}
\end{equation}
\end{lemma}

\begin{proof}
Indeed, by a simple use of the dominated convergence theorem, the
power series (\ref{power series def}), the absolute convergence of
the iteration of the series and the integral for $\text{Re}(\nu)<\frac{1}{2}$
and finally the integral representation (\ref{as laplace traaansform}),
we obtain 
\begin{align*}
\lim_{p\rightarrow\infty}\mathcal{K}_{\nu,p}(x) & =\intop_{0}^{\infty}\,\frac{y^{-\nu-\frac{1}{2}}(y+1)^{-\nu-\frac{1}{2}}}{e^{(2y+1)x}-1}dy=\sum_{n=1}^{\infty}\intop_{0}^{\infty}y^{-\nu-\frac{1}{2}}(y+1)^{-\nu-\frac{1}{2}}\,e^{-(2y+1)xn}dy\\
 & =\frac{\Gamma\left(\frac{1}{2}-\nu\right)(2x)^{\nu}}{\sqrt{\pi}}\sum_{n=1}^{\infty}n^{\nu}\,K_{\nu}(xn),
\end{align*}
while
\begin{align*}
\lim_{p\rightarrow0^{+}}\mathcal{K}_{\nu,p}(x) & =-\intop_{0}^{\infty}\,\frac{y^{-\nu-\frac{1}{2}}(y+1)^{-\nu-\frac{1}{2}}}{e^{(2y+1)x}+1}dy=\sum_{n=1}^{\infty}(-1)^{n}\intop_{0}^{\infty}y^{-\nu-\frac{1}{2}}(y+1)^{-\nu-\frac{1}{2}}\,e^{-(2y+1)xn}dy\\
 & =\frac{\Gamma\left(\frac{1}{2}-\nu\right)(2x)^{\nu}}{\sqrt{\pi}}\sum_{n=1}^{\infty}(-1)^{n}\,n^{\nu}\,K_{\nu}(xn).
\end{align*}
\end{proof}
\begin{theorem}
If $\alpha$ and $\beta$ are positive numbers such that $\alpha\beta=\pi^{2}$
and if $s$ is a complex number such that $\text{Re}(s)<1$, then
the following identity takes place
\begin{align}
\frac{2^{-s}\sqrt{\pi}}{\Gamma\left(\frac{1-s}{2}\right)}\alpha^{\frac{1-s}{2}}\sum_{n=1}^{\infty}\frac{p^{2}+\lambda_{n}^{2}}{p\left(p+\frac{1}{\pi}\right)+\lambda_{n}^{2}}\,\lambda_{n}^{-s}\,\mathcal{K}_{\frac{s}{2},p}\left(2\lambda_{n}\alpha\right)-\frac{2^{-s}\sqrt{\pi}}{\Gamma\left(\frac{1-s}{2}\right)}\beta^{\frac{1-s}{2}}\,\sum_{n=1}^{\infty}\frac{p^{2}+\lambda_{n}^{2}}{p\left(p+\frac{1}{\pi}\right)+\lambda_{n}^{2}}\,\lambda_{n}^{-s}\,\mathcal{K}_{\frac{s}{2},p}\left(2\lambda_{n}\beta\right)\nonumber \\
=\frac{\Gamma\left(-\frac{s}{2}\right)\,\eta_{p}(-s)}{4\left(1+\frac{1}{\pi p}\right)}\left\{ \beta^{\frac{1+s}{2}}-\alpha^{\frac{1+s}{2}}\right\} +\frac{\Gamma\left(\frac{s}{2}\right)\zeta_{p}(s)}{4}\left\{ \beta^{\frac{1-s}{2}}-\alpha^{\frac{1-s}{2}}\right\} ,\label{generalized Guinand with new function}
\end{align}
where $\mathcal{K}_{\nu,p}(x)$ denotes the integral defined by (\ref{modified Bessel function kosh case}). 
\end{theorem}

\begin{proof}
In Theorem 4.1. we have seen a representation (namely (\ref{representation Re(s)<1}))
which extends the Epstein zeta function, $\zeta_{p,p^{\prime}}(s,c)$,
to the half-plane $\text{Re}(s)<1$. It was deduced by applying the first analogue of Watson's
formula to the sum with respect to the variable of summation $m$.
But in the region of absolute convergence of $\zeta_{p,p^{\prime}}(s,c)$
it is possible to reverse the orders of summation: indeed, for $\text{Re}(s)>1$,
\begin{align}
\left(\frac{\pi}{\sqrt{c}}\right)^{-s}\Gamma(s)\,\zeta_{p,p^{\prime}}(s,c) & =\left(\frac{\pi}{\sqrt{1/c}}\right)^{-s}\Gamma(s)\,\sum_{m,n\neq0}\frac{\left(p^{2}+\lambda_{n}^{2}\right)\cdot\left(p^{\prime2}+\lambda_{n}^{\prime2}\right)}{\left(p(p+\frac{1}{\pi})+\lambda_{n}^{2}\right)\cdot\left(p^{\prime}(p^{\prime}+\frac{1}{\pi})+\lambda_{n}^{\prime2}\right)}\,\frac{1}{\left(\lambda_{n}^{\prime2}+\lambda_{m}^{2}c^{-1}\right)^{s}}\nonumber \\
 & =\left(\frac{\pi}{\sqrt{1/c}}\right)^{-s}\Gamma(s)\,\zeta_{p^{\prime},p}\left(s,\frac{1}{c}\right).\label{property in the region of absolute convergence}
\end{align}

Therefore, an alternative representation of $\zeta_{p,p^{\prime}}(s,c)$
can be obtained once we apply Theorem 4.1. to $\zeta_{p^{\prime},p}\left(s,\frac{1}{c}\right)$.
In fact, using (\ref{representation Re(s)<1}) with $c$ being replaced by
$1/c$ and with $p$ being replaced by $p^{\prime}$, we find that
an equivalent continuation of $\zeta_{p,p^{\prime}}\left(s,c\right)$
to the half-plane $\text{Re}(s)<1$ is 
\begin{align}
\left(\frac{\pi}{\sqrt{c}}\right)^{-s}\Gamma(s)\,\zeta_{p,p^{\prime}}(s,\,c) & =\frac{2c^{-s/2}\,\pi^{-s}\Gamma(s)}{1+\frac{1}{\pi p}}\,\zeta_{p^{\prime}}(2s)+2\,c^{\frac{s-1}{2}}\,\pi^{-\left(s-\frac{1}{2}\right)}\Gamma\left(s-\frac{1}{2}\right)\zeta_{p}(2s-1)+\nonumber \\
+\frac{2^{4-2s}\,\pi^{1-s}\,c^{\frac{s-1}{2}}}{\Gamma(1-s)} & \sum_{n=1}^{\infty}\frac{p^{2}+\lambda_{n}^{2}}{p\left(p+\frac{1}{\pi}\right)+\lambda_{n}^{2}}\,\lambda_{n}^{1-2s}\,\intop_{0}^{\infty}\,\frac{y^{-s}(y+1)^{-s}}{\sigma^{\prime}\left(\frac{\lambda_{n}}{\sqrt{c}}\,(2y+1)\,\right)e^{\frac{2\pi\lambda_{n}}{\sqrt{c}}(2y+1)}-1}\,dy.\label{c replaced by 1/c proof}
\end{align}

Since (\ref{representation Re(s)<1}) and (\ref{c replaced by 1/c proof})
both represent the analytic continuation of $\zeta_{p,p^{\prime}}(s,c)$
to the region $\text{Re}(s)<1$, by uniqueness of the continuation we must have 
\begin{align}
\frac{2c^{s/2}\,\pi^{-s}\Gamma(s)}{1+\frac{1}{\pi p^{\prime}}}\,\zeta_{p}(2s)+2\,c^{\frac{1-s}{2}}\,\pi^{-\left(s-\frac{1}{2}\right)}\Gamma\left(s-\frac{1}{2}\right)\zeta_{p^{\prime}}(2s-1)+\nonumber \\
+\frac{2^{4-2s}\,\pi^{1-s}\,c^{\frac{1-s}{2}}}{\Gamma(1-s)}\sum_{n=1}^{\infty}\frac{p^{\prime2}+\lambda_{n}^{\prime2}}{p^{\prime}\left(p^{\prime}+\frac{1}{\pi}\right)+\lambda_{n}^{\prime2}}\,\lambda_{n}^{\prime1-2s}\,\intop_{0}^{\infty}\,\frac{y^{-s}(y+1)^{-s}}{\sigma\left(\lambda_{n}^{\prime}\,(2y+1)\,\sqrt{c}\right)e^{2\pi\lambda_{n}^{\prime}(2y+1)\sqrt{c}}-1}dy\nonumber \\
=\frac{2c^{-s/2}\,\pi^{-s}\Gamma(s)}{1+\frac{1}{\pi p}}\,\zeta_{p^{\prime}}(2s)+2\,c^{\frac{s-1}{2}}\,\pi^{-\left(s-\frac{1}{2}\right)}\Gamma\left(s-\frac{1}{2}\right)\zeta_{p}(2s-1)+\nonumber \\
+\frac{2^{4-2s}\,\pi^{1-s}\,c^{\frac{s-1}{2}}}{\Gamma(1-s)}\sum_{n=1}^{\infty}\frac{p^{2}+\lambda_{n}^{2}}{p\left(p+\frac{1}{\pi}\right)+\lambda_{n}^{2}}\,\lambda_{n}^{1-2s}\,\intop_{0}^{\infty}\,\frac{y^{-s}(y+1)^{-s}}{\sigma^{\prime}\left(\frac{\lambda_{n}}{\sqrt{c}}\,(2y+1)\,\right)e^{\frac{2\pi\lambda_{n}^{\prime}}{\sqrt{c}}(2y+1)}-1}\,dy, & \,\,\text{Re}(s)<1.\label{protoform of Guinand}
\end{align}
Take the substitution $c=\alpha^{2}/\pi^{2}$,
replace $s$ by $\frac{1+s}{2}$ and let $p=p^{\prime}$ in (\ref{protoform of Guinand}): recalling
the condition $\alpha\beta=\pi^{2}$, we may reduce (\ref{protoform of Guinand})
to 
\begin{align}
\frac{\pi^{-\left(s+1\right)}\,\Gamma\left(\frac{s+1}{2}\right)}{4\left(1+\frac{1}{\pi p}\right)}\,\zeta_{p}(s+1)\left\{ \alpha^{\frac{1+s}{2}}-\beta^{\frac{1+s}{2}}\right\} +\frac{1}{4\sqrt{\pi}}\Gamma\left(\frac{s}{2}\right)\zeta_{p}(s)\left\{ \alpha^{\frac{1-s}{2}}-\beta^{\frac{1-s}{2}}\right\} \nonumber \\
=\frac{2^{-s}}{\Gamma\left(\frac{1-s}{2}\right)}\beta^{\frac{1-s}{2}}\,\sum_{n=1}^{\infty}\frac{p^{2}+\lambda_{n}^{2}}{p\left(p+\frac{1}{\pi}\right)+\lambda_{n}^{2}}\,\lambda_{n}^{-s}\,\intop_{0}^{\infty}\,\frac{y^{-\frac{s+1}{2}}(y+1)^{-\frac{s+1}{2}}}{\sigma\left(\frac{\lambda_{n}\beta}{\pi}\,(2y+1)\,\right)e^{2\beta\lambda_{n}(2y+1)}-1}\,dy-\nonumber \\
-\frac{2^{-s}}{\Gamma\left(\frac{1-s}{2}\right)}\alpha^{\frac{1-s}{2}}\sum_{n=1}^{\infty}\frac{p^{2}+\lambda_{n}^{2}}{p\left(p+\frac{1}{\pi}\right)+\lambda_{n}^{2}}\,\lambda_{n}^{-s}\,\intop_{0}^{\infty}\,\frac{y^{-\frac{s+1}{2}}(y+1)^{-\frac{s+1}{2}}}{\sigma\left(\frac{\lambda_{n}\alpha}{\pi}\,(2y+1)\,\right)e^{2\alpha\lambda_{n}(2y+1)}-1}dy\label{guinand almost finished}
\end{align}

Now, (\ref{generalized Guinand with new function}) is finally obtained
by invoking the functional equation to $\zeta_{p}(s+1)$ on the first
term lying on the left of (\ref{guinand almost finished}) and invoking
the definition of $\mathcal{K}_{\nu,p}(x)$ (\ref{modified Bessel function kosh case}).
\end{proof}

\begin{remark}
By taking a proper deformation of the contour defining the integral
(\ref{modified Bessel function kosh case}), it is possible to have
a formula similar to (\ref{generalized Guinand with new function})
and valid for every $s\in\mathbb{C}$. See Remark 3.2. above.
\end{remark}

With a slight change in the sequence of Ramanujan's statements in
\cite{Ramanujan_lost}, we now establish a generalized version of Entry 3.3.3., also known as Koshliakov's
formula (\ref{Koshliakov given intro}).

\begin{corollary}

Let $\alpha$ and $\beta$ denote positive numbers such that $\alpha\beta=\pi^{2}$.
Define
\begin{equation}
\gamma_{p}:=C_{p}^{(2)}+\frac{C_{p}^{(1)}}{1+\frac{1}{\pi p}}-\gamma-\frac{2e^{2\pi p}\,Q_{2\pi p}(0)}{1+\frac{1}{\pi p}}-\log(2\pi)\left(1-\frac{1}{1+\frac{1}{\pi p}}\right),\label{another general Euler Mascheroni}
\end{equation}
where $C_{p}^{(1)}$, $C_{p}^{(2)}$ and $Q_{\mu}(s)$ are respectively given by (\ref{Euler Mascheroni Koshliakov sense}), (\ref{second analogue Euler Mascheroni}) and (\ref{Incomplete Gamma function Kosh sense}). 

\bigskip{}

Then the following generalization of Entry 3.3.3, (\ref{Koshliakov given intro}), holds
\begin{equation}
\sqrt{\alpha}\left(\frac{1}{4}\gamma_{p}-\frac{\log(4\beta)}{4\left(1+\frac{1}{\pi p}\right)}+\sum_{n=1}^{\infty}\frac{p^{2}+\lambda_{n}^{2}}{p\left(p+\frac{1}{\pi}\right)+\lambda_{n}^{2}}\,\mathcal{K}_{0,p}\left(2\alpha\lambda_{n}\right)\right)=\sqrt{\beta}\left(\frac{1}{4}\gamma_{p}-\frac{\log(4\alpha)}{4\left(1+\frac{1}{\pi p}\right)}+\sum_{n=1}^{\infty}\frac{p^{2}+\lambda_{n}^{2}}{p\left(p+\frac{1}{\pi}\right)+\lambda_{n}^{2}}\,\mathcal{K}_{0,p}\left(2\beta\lambda_{n}\right)\right),\label{Koshliakov formula general final version paper}
\end{equation}
where $\mathcal{K}_{\nu,p}(x)$ denotes the integral defined by (\ref{modified Bessel function kosh case}).
\end{corollary}

\begin{proof}
The proof comes from letting $s\rightarrow0$ in the generalized Ramanujan-Guinand's formula
(\ref{generalized Guinand with new function}) and from invoking the Laurent
expansion for $\Gamma(s)$ around $s=0$, (\ref{meromorphic Gamma}), 
together with [\cite{KOSHLIAKOV}, p. 22, pp. 48 and 49, eqs. (52), (54){]}
\begin{equation}
\zeta_{p}(s)=-\frac{1}{2}\,\frac{1}{1+\frac{1}{\pi p}}+\left(\frac{1}{2}C_{p}^{(2)}-\frac{\gamma}{2}-\frac{e^{2\pi p}}{1+\frac{1}{\pi p}}Q_{2\pi p}(0)-\frac{\log(2\pi)}{2}\right)\,s+O\left(s^{2}\right)\label{First Laurent Kosh}
\end{equation}
and
\begin{equation}
\eta_{p}(s)=-\frac{1}{2}+\left(\frac{1}{2}C_{p}^{(1)}-\frac{\gamma}{2}-\frac{\log(2\pi)}{2}\right)\,s+O\left(s^{2}\right).\label{Seocnd Laurent KOSSSH}
\end{equation}

Using these expansions under the limit $s\rightarrow0$ and rearranging
the terms gives (\ref{Koshliakov formula general final version paper}).
\end{proof}

\bigskip{}

By letting $p\rightarrow0^{+}$ or $p\rightarrow\infty$, we are able
to deduce formulas akin to the classical ones, because in these cases
the integral (\ref{modified Bessel function kosh case}) can be written
as a series of Bessel functions (see Lemma 5.1. above). 

\begin{corollary}
If $\alpha$ and $\beta$ are positive numbers such that $\alpha\beta=\pi^{2}$
and if $s$ is any complex number, then the following identities take
place
\begin{align}
\sqrt{\alpha}\,\sum_{m,n=1}^{\infty}\left(\frac{m}{n}\right)^{\frac{s}{2}}\,K_{\frac{s}{2}}\left(2m\,n\alpha\right) & -\sqrt{\beta}\,\sum_{m,n=1}^{\infty}\left(\frac{m}{n}\right)^{\frac{s}{2}}\,K_{\frac{s}{2}}\left(2m\,n\beta\right)=\nonumber \\
\frac{1}{4}\Gamma\left(-\frac{s}{2}\right)\zeta(-s)\left\{ \beta^{(1+s)/2}-\alpha^{(1+s)/2}\right\}  & +\frac{1}{4}\Gamma\left(\frac{s}{2}\right)\zeta(s)\left\{ \beta^{(1-s)/2}-\alpha^{(1-s)/2}\right\} ,\label{classical Guinand p infinity}
\end{align}
\begin{align}
\sqrt{\alpha}\,\sum_{m,n=1}^{\infty}\frac{(-1)^{m}m^{s/2}}{(2n-1)^{s/2}}K_{\frac{s}{2}}\left(\left(2n-1\right)m\,\alpha\right)-\sqrt{\beta}\,\sum_{m,n=1}^{\infty}\frac{(-1)^{m}m^{s/2}}{(2n-1)^{s/2}}K_{\frac{s}{2}}\left(\left(2n-1\right)m\,\beta\right)\nonumber \\
=\frac{\Gamma\left(\frac{s}{2}\right)\left(2^{s/2}-2^{-s/2}\right)\zeta(s)}{4}\left\{ \beta^{\frac{1-s}{2}}-\alpha^{\frac{1-s}{2}}\right\} .\label{classical Guinand p zero}
\end{align}

Moreover, we have
\begin{equation}
\sqrt{\alpha}\left(\frac{1}{4}\gamma-\frac{1}{4}\log(4\beta)+\sum_{m,n=1}^{\infty}K_{0}(2mn\,\alpha)\right)=\sqrt{\beta}\left(\frac{1}{4}\gamma-\frac{1}{4}\log(4\alpha)+\sum_{m,n=1}^{\infty}K_{0}(2mn\,\beta)\right),\label{Classical Kosh}
\end{equation}
\begin{equation}
\sqrt{\alpha}\left(\sum_{m,n=1}^{\infty}(-1)^{m}\,K_{0}\left(\left(2n-1\right)m\,\alpha\right)-\frac{\log(2)}{4}\right)=\sqrt{\beta}\left(\sum_{m,n=1}^{\infty}(-1)^{m}K_{0}\left(\left(2n-1\right)m\,\beta\right)-\frac{\log(2)}{4}\right).\label{Modified Koshliakooov}
\end{equation}
\end{corollary}

\begin{proof}

We only prove the classical Ramanujan-Guinand formula (\ref{classical Guinand p infinity}) from our generalized version of Entry 3.3.1.
The proofs of the remaining formulas (\ref{classical Guinand p zero}), (\ref{Classical Kosh}) and (\ref{Modified Koshliakooov}) follow the same principle. The only part that
it is not immediate is the recovery of the Modified Bessel function appearing in (\ref{classical Guinand p infinity})
from the left-hand side of (\ref{generalized Guinand with new function}). However, this was done in Lemma 5.1. above. Still, assuming that $\text{Re}(s)<1$ and justifying once more all the intermediate steps by absolute convergence, we find that the first term on the
left of (\ref{generalized Guinand with new function}) is
\begin{align}
\frac{2^{-s}\sqrt{\pi}}{\Gamma\left(\frac{1-s}{2}\right)}\alpha^{\frac{1-s}{2}}\sum_{n=1}^{\infty}\frac{1}{n^{s}}\,\intop_{0}^{\infty}\,\frac{y^{-\frac{s+1}{2}}(y+1)^{-\frac{s+1}{2}}}{e^{2\alpha n(2y+1)}-1}dy & =\frac{2^{-s}\sqrt{\pi}}{\Gamma\left(\frac{1-s}{2}\right)}\alpha^{\frac{1-s}{2}}\sum_{m,n=1}^{\infty}\frac{1}{n^{s}}\intop_{0}^{\infty}\,e^{-2\alpha m\,n(2y+1)}y^{-\frac{s+1}{2}}(y+1)^{-\frac{s+1}{2}}dy\nonumber \\
=\sqrt{\alpha}\,\sum_{m,n=1}^{\infty}\left(\frac{m}{n}\right)^{\frac{s}{2}}\,K_{\frac{s}{2}}\left(2m\,n\alpha\right) & =\sqrt{\alpha}\,\sum_{n=1}^{\infty}\sum_{d|n}d^{-s}\,n^{\frac{s}{2}}\,K_{\frac{s}{2}}\left(2n\alpha\right)=\sqrt{\alpha}\,\sum_{n=1}^{\infty}\sigma_{-s}(n)\,n^{s/2}K_{\frac{s}{2}}\left(2n\alpha\right).\label{extreme right side}
\end{align}

Note now that the expression on the extreme right-hand side of (\ref{extreme right side})
defines a series that converges uniformly and absolutely for every
$s\in\mathbb{C}$. Since its summands are analytic functions of $s$,
this series defines an analytic function of $s$. By the principle
of analytic continuation, we now see that (\ref{classical Guinand p infinity})
must hold for every $s\in\mathbb{C}$.
\end{proof}

\bigskip{}
We now note that our proof of Theorem 5.1. can be actually used to
prove a more general result than (\ref{generalized Guinand with new function}).
Note that we can use the identity (\ref{protoform of Guinand}), take
the substitutions $c=\alpha^{2}/\pi^{2}$, $s\leftrightarrow\frac{1+s}{2}$
but not imposing that $p=p^{\prime}$. In this case we get the following
result.

\begin{theorem}
If $\alpha$ and $\beta$ are positive numbers such that $\alpha\beta=\pi^{2}$
and if $s$ is a complex number such that $\text{Re}(s)<1$, then
the following identity takes place
\begin{align}
\frac{\Gamma\left(-\frac{s}{2}\right)}{4}\left\{ \frac{\beta^{\frac{s+1}{2}}}{1+\frac{1}{\pi p}}\eta_{p^{\prime}}(-s)-\frac{\alpha^{\frac{s+1}{2}}}{1+\frac{1}{\pi p^{\prime}}}\eta_{p}(-s)\right\} +\frac{\Gamma\left(\frac{s}{2}\right)}{4}\left\{ \beta^{\frac{1-s}{2}}\zeta_{p}(s)-\alpha^{\frac{1-s}{2}}\zeta_{p^{\prime}}(s)\right\}  & =\nonumber \\
=\frac{2^{-s}\sqrt{\pi}}{\Gamma\left(\frac{1-s}{2}\right)}\,\alpha^{\frac{1-s}{2}}\sum_{n=1}^{\infty}\frac{p^{\prime2}+\lambda_{n}^{\prime2}}{p^{\prime}\left(p^{\prime}+\frac{1}{\pi}\right)+\lambda_{n}^{\prime2}}\,\lambda_{n}^{\prime-s}\,\mathcal{K}_{\frac{s}{2},p}\left(2\lambda_{n}^{\prime}\alpha\right)-\frac{2^{-s}\sqrt{\pi}}{\Gamma\left(\frac{1-s}{2}\right)}\,\beta^{\frac{1-s}{2}}\,\sum_{n=1}^{\infty}\frac{p^{2}+\lambda_{n}^{2}}{p\left(p+\frac{1}{\pi}\right)+\lambda_{n}^{2}}\,\lambda_{n}^{-s}\mathcal{K}_{\frac{s}{2},p^{\prime}}\left(2\lambda_{n}\beta\right),\label{More general Koshliakov final version}
\end{align}
where $\mathcal{K}_{\nu,p}(x)$ denotes the integral defined by (\ref{modified Bessel function kosh case}).
\end{theorem}

Using this more general version of the Ramanujan-Guinand formula,
we can establish a wide class of Corollaries. The following result
comes easily from taking $p^{\prime}\rightarrow\infty$ in (\ref{More general Koshliakov final version})
and appealing to Lemma 5.1.

\begin{corollary}
If $\alpha$ and $\beta$ are positive numbers such that $\alpha\beta=\pi^{2}$
and if $s$ is a complex number satisfying the condition $\text{Re}(s)<1$, then the following identity takes
place
\begin{align}
\frac{\Gamma\left(-\frac{s}{2}\right)}{4}\left\{ \frac{\beta^{\frac{s+1}{2}}}{1+\frac{1}{\pi p}}\zeta(-s)-\alpha^{\frac{s+1}{2}}\eta_{p}(-s)\right\} +\frac{\Gamma\left(\frac{s}{2}\right)}{4}\left\{ \beta^{\frac{1-s}{2}}\zeta_{p}(s)-\alpha^{\frac{1-s}{2}}\zeta(s)\right\}  & =\nonumber \\
=\frac{2^{-s}\sqrt{\pi}}{\Gamma\left(\frac{1-s}{2}\right)}\,\alpha^{\frac{1-s}{2}}\sum_{n=1}^{\infty}\,n^{-s}\,\mathcal{K}_{\frac{s}{2},p}\left(2n\alpha\right)-\sqrt{\beta}\,\sum_{m,n=1}^{\infty}\frac{p^{2}+\lambda_{n}^{2}}{p\left(p+\frac{1}{\pi}\right)+\lambda_{n}^{2}}\,\left(\frac{m}{\lambda_{n}}\right)^{\frac{s}{2}}\,K_{\frac{s}{2}}\left(2m\lambda_{n}\beta\right).\label{another corollllllary general guuuiiiinand}
\end{align}

In particular, (\ref{Guinand given at intro}) holds.  Moreover,
\begin{align*}
\frac{\Gamma\left(-\frac{s}{2}\right)\zeta(-s)}{4}\alpha^{\frac{s+1}{2}}(1-2^{1+s})+\frac{\Gamma\left(\frac{s}{2}\right)\zeta(s)}{4}\left\{ \beta^{\frac{1-s}{2}}(2^{s}-1)-\alpha^{\frac{1-s}{2}})\right\}  & =\\
=\sqrt{\alpha}\,\sum_{m,n=1}^{\infty}(-1)^{n}\,\left(\frac{n}{m}\right)^{s/2}\,K_{\frac{s}{2}}\left(2mn\,\alpha\right)-2^{s/2}\,\sqrt{\beta}\,\sum_{m,n=1}^{\infty}\left(\frac{m}{2n-1}\right)^{\frac{s}{2}}\,K_{\frac{s}{2}}\left((2n-1)\,m\beta\right),
\end{align*}
for every $s\in\mathbb{C}$.
\end{corollary}

Our next corollary is obtained when we take $p^{\prime}\rightarrow0^{+}$
in (\ref{More general Koshliakov final version}) and use Lemma 5.1.

\begin{corollary}
If $\alpha$ and $\beta$ are positive numbers such that $\alpha\beta=\pi^{2}$
and if $s$ is a complex number satisfying the condition $\text{Re}(s)<1$, then the following identity takes
place
\begin{align*}
\frac{\Gamma\left(-\frac{s}{2}\right)\zeta(-s)}{4}\,\frac{\beta^{\frac{s+1}{2}}}{1+\frac{1}{\pi p}}\left(2^{1+s}-1\right)+\frac{\Gamma\left(\frac{s}{2}\right)}{4}\left\{ \beta^{\frac{1-s}{2}}\zeta_{p}(s)-\alpha^{\frac{1-s}{2}}\left(2^{s}-1\right)\zeta(s)\right\}  & =\\
=\frac{\sqrt{\pi}}{\Gamma\left(\frac{1-s}{2}\right)}\,\alpha^{\frac{1-s}{2}}\sum_{n=1}^{\infty}\frac{\mathcal{K}_{\frac{s}{2},p}\left((2n-1)\alpha\right)}{\left(2n-1\right)^{s}}-\sqrt{\beta}\,\sum_{m,n=1}^{\infty}\frac{(p^{2}+\lambda_{n}^{2})\,(-1)^{m}}{p\left(p+\frac{1}{\pi}\right)+\lambda_{n}^{2}}\left(\frac{m}{\lambda_{n}}\right)^{s/2}\,K_{\frac{s}{2}}(2m\,\lambda_{n}\beta).
\end{align*}

In particular, (\ref{classical Guinand p zero}) holds. 
\end{corollary}

\begin{remark}
Another proof of Corollary 5.3. can be obtained if we reverse the roles
of $p$ and $p^{\prime}$ and use the functional equation (\ref{Functional equation particular case}).
Writing each side of (\ref{Functional equation particular case})
by using their respective analogue of the Selberg-Chowla formula,
(\ref{representation Re(s)<1}) or (\ref{Selberg-Chowla for second Epstein}),
we can easily derive (\ref{another corollllllary general guuuiiiinand}). 
\end{remark}

\bigskip{}

Finally, we derive a generalized transformation formula for the logarithm
of the Dedekind $\eta-$function. We remark that a particular case
of our formula (when $p=p^{\prime}$) was derived by Dixit and Gupta
{[}\cite{DG}, Theorem 4.5., p. 15{]}. Our proof is drastically different because we use the first analogue of Kronecker's limit formula (\ref{corrected Kronecker Limit formula}) given in the previous section. 

\begin{corollary}

Let $p,p^{\prime}\in\mathbb{R}_{+}$ and $\sigma(t)$, $\sigma^{\prime}(t)$
be defined by (\ref{definition sigma p sigma p'}). Then, if $\alpha,\beta>0$
are such that $\alpha\beta=\pi^{2}$, the following generalization of Entry 3.3.2. takes place 
\begin{align}
\,\sum_{n=1}^{\infty}\frac{p^{\prime2}+\lambda_{n}^{\prime2}}{p^{\prime}\left(p^{\prime}+\frac{1}{\pi}\right)+\lambda_{n}^{\prime2}}\cdot\frac{\lambda_{n}^{\prime-1}}{\sigma\left(\frac{\lambda_{n}^{\prime}\alpha}{\pi}\right)e^{2\alpha\lambda_{n}^{\prime}}-1}-\sum_{n=1}^{\infty}\frac{p^{2}+\lambda_{n}^{2}}{p\left(p+\frac{1}{\pi}\right)+\lambda_{n}^{2}}\cdot\frac{\lambda_{n}^{-1}}{\sigma^{\prime}\left(\frac{\lambda_{n}\beta}{\pi}\right)e^{2\beta\lambda_{n}}-1}\nonumber \\
=\frac{1}{12\left(1+\frac{1}{\pi p}\right)\left(1+\frac{1}{\pi p^{\prime}}\right)}\left\{ \beta\,\frac{1+\frac{3}{\pi p^{\prime}}(1+\frac{1}{\pi p^{\prime}})}{1+\frac{1}{\pi p^{\prime}}}-\alpha\,\frac{1+\frac{3}{\pi p}(1+\frac{1}{\pi p})}{1+\frac{1}{\pi p}}\right\} +\frac{C_{p}^{(1)}-C_{p^{\prime}}^{(1)}}{2}+\frac{1}{4}\,\log\left(\frac{\alpha}{\beta}\right).\label{Dedekind generalized final version}
\end{align}

\end{corollary}

\begin{proof}
We have seen in the proof of Theorem 5.1. that the relation holds
\begin{equation}
\zeta_{p,p^{\prime}}(s,\,c)=c^{-s}\,\zeta_{p^{\prime},p}\left(s,\,\frac{1}{c}\right).\label{relation valid for all s}
\end{equation}

Employing Kronecker's limit formula to $\zeta_{p^{\prime},p}\left(s,\,\frac{1}{c}\right)$
(i.e., using (\ref{corrected Kronecker Limit formula}) with $c$
replaced by $c^{-1}$ and $p$ by $p^{\prime}$), we deduce that $c^{-s}\zeta_{p^{\prime},p}\left(s,\frac{1}{c}\right)$
must have the meromorphic expansion
\begin{align}
c^{-s}\zeta_{p,p^{\prime}}\left(s,\frac{1}{c}\right) & =\frac{\pi}{\sqrt{c}}\,\frac{1}{s-1}-\frac{\pi\log(c)}{\sqrt{c}}+\frac{\pi^{2}}{3c}\,\frac{1+\frac{3}{\pi p^{\prime}}(1+\frac{1}{\pi p^{\prime}})}{\left(1+\frac{1}{\pi p}\right)\left(1+\frac{1}{\pi p^{\prime}}\right)^{2}}+\nonumber \\
+\pi\sqrt{c}\, & \left(2C_{p}^{(1)}-\log\left(\frac{4}{c}\right)+4\,\sum_{n=1}^{\infty}\frac{p^{2}+\lambda_{n}^{2}}{p\left(p+\frac{1}{\pi}\right)+\lambda_{n}^{2}}\cdot\frac{\lambda_{n}^{-1}}{\sigma^{\prime}\left(\sqrt{\frac{1}{c}}\lambda_{n}\right)e^{2\pi\sqrt{\frac{1}{c}}\lambda_{n}}-1}\right)+O\left(s-1\right)\label{second meromorphic Kronecker}
\end{align}

Since relation (\ref{relation valid for all s}) holds for every $s\in\mathbb{C}$
by analytic continuation, the constant terms of the meromorphic expansions
(\ref{second meromorphic Kronecker}) and (\ref{corrected Kronecker Limit formula})
must be the same. Henceforth, equating the constant terms of (\ref{corrected Kronecker Limit formula})
and (\ref{second meromorphic Kronecker}) and replacing $c=\alpha^{2}/\pi^{2}$ gives (\ref{Dedekind generalized final version}). 
\end{proof}

Our next corollary presents a particular case of the previous result when we take $p\rightarrow0^{+}$
and $p^{\prime}\rightarrow\infty$. It is not explicitly given in \cite{DG}, so we shall state it here. 
\begin{corollary}
Let $\alpha,\beta>0$
be such that $\alpha\beta=\pi^{2}$. Then the following identity takes place   
\[
\sum_{n=1}^{\infty}\frac{1}{n\left(e^{2\alpha n}+1\right)}+2\,\sum_{n=1}^{\infty}\frac{1}{(2n-1)}\,\frac{1}{\left(e^{\beta(2n-1)}-1\right)}=\frac{\alpha}{4}-\log(2)-\frac{1}{4}\log\left(\frac{\alpha}{\beta}\right).
\]
\end{corollary}

\bigskip{}

Finally, using the same method as in the previous Corollary 5.5. (proof of a generalized Entry 3.3.2.), we state
the final result of this paper, which is also an extension of a formula
found in one of Ramanujan's notebooks [\cite{Ramanujan_lost}, p. 318],
although discovered earlier by Schl\"omlich (see [\cite{berndt_notebooks_II}, p. 256] for a historical account of this particular formula). We will omit its proof,
because it consists in using (\ref{Dedekind generalized final version})
with $p=p^{\prime}$, dividing by $\beta-\alpha$ and then let $\alpha\rightarrow\beta$.
\begin{corollary}
Let $p\in\mathbb{R}_{+}$ and let $\sigma(t)$ be defined by (\ref{definition sigma p sigma p'}). Then
the following identity holds 
\[
\sum_{n=1}^{\infty}\frac{p^{2}+\lambda_{n}^{2}}{p\left(p+\frac{1}{\pi}\right)+\lambda_{n}^{2}}\,\frac{e^{2\pi\lambda_{n}}}{\left(\sigma(\lambda_{n})e^{2\pi\lambda_{n}}-1\right)^{2}}\cdot\left\{ \pi\,\sigma\left(\lambda_{n}\right)+\frac{p}{\left(p-\lambda_{n}\right)^{2}}\right\} =\frac{\pi}{24}\cdot\frac{1+\frac{3}{\pi p}(1+\frac{1}{\pi p})}{\left(1+\frac{1}{\pi p}\right)^{3}}-\frac{1}{8}.
\]
\end{corollary}
Taking $p\rightarrow\infty$ gives
Schl{\"o}mlich's formula. By taking the limit $p\rightarrow0^{+}$
we obtain a companion of it, which seems to be new. We state them in the following Corollary.

\begin{corollary}

The following identities hold: 
\begin{equation}
\sum_{n=1}^{\infty}\frac{e^{2\pi n}}{\left(e^{2\pi n}-1\right)^{2}}=\frac{1}{24}-\frac{1}{8\pi},\,\,\,\,\,\,\,\,\,\,\,\,\sum_{n=1}^{\infty}\frac{e^{\pi(2n-1)}}{\left(e^{\pi(2n-1)}+1\right)^{2}}=\frac{1}{8\pi}.\label{Schlomlich and companion}
\end{equation}
\end{corollary}

\bigskip{}
\bigskip{}
\bigskip{}
\bigskip{}
\textit{Disclosure Statement:} The authors declare that they have no conflict of interest. 

\bigskip{}

\textit{Acknowledgements:} The authors would like to thank the editors for the invitation for a contribution to the memorial publication in honour of our late friend, Professor Jos\'e Carlos Petronilho. 

The first author would also like to thank to Professor Atul Dixit and to Rajat Gupta for their outstanding and engaging talks at the conference OPSFA16, which motivated the search for Epstein zeta functions attached to the Koshliakov sequence. This was the first instance when the first author learned about Koshliakov's beautiful generalization of the Riemann zeta function. 

Both authors were partially supported by CMUP, member of LASI, which is financed by national funds through FCT - Fundação para a Ciência e a Tecnologia, I.P., under the projects with reference UIDB/00144/2020 and UIDP/00144/2020.

The first author also acknowledges the support from FCT (Portugal) through the PhD scholarship 2020.07359.BD.  
\newpage{}

\end{document}